\pdfoutput=1
\documentclass[12pt]{amsart}
\usepackage{txfonts}
\usepackage{amscd,amssymb,mathabx}
\usepackage{graphicx,calrsfs}
\usepackage{subcaption}
\usepackage[colorlinks,plainpages,backref,urlcolor=blue]{hyperref}
\usepackage{tikz}
\usetikzlibrary{cd, calc, arrows}
\usepackage{mdwlist}
\usepackage{enumitem} 
\usepackage{etoolbox}
\usepackage{spectralsequences}

\numberwithin{equation}{section}

\newtheorem{theorem}{Theorem}[section]
\newtheorem{corollary}[theorem]{Corollary}
\newtheorem{lemma}[theorem]{Lemma}
\newtheorem{prop}[theorem]{Proposition}

\theoremstyle{definition}

\newtheorem{example}[theorem]{Example}
\newtheorem{remark}[theorem]{Remark}

\newtheorem*{ack}{Acknowledgments}

\AtBeginEnvironment{example}{%
  \pushQED{\qed}%
}
\AtEndEnvironment{example}{\popQED\endexample}

\newcommand{\N}{\mathbb{N}}
\newcommand{\Z}{\mathbb{Z}}
\newcommand{\Q}{\mathbb{Q}}
\newcommand{\R}{\mathbb{R}}
\newcommand{\C}{\mathbb{C}}

\newcommand{\T}{\mathbb{T}}

\renewcommand{\P}{\mathbb{P}}
\newcommand{\CP}{\mathbb{CP}}

\renewcommand{\k}{\Bbbk}

\newcommand{\RR}{\mathcal{R}}
\newcommand{\V}{\mathcal{V}}
\newcommand{\VV}{\mathcal{V}}
\newcommand{\A}{{\mathcal{A}}}
\newcommand{\B}{{\mathcal{B}}}

\newcommand{\WW}{\mathcal{W}}
\newcommand{\XX}{\mathcal{X}}

\newcommand{\NN}{\mathcal{N}}

\newcommand{\h}{{\mathfrak{h}}}

\newcommand{\bo}{{\mathbf 1}}
\newcommand{\bz}{{\mathbf 0}}
\newcommand{\bm}{{\mathbf{m}}}

\DeclareMathOperator{\rank}{rank}
\DeclareMathOperator{\corank}{corank}
\DeclareMathOperator{\gr}{gr}
\DeclareMathOperator{\im}{im}
\DeclareMathOperator{\coker}{coker}
\DeclareMathOperator{\codim}{codim}

\DeclareMathOperator{\id}{id}
\DeclareMathOperator{\ab}{{ab}}
\DeclareMathOperator{\abf}{{abf}}

\DeclareMathOperator{\ch}{char}

\DeclareMathOperator{\Hom}{{Hom}}
\DeclareMathOperator{\Tor}{{Tor}}
\DeclareMathOperator{\Ext}{{Ext}}

\DeclareMathOperator{\ann}{{ann}}

\DeclareMathOperator{\pr}{pr}
\DeclareMathOperator{\Aut}{Aut}

\DeclareMathOperator{\depth}{depth}

\DeclareMathOperator{\Tors}{Tors}
\DeclareMathOperator{\TC}{TC}
\DeclareMathOperator{\Der}{Der}
\DeclareMathOperator{\orb}{orb}

\DeclareMathOperator{\ii}{i}

\newcommand{\surj}{\twoheadrightarrow}
\newcommand{\inj}{\hookrightarrow}
\newcommand{\isom}{\xrightarrow{\,\simeq\,}}

\newcommand{\bwedge}{\mbox{\normalsize $\bigwedge$}}

\newcommand{\abs}[1]{\left| #1 \right|}

\def\set#1{{\{ #1\}}}

\newcommand{\mysetminus}{\,\setminus\,}

\definecolor{dkgreen}{RGB}{0,100,0}
\definecolor{dkbrown}{RGB}{139,69,19}

\newcommand{\pcoor}[1]{%
  \begingroup\lccode`~=`: \lowercase{\endgroup
  \edef~}{\mathbin{\mathchar\the\mathcode`:}\nobreak}%
  [
  \begingroup
  \mathcode`:=\string"8000
  #1%
  \endgroup 
  ]
}

\definecolor{lime}{HTML}{A6CE39}
\DeclareRobustCommand{\orcidicon}{
	\begin{tikzpicture}
	\draw[lime, fill=lime] (0,0) 
	circle [radius=0.16] 
	node[white] {{\fontfamily{qag}\selectfont \tiny ID}};
	\draw[white, fill=white] (-0.0625,0.095) 
	circle [radius=0.007];
	\end{tikzpicture}
	\hspace{-2mm}
}
\foreach \x in {A, ..., Z}{\expandafter\xdef\csname orcid\x\endcsname{\noexpand\href{%
https://orcid.org/\csname orcidauthor\x\endcsname}{\noexpand\orcidicon}}}

\usepackage{xpatch}
\makeatletter   
\xpatchcmd{\@tocline}
{\hfil\hbox to\@pnumwidth{\@tocpagenum{#7}}\par}
{\ifnum#1<0\hfill\else\dotfill\fi\hbox to\@pnumwidth{\@tocpagenum{#7}}\par}
{}{}
\makeatother     

\makeatletter
 \def\l@subsection{\@tocline{2}{0pt}{4pc}{6pc}{}}
\def\l@subsubsection{\@tocline{3}{0pt}{8pc}{8pc}{}}
 \makeatother

\begin{document}
\date{May 4, 2024}

\title[Milnor fibrations of arrangements with trivial algebraic monodromy]{%
Milnor fibrations of arrangements with trivial\\ algebraic monodromy} 

\author[Alexandru~I.~Suciu]{Alexandru~I.~Suciu$^1$\!\!\orcidA{}}
\address{Department of Mathematics,
Northeastern University,
Boston, MA 02115, USA}
\email{\href{mailto:a.suciu@northeastern.edu}{a.suciu@northeastern.edu}}
\urladdr{\href{https://suciu.sites.northeastern.edu}%
{https://suciu.sites.northeastern.edu}}
\thanks{$^1$Supported in part by Simons Foundation Collaboration 
Grants for Mathematicians \#693825}



\begin{abstract}
Each complex hyperplane arrangement gives rise to a 
Milnor fibration of its complement. Although the Betti numbers of 
the Milnor fiber $F$ can be expressed in terms of the jump loci for 
rank $1$ local systems on the complement, explicit formulas are still lacking 
in full generality, even for $b_1(F)$. We study here 
the ``generic" case (in which $b_1(F)$ is as small as possible), and look 
deeper into the algebraic topology of such Milnor fibrations with trivial 
algebraic monodromy. Our main focus is on the cohomology jump 
loci and the lower central series quotients of $\pi_1(F)$. 
In the process, we produce a pair of arrangements for which 
the respective Milnor fibers have the same Betti numbers, 
yet non-isomorphic fundamental groups: the difference is picked 
by the higher-depth characteristic varieties and by the 
Schur multipliers of the second nilpotent quotients.
\end{abstract}
\maketitle

\tableofcontents

\section{Introduction}
\label{sect:intro}

\subsection{The Milnor fibration}
\label{subsec:intro1}
In a seminal book \cite{Mi}, Milnor introduced a fibration which soon became the 
central object of study in singularity theory. In its simplest form, 
the construction associates to a homogeneous polynomial 
$f\in \C[z_0,\dots , z_d]$ a smooth fibration over $\C^*$, 
defined by restricting the map $f\colon \C^{d+1} \to \C$
to the complement of its zero-set. 
The Milnor fiber, $F=f^{-1}(1)$, is a smooth complex affine variety 
of complex dimension $d$. The monodromy of the fibration, 
$h\colon F\to F$, is given by $h(z)=e^{2\pi \ii/n} z$, where $n=\deg f$. 
A key question is to compute the characteristic polynomials of the 
induced homomorphisms in homology, $h_q\colon H_q(F;\C)\to H_q(F;\C)$. 

We are mainly interested in the case when $f$ has singularities in 
codimension $1$. Arguably the simplest situation in this regard is 
when the polynomial $f$ completely factors into distinct linear forms. 
This situation is neatly described by a hyperplane arrangement, that is, 
a finite collection $\A$ of codi\-mension-$1$ linear subspaces in $\C^{d+1}$. 
Choosing a linear form $f_H$ with kernel $H$ for each hyperplane $H\in \A$, 
we obtain a homogeneous polynomial, 
$f=\prod_{H\in \A} f_H$, which in turn defines the Milnor fibration 
of the complement of the arrangement, $M=M(\A)$, with fiber $F=F(\A)$.   
More generally, if $\bm\colon \A\to \N$, $H\mapsto m_H$ is a choice of 
multiplicities for the hyperplanes comprising $\A$, we may consider the 
polynomial $f_{\bm}=\prod_{H\in \A} f_H^{m_H}$ and the corresponding 
Milnor fibration, with fiber $F_{\bm}$. 

To analyze these fibrations, it is most natural to use the rich 
combinatorial structure encoded in the intersection lattice of $\A$, 
that is, the poset of all intersections of hyperplanes in $\A$, 
ordered by reverse inclusion and ranked by codimension. 
A much-studied question in the subject asks: 
Is the characteristic polynomial of the algebraic monodromy of the 
(usual) Milnor fibration, $\Delta_{\A, q}(t) = \det(tI-h_q)$, determined by 
the intersection sub-lattice $L_{\le q+1}(\A)$? Despite much effort---and 
some progress---over the past 30--40 years, the problem is still open, 
even in degree $q=1$.

In this paper, we take a different tack, and focus instead on the  
``generic" situation, to wit, on those 
hyperplane arrangements for which the monodromy of the Milnor 
fibration acts trivially on the homology of the Milnor fiber, 
either with $\Z$ or with $\C$ coefficients. 

\subsection{Cohomology jump loci}
\label{subsec:intro:cjl}
We start by analyzing the structure of the {\em characteristic varieties} 
(the jump loci for homology in rank $1$ local systems) and the 
{\em resonance varieties} (the jump loci of the Koszul complex 
associated to the cohomology algebra) of the Milnor fiber 
of a multi-arrangement in the trivial algebraic monodromy setting.

Let $U=\P(M)$ be the projectivization of the complement $M=M(\A)$.
Since $U$ is a smooth, connected, quasi-projective variety, its characteristic 
varieties, $\V^q_s(U)$, are finite unions of torsion-translates of algebraic subtori 
of the character group, $\Hom(\pi_1(U),\C^*)=H^1(U;\C^*)$, see~\cite{Ar, BW}.  
Since $U$ is also a formal space, its resonance varieties, $\RR^q_s(U)$,  
coincide with the tangent cone at the trivial character to $\V^q_s(U)$, 
see \cite{CS99, DPS-duke, DP-ccm}.  
As shown in \cite{FY}, the varieties $\RR^1_s(U)$ 
may be described solely in terms of multinets on sub-arrangements 
of $\A$.  In general, though, the varieties $\V^1_s(U)$
may contain components which do not pass through the origin, see 
\cite{Su02, CDO03, DeS-plms}. We explain in detail the relationship 
between the cohomology jump loci of $M$ and $U$ in 
Proposition \ref{prop:cjl-MU} and Corollary \ref{cor:cjl-MU-bis}.

Now  let $(\A,\bm)$ be a multi-arrangement in $\C^{d+1}$ and let 
$F_{\bm}\to M\to\C^*$ be the Milnor fibration of the complement, 
with monodromy $h\colon F_{\bm}\to F_{\bm}$, 
We then have a regular $\Z_N$-cover, $\sigma_{\bm}\colon F_{\bm}\to U$, 
where $N=\sum_{H\in \A} m_H$.
In Theorem \ref{thm:cjl-mf-trivialmono}, we prove the following result, 
which relates the degree $1$ cohomology jump loci of $F_{\bm}$ to those of 
$U=\P(M)$, under a trivial algebraic monodromy assumption.

\begin{theorem}
\label{intro:cjl-mf-trivialmono}
Suppose the map $h\colon F_{\bm}\to F_{\bm}$ 
induces the identity on $H_1(F_{\bm};\Q)$. Then, 
\begin{enumerate}[itemsep=2pt, topsep=-1pt]
\item  \label{intro-res-f-bis}
The induced homomorphism $\sigma^*_{\bm}\colon H^1(U;\C)\to H^1(F_{\bm};\C)$ 
is an isomorphism that identifies $\RR^1_s(U)$ with $\RR^1_s(F_{\bm})$, for all $s\ge 1$.
\item \label{intro-cv-f-bis}
The induced homomorphism  $\sigma^*_{\bm}\colon H^1(U;\C^*)\to H^1(F_{\bm};\C^*)^0$ 
is a surjection with kernel isomorphic to $\Z_N$. Moreover, 
\begin{enumerate}
\item \label{intro-pa}
For each $s\ge 1$, the map $\sigma^*_{\bm}$ establishes a bijection 
between the sets of irreducible components of $\VV^1_s(U)$ and 
$\WW^1_s(F_{\bm})$ that pass through the identity. 
\item \label{intro-pb}
The map 
$\sigma^*_{\bm} \colon \VV^1_1(U)\to \WW^1_1(F_{\bm})$ is a surjection.
\end{enumerate}
\end{enumerate}
\end{theorem}

Here, $H^1(F_{\bm};\C^*)^0$ denotes the identity component of the 
character group $H^1(F_{\bm};\C^*)$, while $\WW^1_s(F_{\bm})$ denotes 
its intersection with $\VV^1_s(F_{\bm})$. The theorem builds on and 
sharpens results of Dimca and Papadima from \cite{DP-pisa}.

\subsection{Abelian duality and propagation}
\label{subsec:intro dual}
It has long been recognized that complements of complex hyperplane 
arrangements satisfy certain vanishing properties for homology with 
coefficients in local systems. In \cite{DSY16, DSY17}, we revisited this 
subject, in a more general framework.  

Given a connected, finite-type CW-complex $X$ with fundamental group $G$, 
we say that $X$ is an {\em $\ab$-duality space}\/ of dimension $m$ if $H^q(X;\Z{G}^{\ab})=0$ 
for $q\ne m$ and $H^{m}(X;\Z{G}^{\ab})$ is non-zero and torsion-free. Replacing the 
abelianization of $G$ by the torsion-free abelianization, $G_{\abf}=G_{\ab}/\Tors$,  
we obtain the analogous notion of {\em $\abf$-duality space}\/ (of dimension $m$). 
These properties imposes stringent conditions on the cohomological invariants of the 
space $X$. Most notably, as shown in \cite{DSY17}, if $X$ is an $\ab$-duality space 
of dimension $n$, then the characteristic varieties of $X$ {\em propagate}, 
that is, $\{\bo\}=\V^0_1(X)\subseteq \V^1_1(X)\subseteq\cdots\subseteq\V^{m}_1(X)$. 

It was shown in \cite{DSY16, DSY17} that complements of hyperplane arrangements 
are $\ab$-duality spaces; see also \cite{DeS-sigma, LMW} for generalizations of this result. 
Moreover, it was shown in \cite{DSY17} that the $\ab$-duality property behaves well under a 
certain type of ``$\ab$-exact" fibrations. Making use of these results, together with their 
adaptations in the $\abf$-duality/$\abf$-exact context, we establish in 
Theorem \ref{thm:mf-abel-duality} and Corollary \ref{cor:mf-propagate} 
the following:

\begin{theorem}
\label{intro:mf-abel-duality}
Let $\A$ be a central arrangement of rank $r$ and let 
$F_{\bm}=F_{\bm}(\A)$ be the Milnor fiber associated to a multiplicity 
vector $\bm\colon \A\to \N$. 
\begin{enumerate}[itemsep=1pt]
\item \label{mfab1}
Suppose the monodromy action on $H_1(F_{\bm};\Z)$ is trivial. Then,
\begin{enumerate}
\item  \label{mfab2}
$F_{\bm}$ is an $\ab$-duality space of dimension $r-1$. 
\item  \label{mfab3}
The characteristic varieties of $F_{\bm}$ propagate; that is,
\[
\V^1_1(F_{\bm})\subseteq \V^2_1(F_{\bm})
\subseteq\cdots\subseteq\V^{r-1}_1(F_{\bm}).
\]
\end{enumerate}
\item  \label{mfab4}
Suppose the monodromy action on $H_1(F_{\bm};\Q)$ is trivial. Then.
\begin{enumerate}
\item   \label{mfab5}
$F_{\bm}$ is an $\abf$-duality space of dimension $r-1$. 
\item  \label{mfab6}
The restricted characteristic varieties of $F_{\bm}$ propagate; that is,
\[
\WW^1_1(F_{\bm})\subseteq \WW^2_1(F_{\bm})\subseteq
\cdots\subseteq\WW^{r-1}_1(F_{\bm}).
\]
\end{enumerate}
\end{enumerate}
\end{theorem}

This result strengthens \cite[Thm.~6.7]{DSY17}, 
where only part \eqref{mfab1} is proved (in the particular case when 
$F=F(\A)$ is the usual Milnor fiber of an essential arrangement), 
but not part \eqref{mfab4}. 
We also show: If the monodromy action on $H_i(F_{\bm};\Q)$ is trivial 
for $i\le q$, then the resonance varieties of $F_{\bm}$ propagate 
in that range; that is,  $\RR^1_1(F_{\bm})\subseteq\cdots\subseteq \RR^{q}_1(F_{\bm})$.

\subsection{Associated graded Lie algebras}
\label{subsec:intro lcs}

The lower central series (LCS) of a group $G$ is defined inductively 
by setting $\gamma_1 (G)=G$ and $\gamma_{k+1}(G) =[G,\gamma_k (G)]$ 
for $k\ge 1$. This series is both normal and central; therefore, 
its successive quotients, $\gr_k(G)= \gamma_k(G)/\gamma_{k+1}(G)$, 
are abelian groups. The first such quotient 
coincides with the abelianization, $G_{\ab}=H_1(G;\Z)$.  
The associated graded Lie algebra of the group, $\gr(G)$, is the direct 
sum of the groups $\gr_k(G)$, with Lie bracket (compatible 
with the grading) induced from the group commutator. Important 
in this context is also the Chen Lie algebra of $G$, that is, the 
associated graded Lie algebra $\gr(G/G'')$ of the maximal 
metabelian quotient of $G$. 

When the group $G$ is finitely generated, the LCS quotients of $G$ are 
also finitely generated. We let $\phi_k(G)\coloneqq \rank(\gr_k(G))$ be the 
ranks of these abelian groups and we let $\theta_k(G)\coloneqq \rank(\gr_k(G/G''))$ 
be the Chen ranks of $G$. Quite a bit is known about the LCS 
ranks and the Chen ranks of arrangement groups, though almost nothing 
is known about the corresponding ranks for the Milnor fiber groups. 
As a first step in this direction, we show that the former determine 
the latter when the algebraic monodromy is trivial. More precisely, 
we prove in Theorems \ref{thm:trivial-mono-z} an \ref{thm:trivial-mono-q} 
the following statements.

\begin{theorem}
\label{thm:intro lcs}
Let $(\A,\bm)$ be a multi-arrangement and let 
$F_{\bm}$ be the corresponding Milnor fiber, 
with monodromy $h\colon  F_{\bm} \to F_{\bm}$.
Set $G=\pi_1(M(\A))$ and $K=\pi_1(F_{\bm})$. 
\begin{enumerate}[itemsep=2.5pt]
\item
If $h_*\colon H_1(F_{\bm};\Z)\to H_1(F_{\bm};\Z)$ 
is the identity map, then 
$\gr_{\ge 2}(K) \cong \gr_{\ge 2}(G)$ and 
$\gr_{\ge 2}(K/K'')  \cong  \gr_{\ge 2}(G/G'')$, 
as graded Lie algebras.
\item
If $h_*\colon H_1(F_{\bm};\Q)\to H_1(F_{\bm};\Q)$ is the identity map,  
then 
$\gr_{\ge 2}(K)\otimes \Q \cong \gr_{\ge 2}(G) \otimes \Q$ and 
$\gr_{\ge 2}(K/K'') \otimes \Q  \cong  \gr_{\ge 2}(G/G'') \otimes \Q$, 
as graded Lie algebras. 
\end{enumerate}
In either case, $\phi_k(K)=\phi_k(G)$ and $\theta_k(K)=\theta_k(G)$ 
for all $k\ge 2$.
\end{theorem}

Consequently, if the algebraic monodromy is trivial, both the LCS ranks 
and the Chen ranks of $\pi_1(F_{\bm})$ are combinatorially determined. 

\subsection{Constructions and examples}
\label{subsec:intro ex}
In Section \ref{sect:constructions}, we describe several classes of 
hyperplane arrangements for which the Milnor fibration has trivial 
algebraic monodromy. The simplest are the Boolean arrangements, 
followed by the generic arrangements. In both cases, complete answers 
regarding the homology of the Milnor fiber are known. We review these 
classical topics, in the more general context of arrangements 
with multiplicities.

Next, we consider the class of decomposable arrangements. 
Following \cite{PS-imrn04}, we say that an arrangement $\A$ is 
{\em decomposable}\/ (over $\Q$) if there are no elements in 
$\gr_3(\pi_1(M(\A))\otimes \Q$ besides those coming from the 
rank $2$ flats; that is, if $\phi_3(\pi_1(M(\A))= 
\sum_{X\in L_2(\A)} \binom{\mu(X)}{2}$, 
where $\mu\colon L(\A) \to \Z$ is the M\"{o}bius function.
As shown in \cite{Su-decomp}, for any choice of multiplicities $\bm$ 
on such an arrangement, the algebraic monodromy of the Milnor fibration, 
$h_*\colon H_1(F_{\bm};\Q)\to H_1(F_{\bm};\Q)$, is trivial, provided a 
certain technical condition is satisfied. Other classes 
of arrangements for which this conclusion holds are those for which certain 
multiplicities conditions are satisfied (see \cite{CDO03,Li02,Wi,WS,LX}), or 
the associated double point graph is connected and satisfies some additional 
requirements (see \cite{Bt14, SS17, Ve}). 

In \cite{Fa93}, Falk constructed a pair of rank-$3$ arrangements 
that have non-isomorphic intersection lattices, yet whose 
complements are homotopy equivalent. In Section \ref{sect:mf-falk}, we 
analyze in detail the Milnor fibrations of these arrangements. In both cases, 
the monodromy acts as the identity on first integral homology of the Milnor fiber. 
Nevertheless, the respective Milnor fibers are not homotopy equivalent. 
The difference is picked by both the degree-$1$, depth-$2$ characteristic 
varieties, and by the Schur multipliers of the second nilpotent quotients 
of their fundamental groups.

As shown in \cite{Su02}, deleting a suitable hyperplane from the 
$\operatorname{B}_3$ reflection arrangement yields an arrangement 
$\A$ of $8$ hyperplanes for which the variety $\V^1_1(M(\A))$ 
has an irreducible component 
(a subtorus translated by a character of order $2$) 
that does not pass through the identity of the character torus. 
As a consequence, there is a choice of multiplicities $\bm$ on $\A$ 
such that the monodromy $h\colon F_{\bm}\to F_{\bm}$ acts trivially 
on $H_1(F_{\bm};\Q)=\Q^7$ but not on $H_1(F_{\bm};\Z)=\Z^7\oplus \Z_2^2$, 
see \cite{CDS, DeS-plms}. We illustrate our techniques in Section \ref{sect:delB3} 
with a computation of the degree-$1$ characteristic varieties of $F_{\bm}$ and 
the low-degree LCS quotients and Chen groups of $\pi_1(F_{\bm})$.
Using a different approach, Yoshinaga constructed in \cite{Yo20} an arrangement 
$\A$ of $16$ hyperplanes such that the usual Milnor fiber itself, $F=F(\A)$, has non-trivial 
$2$-torsion. We summarize in Section \ref{sect:mf-yoshi} the information our techniques 
yield in this case regarding the LCS quotients and the Chen groups of $\pi_1(F)$.

\subsection{Organization of the paper}
\label{subsec:organize}
Roughly speaking, the paper is divided into three parts. 
The first one deals with some basic notions regarding hyperplane arrangements.  
In \S\ref{sect:hyper-comp} we discuss the combinatorics of an arrangement $\A$,  
as it relates to the topology of the complement $M(\A)$, while in 
\S\ref{sect:cjl-arr} we review the resonance and characteristic 
varieties of $\A$.

The second part covers the Milnor fibration of  
a multi-arrangement $(\A,\bm)$. In \S\ref{sect:mf} we 
discuss the homology of the Milnor fiber  $F_{\bm}$ and the 
monodromy action in homology. Under the assumption that this 
action is trivial, we investigate several topological invariants of 
the Milnor fiber: the cohomology jump loci in 
\S\ref{sect:cjl-milnor}, abelian duality and propagation 
of cohomology jump loci in \S\ref{sect:abel-prop}, and the 
lower central series of $\pi_1(F_{\bm})$~in~\S\ref{sect:lcs}.

The third part starts with \S\ref{sect:constructions}, where we 
describe ways to construct arrangements with trivial algebraic 
monodromy. The techniques developed in this work are illustrated 
with several examples worked out in detail: the pair of Falk 
arrangements in \S\ref{sect:mf-falk}, the deleted $\operatorname{B}_3$ 
arrangement in \S\ref{sect:delB3}, and Yoshinaga's icosidodecahedral 
arrangement in \S\ref{sect:mf-yoshi}.

\section{Complements of hyperplane arrangements}
\label{sect:hyper-comp}

\subsection{Hyperplane arrangements}
\label{subsec:hyp arr}

An {\em arrangement of hyperplanes}\/ is a finite set $\A$ of 
codimension-$1$ linear subspaces in a finite-dimensional 
complex vector space $\C^{d+1}$.  The combinatorics of the 
arrangement is encoded in its {\em intersection lattice}, $L(\A)$,   
that is, the poset of all intersections of hyperplanes in $\A$ (also 
known as flats), ordered by reverse inclusion, and ranked by 
codimension. 

Without much loss of generality, we will assume throughout 
that the arrangement is {\em central}, that is, all the hyperplanes 
pass through the origin.  For each hyperplane $H\in \A$, let 
$f_H\colon \C^{d+1} \to \C$ be a linear form with kernel $H$. 
The product $f=\prod_{H\in \A} f_H$, then, is a defining 
polynomial for the arrangement, unique up to a non-zero 
constant factor. Notice that $f$ is a homogeneous polynomial 
of degree equal to $n=\abs{\A}$, the number of hyperplanes 
comprising $\A$. 

The complement of the arrangement, 
$M(\A)=\C^{d+1}\mysetminus \bigcup_{H\in\A}H$, 
is a connected, smooth, complex quasi-projective variety.  Moreover,  
$M=M(\A)$ is a Stein manifold, and thus it has the homotopy type of a 
CW-complex $K$ of dimension at most $d+1$. In fact, $M$ splits 
off the complex linear subspace $\bigcap_{H\in \A} H$, whose dimension 
we call the {\em corank}\/ of $\A$. Thus, setting $\rank(\A)\coloneqq 
d+1-\corank(\A)$, we have that $\dim(K)\le \rank(\A)$. 
If $\corank(\A)=0$, we will say that $\A$ is {\em essential}. .

The group $\C^*$ acts freely on $\C^{d+1}\mysetminus \set{0}$ via 
$\zeta\cdot (z_0,\dots,z_{d})=(\zeta z_0,\dots, \zeta z_{d})$. 
The orbit space is the complex projective space of dimension $d$, 
while the orbit map, $\pi\colon \C^{d+1}\mysetminus \set{0} \to \CP^{d}$,  
$z \mapsto [z]$, is the Hopf fibration. 
The set  $\P(\A)=\set{\pi(H)\colon H\in \A}$ is an 
arrangement of codimension $1$ projective subspaces in $\CP^{d}$. 
Its complement, $U=U(\A)$, coincides with the quotient 
$\P(M)=M/\C^*$. 
The Hopf map restricts to a bundle map, $\pi\colon M\to U$, with fiber 
$\C^{*}$. Fixing a hyperplane $H_0\in \A$, we see that $\pi$ is 
also the restriction to $M$ of the bundle map 
$\C^{d+1}\mysetminus H_0\to \CP^{d} \mysetminus \pi(H_0) \cong \C^{d}$.  
This latter bundle is trivial, and so we have a diffeomorphism  
$M \cong U\times \C^*$.

\subsection{Fundamental group}
\label{subsec:pi1}
Fix a basepoint $x_0$ in the complement of $\A$, and consider the 
fundamental group $G(\A)=\pi_1(M(\A),x_0)$. 
For each hyperplane $H\in \A$, pick a meridian curve about $H$, 
oriented compatibly with the complex orientations on $\C^{d+1}$ and $H$, 
and let $\gamma_H$ denote the based homotopy class of this curve, 
joined to the basepoint by a path in $M$.  By the van Kampen theorem, 
then, the arrangement group, $G=G(\A)$, is generated by the set 
$\set{\gamma_H : H\in \A}$. Using the braid monodromy algorithm from 
\cite{CS97}, one may obtain a finite presentation 
of the form $G=F_n/R$, where $F_n$ is the rank $n$ free group on 
the set of meridians and the relators in $R$ belong to the commutator 
subgroup $F_n'$. Consequently, the abelianization of the arrangement 
group, $G_{\ab}=H_1(G;\Z)$, is isomorphic to $\Z^n$.

\begin{example}
\label{ex:config}
The reflection arrangement of type $\operatorname{A}_{n-1}$, also known 
as the braid arrangement, consists of the diagonal hyperplanes 
$H_{ij}=\set{z_i-z_j=0}$ in  $\C^{n}$. The intersection 
lattice is the lattice of partitions of the set $\set{1,\dots,n}$, 
ordered by refinement.  The complement $M$ is the configuration space 
of $n$ ordered points in $\C$, which is a classifying space for 
the Artin pure braid group on $n$ strings, $P_{n}$. 
\end{example} 

Under the diffeomorphism $M\cong U\times \C^{*}$, the arrangement group 
splits as $\pi_1(M)\cong \pi_1(U)\times \pi_1(\C^{*})$, where the central subgroup 
$\pi_1(\C^{*})=\Z$ corresponds to the subgroup of $\pi_1(M)$
generated by the product of the meridional curves $\gamma_H$ (taken 
in the order given by an ordering of the hyperplanes). We shall denote 
by $\overline{\gamma}_H=\pi_{\sharp}(\gamma_H)$ the image of $\gamma_H$ 
under the induced homomorphism $\pi_{\sharp}\colon \pi_1(M)\surj \pi_1(U)=\pi_1(M)/\Z$. 

For the purpose of computing the group $G(\A)=\pi_1(M(\A))$, it is enough to assume that 
the arrangement $\A$ lives in $\C^3$, in which case $\bar{\A}=\P(\A)$ 
is an arrangement of (projective) lines in $\CP^2$. This is clear when the 
rank of $\A$ is at most $2$, and may be achieved otherwise 
by taking a generic $3$-slice, an operation which does not
change either the poset $L_{\le 2}(\A)$ or the group $G(\A)$. 
For a rank-$3$ arrangement, the set $L_1(\A)$ is in $1$-to-$1$ correspondence 
with the lines of $\bar{\A}$, while $L_2(\A)$ is in $1$-to-$1$ correspondence 
with the intersection points of $\bar{\A}$.  Moreover, the poset structure of $L_{\le 2}(\A)$ 
mirrors the incidence structure of the point-line configuration $\bar{\A}$.   

The {\em localization}\/ of an arrangement $\A$ at a flat $X\in L(\A)$ 
is defined as the sub-arrange\-ment $\A_{X}\coloneqq \{H\in \A\mid H\supset X\}$. 
The inclusion $\A_{X}\subset \A$ gives rise to an inclusion of complements, 
$j_X\colon M(\A) \inj M(\A_{X})$.  Choosing a point $x_0$ sufficiently 
close to $\bz\in \C^{d+1}$, we can make it a common basepoint for both 
$M(\A)$ and all the local complements $M(\A_X)$.  As shown in \cite{DSY16}, 
there exist basepoint-preserving maps $r_X\colon M(\A_X)\to M(\A)$ 
such that $j_X\circ r_X\simeq \id$ relative to $x_0$; moreover, if 
$H\in \A$ and $H\not\supset X$, then the map $r_X \circ j_X \circ r_H$ is 
null-homotopic. Consequently, the induced homomorphisms 
$(r_X)_{\sharp}\colon G(\A_X) \to G(\A)$ are all injective.  

For an arrangement $\A$ in $\C^3$, we will say that a rank-$2$ flat $X$ 
has multiplicity $q=q_X$ if $\abs{\A_X}=q$, or, equivalently, if the point 
$\P(X)$ has exactly $q$ lines from $\bar{\A}$ passing through it. 
In this case, the localized sub-arrangement $\A_X$ is a pencil of 
$q$ planes. Consequently, $M(\A_X)$ is homeomorphic to 
$(\C \mysetminus \{\text{$q-1$ points}\}) \times \C^* \times \C$,
and thus it is a classifying space for the group 
$G(\A_X) \cong F_{q-1}\times \Z$.

\subsection{Cohomology ring}
\label{subsec:OS}

The cohomology ring of a hyperplane arrangement 
complement $M=M(\A)$ was computed by Brieskorn \cite{Br}, 
building on the work of Arnol'd on the cohomology 
ring of the pure braid group. In \cite{OS}, Orlik and 
Solomon gave a simple description of this ring, solely 
in terms of the intersection lattice $L(\A)$, as follows.  
Fix a linear order on $\A$, and let $E$ be the exterior 
algebra over $\Z$ with generators $\set{e_H \mid H\in \A}$ 
in degree $1$.  Next, define a differential $\partial \colon E\to E$ 
of degree $-1$, starting from $\partial(1)=0$ 
and $\partial(e_H)=1$, and extending $\partial$ to 
a linear operator on $E$, using the graded Leibniz rule. 
Finally, let $I(\A)$ be the ideal of $E$ generated by $\partial e_{\B}$, 
for all $\B\subset \A$ such that $\codim \bigcap_{H\in \B} H < \abs{\B}$, 
where $e_\B\coloneqq \prod_{H\in \B} e_H$.  Then 
\begin{equation}
\label{eq:OS-alg}
H^*(M(\A);\Z)\cong E/I(\A).
\end{equation}

The inclusions $\{j_X\}_{X\in L(\A)}$ assemble into a map 
$j\colon M\to \prod_{X\in L(\A)} M(\A_X)$. The work of 
Brieskorn \cite{Br} insures that the homomorphism induced 
by $j$ in cohomology is an isomorphism in all 
positive degrees.  By the K\"{u}nneth formula, then, we have that 
$H^k(M;\Z) \cong \bigoplus_{X\in L_k(\A)} H^k(M(\A_X);\Z)$, 
for all $k\ge 1$.  It follows that the homology groups of 
the complement of $\A$ are torsion-free, with ranks given by
\begin{equation}
\label{eq:betti-M}
b_k(M)=\sum_{X\in L_k(\A)} (-1)^k \mu(X),
\end{equation}
where $\mu\colon L(\A) \to \Z$ is the M\"{o}bius function, 
defined inductively by $\mu(\C^{d+1})=1$ 
and $\mu(X)=-\sum_{Y\supsetneq X} \mu(Y)$. The homology 
groups of the projectivized complement, $U=\P(M)$, are 
also torsion free, with ranks computed inductively from the formulas 
$b_0(U)=1$ and $b_k(U)+b_{k-1}(U)=b_k(M)$ for $k\ge 1$.

In particular, we have that 
$H_1(M;\Z)\cong \Z^n$, with basis $\set{x_H : H\in \A}$, where 
$x_H$ is the homology class represented by the 
meridional curve $\gamma_H$. Moreover, 
$H_1(U;\Z)=H_1(M;\Z) \slash \big(\sum_{H\in\A} x_H\big)\cong \Z^{n-1}$. 
We will denote by $\overline{x}_H=[\overline{\gamma}_H]$ the image of 
$x_H$ in $H_1(U;\Z)$.

\subsection{Formality}
\label{subsec:arr-formal}
A connected, finite-type CW-complex $X$ is said to be {\em formal}\/ 
if its rational cohomology algebra, $H^*(X;\Q)$, can be connected by a 
zig-zag of quasi-isomorphisms to $A^{*}_{\mathrm{PL}}(X)$, the algebra 
of polynomial differential forms on $X$ defined by Sullivan in \cite{Sullivan77}. 
The notion of $q$-formality is defined similarly, with the cdga morphisms 
in the zig-zag only being required to induce isomorphisms in degrees 
up to $q$ and monomorphisms in degree $q+1$. It is known that a 
$q$-formal CW-complex of dimension $q+1$ is actually formal. 
Moreover, if $Y\to X$ is a finite, regular cover and $Y$ is $q$-formal, 
then $X$ is also $q$-formal.  For more on this topic we refer to 
\cite{SW-forum, Su-formal} and references therein.

For an arrangement $\A$ in $\C^{d+1}$, the complement $M$ is formal, 
in a very strong sense. Indeed, for each $H\in \A$, the $1$-form 
$\omega_H= \frac{1}{2\pi \ii} d \log f_H$ on $\C^{d+1}$ restricts 
to a $1$-form on $M$.  As shown by Brieskorn \cite{Br},  
if $\mathcal{D}$ denotes the subalgebra 
of the de~Rham algebra $\Omega^*_{\rm dR}(M)$ 
generated over $\R$ by these $1$-forms, 
the correspondence $\omega_H \mapsto [\omega_H]$ 
induces an isomorphism $\mathcal{D} \to H^*(M;\R)$. 
Sullivan's machinery from \cite{Sullivan77} then 
implies that $M$ is formal.  Alternatively, it is known 
that the mixed Hodge structure on $H^*(M;\Q)$ 
is pure; thus, the ``purity implies formality" results of 
Dupont \cite{Du} and Chataur--Cirici \cite{CC} yield another 
proof of the formality of $M$.

\section{Cohomology jump loci of arrangements}
\label{sect:cjl-arr}

\subsection{Resonance varieties}
\label{subsec:res-var}

Let $A$ be a graded, graded-commutative algebra over $\C$. 
We will assume that each graded piece $A^q$ is finite-dimensional 
and $A^0=\C$. For each element $a \in A^1$, we turn the algebra 
$A$ into a cochain complex, $(A , \delta_a)$, with differentials 
$\delta^q_a\colon A^q\to A^{q+1}$, $u\mapsto au$. The fact that 
$\delta_a^{q+1}\circ \delta_a^q=0$ follows at once from the observation that $a^2=-a^2$ 
(by graded-commutativity of multiplication in $A$), which implies $a^2=0$. 
By definition, the (degree $q$, depth $s$) {\em resonance varieties}\/ of $A$ 
are the jump loci for the cohomology of this complex, 
\begin{equation}
\label{eq:resa}
\RR^q_s(A)= \{a \in A^1 \mid \dim_{\C} H^q(A , \delta_{a}) \ge  s\}.
\end{equation}

These sets are Zariski-closed, homogeneous subsets of the affine 
space $A^1$; in particular, they are either empty or they contain the 
zero-vector $\bz\in A^1$. 
Setting $b_q(A,a)\coloneqq \dim_{\C} H^q(A , \delta_{a})$ 
for the Betti numbers of the cochain complex $(A , \delta_{a})$, we 
see that $b_q(A,\bz)$ is equal to the usual Betti number $b_q(A)=\dim_{\C} A^q$. 
Therefore, the point $\bz\in A^1$ belongs to $\RR^q_s(A)$ if and only if 
$b_q(A)\ge s$. In particular, since $A^0=\C$, we have that $\RR^0_1(A)=\{\bz\}$ 
and $\RR^0_s(A)=\emptyset$ if $s>1$.

We will mostly consider the degree 
one resonance varieties. Clearly, these varieties 
depend only on the truncated algebra $A^{\le 2}$.  More explicitly, 
$\RR^1_s(A)$ consists of $\bz$, together with all elements 
$a \in A^1$ for which there exist $u_1,\dots , u_s \in A^1$ 
such that the span of $\{a,u_1,\dots, u_s\}$ has dimension 
$s+1$ and $au_1=\cdots =au_s=0$ in $A^2$. Finally, 
if $\varphi\colon A\to B$ is a morphism of commutative 
graded algebras, and  $\varphi$ is injective in degree $1$, then 
the linear map $\varphi^1\colon A^1\to B^1$ embeds $\RR^1_s(A)$ 
into $\RR^1_s(B)$, for each $s\ge 1$. 

Completely analogous definitions work for algebras $A$ over a field 
$\k$ of characteristic different from $2$. When $\ch(\k)= 2$, special 
care needs to be taken, to account for the fact that the square of 
an element $a\in A^1$ may not vanish in this case; we refer to \cite{Su-bock} 
for details.

Now let $X$ be a connected, finite-type CW-complex. Its cohomology 
algebra, $A=H^*(X;\C)$, with multiplication given by the cup-product, 
satisfies the properties listed at the start of this section. 
Therefore, we may define the resonance varieties of the space $X$ 
to be the sets $\RR^q_s(X)\coloneqq \RR^q_s(H^*(X;\C))$, viewed as 
homogeneous subsets of the affine space $H^1(X;\C)$, and likewise for 
$\RR^q_s(X,\k)\subseteq H^1(X;\k)$.  When $M=M(\A)$ is an 
arrangement complement, the fact that $H_1(M;\Z)$ is torsion-free 
implies that $a^2=0$ for all $a\in H^1(M;\k)$, even when $\ch(\k)=2$;  
thus, the usual definition of resonance works for all fields.

\subsection{Multinets and pencils}
\label{subsec:multinets}
The resonance varieties of complements of hyperplane arrangements 
were introduced in the mid-1990s by Falk \cite{Fa97} and further 
studied in the ensuing decade in papers such as 
\cite{CS99, MS00, LY00, Su01, Su02}.
Work of Falk and Yuzvinsky \cite{FY} greatly clarified 
the nature of the degree $1$ resonance varieties of arrangements. 
Let us briefly review their construction.

A {\em multinet}\/ $\NN$ on an arrangement $\A$ consists 
of a partition $\A_1\sqcup \cdots \sqcup \A_k$ of $\A$ into $k\ge 3$ subsets; 
an assignment of multiplicities $\bm=\{m_H\}_{H\in \A}$; 
and a subset $\XX\subseteq L_2(\A)$, called the base locus, such that 
the following conditions hold:
\begin{enumerate}[itemsep=2pt, topsep=-1pt]
\item  \label{m1} 
There is an integer $\ell$ such that $\sum_{H\in\A_i} m_H=\ell$, 
for all $i\in [k]$.
\item  \label{m2} 
For any two hyperplanes $H$ and $K$ in different classes,  
$H\cap K\in \XX$.
\item  \label{m3} 
For each $X\in\XX$, the sum 
$n_X:=\sum_{H\in\A_i\colon H\supset X} m_H$ is independent of $i$.
\item  \label{m4} 
For each $1\le i \le k$ and $H,K\in \A_{i}$, there is a 
sequence $H=H_0,\ldots, H_r=K$ such that $H_{j-1}\cap H_j\not\in\XX$ for
$1\le j\le r$.
\end{enumerate}

We say that a multinet $\NN$ as above is a $(k,\ell)$-multinet, or simply a $k$-multinet.    
Without essential loss of generality, we may assume that $\gcd (\bm)=1$.  
If all the multiplicities are equal to $1$, the multinet is said to be {\em reduced}. 
If, furthermore, every flat in $\XX$  is contained in precisely one 
hyperplane from each class, the multinet is called a {\em $(k,\ell)$-net}. 

For instance, a $3$-net on $\A$ is a partition into 
$3$ non-empty subsets with the 
property that, for each pair of hyperplanes $H,K\in \A$ in 
different classes, we have $H\cap K=H\cap K\cap L$, for some 
hyperplane $L$ in the third class.
As another example, if $X\in L_2(\A)$ is a $2$-flat 
of multiplicity at least $3$, we may form a net on $\A_{X}$ by 
assigning to each hyperplane $H\supset X$ the multiplicity $1$, 
putting one hyperplane in each class, and setting $\XX=\{X\}$.

Now let $f=\prod_{H\in \A} f_H$ be a defining polynomial for $\A$. 
Given a $k$-multinet $\NN$ on $\A$, 
with parts $\A_i$ and multiplicity vector $\bm$, 
write $f_{i}=\prod_{H\in\A_{i}}f_H^{m_H}$ and define a rational map 
$\psi \colon \C^3\to\CP^1$ by $\psi(x)=\pcoor{f_1(x) : f_2(x)}$. 
There is then a set $D=\{ \pcoor{a_1 : b_1}, \dots , \pcoor{a_k : b_k}\}$ 
of $k$ distinct points in $\CP^1$
such that each of the degree $d$ polynomials $f_1,\dots, f_k$ 
can be written as $f_{i}=a_{i} f_2-b_{i}f_1$, and, 
furthermore, the image of $\psi\colon M(\A)\to\CP^1$ misses $D$, 
see  \cite{FY}.  The corestriction $\psi \colon M(\A) \to \CP^1\mysetminus D$, 
then, is the {\em pencil}\/ associated to the multinet $\NN$. 
Following \cite{PS-plms17, Su-toul}, we may describe the homomorphism 
induced in homology by this pencil, as follows. Let $\alpha_1,\dots ,\alpha_k$ 
be compatibly oriented simple closed curves on $S=\CP^1\mysetminus D$, going 
around the points of $D$, so that $H_1(S;\Z)$ is generated by the homology 
classes $c_{i}=[\alpha_{i}]$, subject to the single relation $\sum_{i=1}^k c_{i}=0$.   
Then the induced homomorphism $\psi_{*} \colon H_1(M;\Z) \to H_1(S;\Z)$ 
is given by $\psi_*(x_H) = m_H c_{i}$ for $H\in \A_{i}$, 
and thus $\psi^{*} \colon H^1(S;\Z) \to H^1(M;\Z)$ 
is given by $\psi^*(c_{i}^{\vee}) = u_i$, where 
$c_{i}^{\vee}$ is the Kronecker dual of $c_i$ 
and $u_i=\sum_{H\in \A_{i}} m_H e_H$.

It follows from the above discussion that the map 
$\psi^{*} \colon H^1(S;\C) \to H^1(M;\C)$ 
is injective, and thus sends $\RR^1_1(S)$ to $\RR^1_1(M)$.
Let us identify $\RR^1_1(S)$ with $H^1(S;\C)=\C^{k-1}$, and view 
$P_{\NN}\coloneqq \psi^*(H^1(S;\C))$ as lying inside 
$\RR_1(\A)\coloneqq \RR^1_1(M)$.  
Then $P_{\NN}$ is the $(k-1)$-dimensional linear 
subspace spanned by the vectors $u_2-u_1,\dots , u_k-u_1$.  
Moreover, as shown in \cite[Thms.~2.4--2.5]{FY}, this subspace is 
an essential component of $\RR_1(\A)$; that is, $P_{\NN}$ is not 
contained in any proper coordinate subspace of $H^1(M;\C)$.
More generally, suppose there is a sub-arrangement 
$\B\subseteq \A$ supporting a multinet $\NN$.   In this case, 
the inclusion $M(\A) \inj M(\B)$ induces a 
monomorphism $H^1(M(\B);\C) \inj H^1(M(\A);\C)$, 
which restricts to an embedding $\RR_1(\B) \inj \RR_1(\A)$.  
The linear space $P_{\NN}$, then, 
lies inside $\RR_1(\B)$, and thus, inside $\RR_1(\A)$.
Conversely, as shown in \cite[Thm.~2.5]{FY} 
all (positive-dimensional) irreducible components 
of $\RR_1(\A)$ arise in this fashion. 

\subsection{Characteristic varieties}
\label{subsec:cv}

Let $X$ be a connected, finite-type CW-complex. Fix a basepoint 
$x_0$ at a $0$-cell; then the fundamental group $G=\pi_1(X,x_0)$ 
is a finitely generated (in fact, finitely presented) group. Therefore, 
the group $\T_G=\Hom(G,\C^*)$ of $\C$-valued, multiplicative 
characters on $G$ is an affine, commutative algebraic group, which 
we will identify with $H^1(X;\C^*)$. Its identity $\bo$ is the trivial 
representation $g\mapsto 1\in \C^*$; the connected component 
of $G$ containing the identity, $\T_G^0$, is an algebraic torus isomorphic 
to $(\C^*)^n$, where $n=b_1(G)$. Moreover, $\T_G/\T_G^0$ 
is in bijection with the finite abelian group $\Tors (G_{\ab})$. 

The  {\em characteristic varieties}\/ of $X$ (in degree $q$ and depth $s$, 
where $q,s\ge 0$) are the jump loci for homology with coefficients in rank-$1$ 
local systems on $X$:
\begin{equation}
\label{eq:cvs}
\V^q_s(X)=\big\{ \rho\in H^1(X;\C^*)  \mid  \dim_{\C} H_q(X;\C_\rho)\ge s\big\}.
\end{equation}

Here $\C_{\rho}=\C$, with $\C[G]$-module structure defined by the 
character $\rho\colon G\to \C^*$ by setting $g\cdot z\coloneqq \rho(g) c$ for 
$g\in G$ and $z\in \C$, while $H_*(X;\C_\rho)$ denotes the homology 
of the chain complex $C_*(\widetilde{X};\C)\otimes_{\C[G]} \C_{\rho}$, where 
$C_*(\widetilde{X};\C)$ is the $G$-equivariant chain complex of the universal 
cover of $X$, with coefficients in $\C$. 

The sets $\V^q_s(X)$ are Zariski-closed subsets of the character group. 
We will denote by $\WW^q_s(X)$ the intersection of $\VV^q_s(X)$ with $\T_G^0$. 
Observe that the (degree $q$) depth of a character $\rho$, defined as 
$\depth_q(\rho)\coloneqq \dim_{\C} H_q(X; \C_{\rho})$, is equal to 
$\max\{s \mid \rho\in \VV^q_s(X)\}$; in particular, $\depth_q(\bo)=b_q(X)$, 
the $q$-th Betti number of $X$. Note also that $\V^0_1(X)=\{\bo\}$ and 
$\V^0_s(X)=\emptyset$ if $s>1$, while $\V^q_0(X)= H^1(X;\C^*)$ for all $q\ge 0$.

Completely analogous definitions work for the characteristic varieties 
$\V^q_s(X,\k)$, viewed as subsets of $H^1(X;\k^*)$, for any field $\k$.

\begin{example}
\label{ex:cv-surfaces}
Let $\Sigma_{g,n}$ be a Riemann surface of genus $g$ with $n$ punctures 
($g,n\ge 0$), and let $\chi\coloneqq \chi(\Sigma_{g,n})=2-2g-n$ be its Euler 
characteristic. Then $\V^1_s(\Sigma_{g,n})$ is equal to $H^1(\Sigma_{g,n};\C^*)$ if 
$s\le -\chi$ and it is contained in $\{\bo\}$, otherwise. 
\end{example}

The characteristic varieties $\VV^1_s(X)$ depend only on the fundamental group 
$G=\pi_1(X)$; thus, we will often denote them by $\VV^1_s(G)$. At least away 
from the trivial character, $\VV^1_s(G)$ is the zero set of the ideal 
$\ann (\bwedge^s G'/G''\otimes \C)$, where the 
$\Z{G_{\ab}}$-module structure on the group $G'/G''$ 
arises from the short exact sequence 
$1\to  G'/G'' \to G/G'' \to G'/G'' \to 1$; see, e.g., \cite{Su-abexact} 
and references therein. 
Therefore, the characteristic varieties $\VV^1_s(G)$ 
of a finitely generated group $G$ depend only on its 
maximal metabelian quotient, $G/G''$. 

\subsection{Homology of finite abelian covers}
\label{subsec:hom-cov}
The characteristic varieties control the Betti numbers of regular, 
connected, finite abelian covers $p\colon Y\to X$.  
For instance, suppose that the deck-trans\-formation group is 
cyclic of order $N$. Then the cover is determined by an 
epimorphism $\chi \colon G\surj \Z_N$, so that 
$\ker(\chi)=\im(p_{\sharp})$. 
Fix an inclusion $\iota \colon \Z_N \inj \C^*$, 
by sending $1$ to $e^{2\pi \ii/N}$. With this choice, the 
map $\chi$ yields a torsion character, 
$\rho=\iota\circ \chi\colon G\to \C^*$. 
Since $\chi$ is surjective, the induced morphism between character groups, 
$\chi^*\colon \T_{\Z_N}\to \T_{G}$, is injective, 
and so $\im(\chi^*)\cong \Z_N$. Furthermore, 
if $\xi\colon G\to \C^*$ is a non-trivial character belonging to 
$\im(\chi^*)$, then $\xi=\rho^{N/k}$ for some positive 
integer $k$ dividing $N$. 

Now view the homology groups $H_q(Y;\C)$ 
as modules over the group algebra $\C[\Z_N]\cong \C[t]/(t^N-1)$. 
By a transfer argument, the invariant submodule, 
$H_q(Y;\C) ^{\Z_N}$, is isomorphic to the trivial module
$H_q(X;\C)\cong (\C[t]/(t-1))^{b_q(X)}$. In fact, 
a result proved in various levels of generality 
in \cite{Li92, Sa95, Hi97, MS02, DeS-plms} yields  
isomorphisms of $\C[\Z_N]$-modules, 
\begin{equation}
\label{eq:decomp}
\begin{aligned}
H_q(Y;\C) &\cong 
\bigoplus_{s\ge 1}\bigoplus_{\xi\in \im(\chi^*) 
\cap  \V_s^q(X)} \C_{\xi}\\
&\cong H_q(X;\C) \oplus 
\bigoplus_{1<k\mid N} \big(\C[t]/\Phi_k(t)\big)^{\depth_q(\rho^{N/k})}, 
\end{aligned}
\end{equation}
where $\Phi_k(t)$ is the $k$-th cyclotomic polynomial. 
Consequently, 
\begin{gather}
\label{eq:betti-cover}
\begin{aligned}
b_q(Y)&=\sum_{s\ge 1}\abs{\im(\chi^*) \cap  \V_s^q(X)}\\
&=b_q(X)+ \sum_{1<k\mid N} \varphi(k) \cdot \depth_q(\rho^{N/k}),
\end{aligned}
\end{gather}
where $\varphi(k)=\deg \Phi_k(t)$ is the Euler totient function. 
Moreover, if $h\colon Y \to Y$ is the deck transformation  
corresponding to the generator $1\in \Z_N$, 
then the characteristic polynomial $\Delta_q(t)=\det (t \cdot \id - h_*)$ 
of the induced automorphism $h_*\colon H_q(Y;\C)\to H_q(Y;\C)$ 
is given by
\begin{equation}
\label{eq:cp-cover}
\Delta_q(t) =   (t-1)^{b_q(X)}  \cdot \prod_{1<k\mid N} 
\Phi_k(t)^{ \depth_q(\rho^{N/k}) }.
\end{equation}

\subsection{Characteristic varieties of arrangements}
\label{subsec:cv-arr}
Let $M$ be a smooth, quasi-projective variety. 
A general result of Arapura \cite{Ar} (as refined in 
\cite{BW}), insures that the characteristic varieties $\VV^q_s(M)$ are 
finite unions of torsion-translated subtori of the character torus.  In degree 
$q=1$, these varieties can be described more precisely, as follows.

Let $S=(\Sigma_{g,r},\mu)$ be a Riemann surface of genus 
$g\ge 0$, with $r\ge 0$ points removed (so that 
$\Sigma_{g,0}=\Sigma_g$), and with $h\ge 0$ marked points,
$(p_1,\mu_1),\dots, (p_h,\mu_{h})$, with $\mu_i\ge 2$.  A surjective, 
holomorphic map $\psi\colon M\to \Sigma_{g,n}$ is called an 
{\em orbifold fibration}\/ (or, a {\em pencil}) if the fiber 
over any non-marked point is connected, the multiplicity of each fiber 
$\psi^{-1}(p_i)$ is equal to $\mu_i$,  
and $\psi$ has an extension to the respective compactifications, 
$\bar{\psi}\colon \overline{M}\to \Sigma_{g}$, which is 
also a surjective, holomorphic map with connected 
generic fibers.  Then each positive-dimensional component 
of $\VV^1_1(M)$ is of the form $T=\psi^*(H^1(S;\C^*))$, for some pencil 
$\psi\colon M\to S$ for which the orbifold Euler characteristic of the surface,  
$\chi^{\orb}(\Sigma_{g,r}, \mu)\coloneqq 
\chi(\Sigma_{g,r})-\sum_{i=1}^{h} (1-1/\mu_i)$, is negative. 

The following result of Artal Bartolo, Cogolludo, and Matei 
(\cite[Prop. 6.9]{ACM}) helps locate characters that lie in the higher-depth 
characteristic varieties.

\begin{theorem}[\cite{ACM}]
\label{thm:acm}
Let $M$ be a smooth, quasi-projective variety.
Suppose $T_1$ and $T_2$ are two distinct, positive-dimensional 
irreducible components of $\VV^1_r(M)$ and $\VV^1_s(M)$, 
respectively.  If $\xi\in T_1\cap T_2$ is a torsion character, 
then $\xi \in \VV^1_{r+s} (M)$. 
\end{theorem}

Now let $\A$ be an arrangement of $n$ hyperplanes in $\C^{d+1}$, with 
complement $M=M(\A)$. The characteristic varieties $\VV^q_s(M)$ 
are subsets of the character torus $H^1(M;\C^*)=(\C^*)^n$.  
Moreover, the tangent cone at the identity $\bo$ to $\VV^q_s(M)$ coincides 
with the resonance variety $\RR^q_s(M)$, for each $q, s\ge 1$. This 
``Tangent Cone Theorem" (which does not hold for all quasi-projective manifolds)
relies in an essential way on the formality of the arrangement complement, and 
was proved in \cite{CS99, LY00, DPS-duke, DP-ccm} in various levels 
of generality.  Let $\exp\colon H^1(M;\C)\to H^1(M;\C^*)$ be the 
coefficient homomorphism induced by the exponential map $\C\to \C^*$. 
Then, if $P\subset H^1(M;\C)$ is one of the linear 
subspaces comprising $\RR^q_s(M)$, its image under 
the exponential map, $\exp(P)\subset H^1(M;\C^*)$, 
is one of the subtori comprising $\VV^q_s(M)$. Furthermore, 
the correspondence $P\leadsto T=\exp(P)$ gives  
a bijection between the components of $\RR^q_s(M)$ 
and the components of $\V^q_s(M)$ passing through $\bo$, which  
in turn yields an identification $\TC_{\bo}(\V^q_s(M))=\RR^q_s(M)$ 
 for each $q, s\ge 1$.

In degree $q=1$, each positive-dimensional 
component of $\VV^1_1(M)$ that passes through $\bo$ is of the form 
$T=\psi^*(H^1(S;\C^*))$, for some pencil 
$\psi\colon M\to S=\CP^{1}\mysetminus \{\text{$k$ points}\}$ with $k\ge 3$. 
An easy computation shows that 
$\VV^1_s(S)=H^1(S;\C^*)=(\C^*)^{k-1}$ for all 
$s \le k-2$.  Hence, the subtorus $T$ is a 
$(k-1)$-dimensional component of $\VV^1_1(M)$ that contains $\bo$ 
and lies inside $\VV^1_{k-2}(M)$. 

\subsection{Torsion-translated subtori}
\label{subsec:tt-subtori}
Let $(\Sigma_{g,r},\mu)$ be a $2$-dimensional orbifold as above.
For our purposes here we may assume $r\ge 1$, in which case the orbifold 
fundamental group $\Gamma\coloneqq \pi_1^{\orb}(\Sigma_{g,r}, \mu)$, 
is isomorphic to the free product $F_{n} * \Z_{\mu_1} * \cdots * \Z_{\mu_{h}}$, 
where $n=2g+r-1$. 
Note that $\Gamma_{\ab}=\Z^{n}\oplus \Lambda$, where 
$\Lambda=\Z_{\mu_1}\oplus \cdots \oplus \Z_{\mu_{\ell}}$ is 
the torsion subgroup, and each component of the character group 
$\T_{\Gamma}=\T_{\Gamma}^{0} \times \T_{\Lambda}$ 
is of the form $\lambda\cdot \T_{\Gamma}^0$ 
for some $\lambda=(\lambda_{1}\dots, \lambda_{h})\in \T_{\Lambda}$. 
Let $\ell(\lambda)=\abs{\{i\colon \lambda_i\ne 1\}}$. 
A computation detailed in \cite[Prop.~2.10]{ACM} shows that 
\begin{equation}
\label{eq:v1piorb}
\V^1_s(\Gamma)=\begin{cases}
\T_{\Gamma} & 
\text{if $s\le n-1$},
\\[2pt]
(\T_{\Gamma} \mysetminus \T_{\Gamma}^0) \cup \{\bo\}  
& \text{if $s= n$},
\\[2pt]
\bigcup_{\ell(\lambda) \ge n-s+1}  \lambda\cdot \T_{\Gamma}^0 
& \text{if $n<s<n+h$},
\end{cases}
\end{equation}
and is empty if $s\ge n+h$. 

Now suppose $M$ is a smooth, quasi-projective variety, and 
$\psi\colon M\to (\Sigma_{g,r},\mu)$ is an orbifold pencil 
with either $n\ge 2$, or $n=1$ and $h>0$. 
Since the generic fiber of $\psi$ is connected, the induced homomorphism 
on orbifold fundamental groups, $\psi_{\sharp}\colon G=\pi_1(M)\to 
\Gamma=\pi_1^{\orb} (\Sigma_{g,r},\mu)$, is surjective. Therefore, the induced 
morphism $\psi_{\sharp}^* \colon \T_{\Gamma}\to \T_{G}$ 
embeds $\V^1_s(\Gamma)$---as computed 
in \eqref{eq:v1piorb}---into $\V^1_s(M)$, for all $s\ge 1$.  In particular, if 
$\psi\colon M\to (\C^*,m)$ is an orbifold pencil with a single multiple 
fiber of multiplicity $m\ge 2$, then 
there is a $1$-dimensional algebraic subtorus $T\subset H^1(M;\C^*)$ 
and a torsion character $\rho\notin T$ such that 
$\VV^1_1(M)$ contains the translated tori $\rho T,\dots, \rho^{m-1} T$.  

As shown in \cite{Su02}, the (degree $1$, depth $1$) characteristic 
variety of an arrangement complement may have irreducible components 
that do not pass through the origin (see Section \ref{subsec:del-b3}). 
A combinatorial machine for producing translated subtori in the characteristic 
varieties of certain arrangements was given in \cite{DeS-plms}. 
Namely, suppose $\A$ admits a {\em pointed multinet}, 
that is, a multinet $\NN$ and a hyperplane $H\in \A$ 
for which $m_H>1$, and $m_H \mid n_X$ 
for each flat $X$ in the base locus such that $X\subset H$.  
Letting $\A'=\A\mysetminus \{H\}$ 
be the deletion of $\A$ with respect to $H$, it turns out 
that $\VV^1_1(M(\A'))$ has a component which is a $1$-dimensional 
subtorus of $H^1(M(\A');\C^*)$, translated by a character of order $m_H$.   
Whether all positive-dimensional translated subtori in the (degree $1$, 
depth $1$) characteristic varieties of arrangements occur in this fashion 
is an open problem. It is also an open problem whether the isolated 
(torsion) points in the characteristic varieties of an arrangement 
are combinatorially determined.

\subsection{Cohomology jump loci of the projectivized complement}
\label{subsec:cjl-U}
Once again, let $\A$ be a (central) hyperplane arrangement in $\C^{d+1}$. 
The next result relates the cohomology jump loci of the complement $M=M(\A)$ 
to those of the projectivized complement, $U=\P(M)$. A more precise relationship 
in degrees $q>1$ will be given in Corollary \ref{cor:cjl-MU-bis}.

\begin{prop}
\label{prop:cjl-MU}
Let $\pi\colon M\to U$ be the restriction of the Hopf map,
 $\pi\colon \C^{d+1}\mysetminus \set{\bz} \to \CP^{d}$, to 
the complement of $\A$, and set $n=\abs{\A}$. Then,
\begin{enumerate}[itemsep=2pt, topsep=-1pt]
\item \label{mu1}
The induced homomorphism $\pi^*\colon H^1(U;\C)\inj H^1(M;\C)$ restricts 
to isomorphisms $\RR^1_s(U)\isom \RR^1_{s}(M)$ for all $1\le s<n$ and 
$\RR^q_1(U)\cup \RR^{q-1}_1(U)\isom \RR^q_1(M)$ for all $q\ge 1$.
\item \label{mu2}
The induced morphism $\pi^*\colon H^1(U;\C^*)\inj H^1(M;\C^*)$ 
restricts to isomorphisms  $\VV^1_s(U)\isom \VV^1_{s}(M)$ for all $1\le s<n$ and 
$\VV^q_1(U)\cup \VV^{q-1}_1(U)\isom \VV^q_1(M)$ for all $q\ge 1$.
\end{enumerate}
\end{prop}

\begin{proof}
As noted previously, upon fixing a hyperplane $H_0\in \A$, 
the restriction to $M=M(\A)$ of the (trivial) bundle map 
$\pi\colon \C^{d+1}\mysetminus H_0\to \CP^{d} \mysetminus \pi(H_0)$  
yields a diffeomorphism $M\isom  U\times \C^*$ so that the following 
diagram commutes,
\begin{equation}
\label{eq:triangle}
\begin{tikzcd}[column sep=14pt, column sep=10pt]
M \ar[rr, "\simeq"] \ar[dr, "\pi"']&&U\times \C^*\ar[dl, "\pr_1"]
\\
&\, U .
\end{tikzcd}
\end{equation}
Thus, we may replace in the argument the map $\pi\colon M\surj U$ by the 
first-coordinate projection map $\pr_1\colon U\times \C^*\to U$.  
At this stage, the claims in depth $s=1$ follow from the product formulas 
for cohomology jump loci from \cite[Prop.~13.1]{PS-plms10}. For completeness, 
we provide a full argument, which works in all cases.

For part \eqref{mu1}, consider the cohomology algebras 
$A=H^*(U\times \C^*;\C)$,  $A_1=H^*(U;\C)$, and $A_2=H^*(\C^*;\C)$, and 
let $a = (a_1, a_2)$ be an element in $A^1=A_1^1\oplus A_2^1$. By the 
K\"{u}nneth formula, the cochain complex $(A,\delta_a)$ splits as a tensor 
product of cochain complexes, $(A_1,\delta_{a_1})\otimes_{\C} (A_2,\delta_{a_2})$. 
Therefore, 
\begin{equation}
\label{eq:bettiAomoto}
b_q(A,a)=\sum_{i+j=q} b_i(A_1,a_1)b_j(A_2,a_2).
\end{equation}
Clearly, $b_0(A_2,0)=b_1(A_2,0)=1$ and $b_j(A_2,a_2)=0$ otherwise.
Therefore, 
\begin{equation}
\label{eq:betti-bis}
b_q(A,(a_1,a_2))=\begin{cases}
b_q(A_1, a_1)+b_{q-1}(A_1, a_1) &\text{if $a_2=0$,}
\\
0 &\text{if $a_2\ne 0$}.
\end{cases}
\end{equation}
In particular, $b_1(A,(a_1,0))=b_1(A_1,a_1)$ if $a_1\ne 0$ and 
$b_1(A,\bz)=b_1(A_1,\bz)+1$. The first claim follows at once 
from these formulas.

For part \eqref{mu2}, let us identify $G=\pi_1(U\times \C^*)$ with 
$\pi_1(U)\times \Z$ and the universal cover of $U\times \C^*$ 
with $\widetilde{U}\times \C$. We then have a $G$-equivariant isomorphism 
of chain complexes, $C_{*}(\widetilde{U\times \C^*})\cong  
C_{*}(\widetilde{U})\otimes_{\C}C_{*}(\C)$. 
Given a character $\rho=(\rho_1,\rho_2)$ in 
$\Hom(G,\C^*)\cong \Hom(\pi_1(U),\C^*)
\times \C^*$, we obtain an isomorphism  
$C_{*}(U\times \C^*,\C_{\rho}) \cong C_{*}(U,\C_{\rho_1})
\otimes_{\C} C_{*}(\C^*,\C_{\rho_2})$. 
Therefore, $H_q(U\times \C^*;\C_{\rho})\cong \bigoplus_{i+j=q} 
H_{i}(U;\C_{\rho_1}) \otimes_{\C} H_{j}(\C^*;\C_{\rho_2})$, 
and the second claim follows from the fact that 
$H_{0}(\C^*;\C)=H_{1}(\C^*;\C)=\C$ and $H_{j}(\C^*;\C_{\rho_2})=0$, 
otherwise.
\end{proof}

Now fix an ordering $H_1,\dots, H_n$ of the hyperplanes in $\A$ and set $H_0=H_n$. 
Then $H^1(M;\C^*)$ may be identified with $(\C^*)^n$, with coordinates $t=(t_1,\dots,t_n)$ and 
$H^1(U;\C^*)$ may be identified with $(\C^*)^{n-1}$, with coordinates $(t_1,\dots,t_{n-1})$. 
The characteristic varieties of $U$ are then given by
\begin{equation}
\label{eq:vux}
\VV^q_s(U)=\{ t\in (\C^{*})^n \mid t\in \text{$\VV^q_s(M)$ 
and $t_1\cdots t_n=1$}\};
\end{equation}
that is, $\VV^q_s(U)$ is the subvariety of $(\C^{*})^n$ obtained by intersecting 
$\VV^q_s(M)$ with the subtorus $(\C^{*})^{n-1}=\set{t : t_1\cdots t_n=1}$.  
Furthermore, the induced homomorphism 
$\pi^*\colon H^1(U;\C^*) \inj H^1(M;\C^*)$ 
may be identified with the monomial map 
\begin{equation}
\label{eq:mono-map}
(\C^{*})^{n-1} \inj (\C^{*})^n, \quad 
(t_1,\dots, t_{n-1})\mapsto \big(t_1,\dots, t_{n-1}, t_1^{-1}\cdots t_{n-1}^{-1}\big). 
\end{equation} 
In turn, this map restricts to isomorphisms $\VV^1_s(U)\isom \VV^1_s(M)$ for all 
$1\le s<n$ and $\VV^q_1(U)\cup \VV^{q-1}_1(U)\isom \VV^q_1(M)$ for all $q\ge 1$, 
where, in fact, $\VV^q_1(U)\cup \VV^{q-1}_1(U)= \VV^q_1(U)$, as we shall see in Corollary 
\ref{arr:arr-prop}.

Similar considerations apply to the resonance varieties of $M$ and $U$, with 
the induced homomorphism $\pi^*\colon H^1(U;\C) \inj H^1(M;\C)$ 
being identified with the linear map 
$\C^{n-1} \inj  \C^n$, 
$(x_1,\dots, x_{n-1})\mapsto (x_1,\dots, x_{n-1}, -(x_1+\cdots + x_{n-1}))$.

\section{Milnor fibrations of arrangements}
\label{sect:mf}

\subsection{The Milnor fibration of a multi-arrangement}
\label{sect:mf-multi}
Let $\A$ be a central arrangement of $n$ hyperplanes in $\C^{d+1}$, 
and fix an ordering on $\A$.  
To each hyperplane $H\in \A$, we may associate a multiplicity $m_H\in \N$.  
This yields a multi-arrangement $(\A,\bm)$, where $\bm=(m_H)_{H\in \A} \in \N^{n}$ 
is the resulting multiplicity vector, and a homogeneous polynomial, 
\begin{equation}
\label{eq:fm}
f_{\bm}=\prod_{H\in \A} f_H^{m_H}
\end{equation}
of degree $N=\sum_{H\in \A} m_H$. Note that $f_{\bm}$ is a 
proper power if and only if $\gcd(\bm)>1$, where 
$\gcd(\bm)=\gcd(m_H\colon H\in \A)$.

The polynomial map $f_{\bm}\colon \C^{d+1} \to \C$ restricts 
to a map $f_{\bm}\colon M(\A) \to \C^{*}$.  As shown by 
Milnor \cite{Mi} in a much more general context, $f_{\bm}$ 
is the projection map of a smooth, locally trivial bundle, 
known as the {\em (global) Milnor fibration}\/ of the 
multi-arrangement $(\A,m)$, 
\begin{equation}
\label{eq:mfib}
\begin{tikzcd}[column sep=26pt]
F_{\bm} \ar[r]  & M \ar[r, "f_{\bm}"] & \C^*. 
\end{tikzcd}
\end{equation}

The typical fiber of this fibration, $f_{\bm}^{-1}(1)$, 
is a smooth manifold of dimension $2d$, called the {\em Milnor fiber}\/ 
of the multi-arrangement, denoted $F_{\bm}=F_{\bm} (\A)$. 
It is readily seen that $F_{\bm}$ is a Stein domain of complex 
dimension $d$, and thus has the homotopy type of a finite CW-complex 
of dimension at most $d$---in fact, of dimension at most $\rank(\A)-1$. 
Moreover, $F_{\bm}$ is connected if and only 
if $\gcd(\bm)=1$, a condition we will assume henceforth. 
As shown in \cite{Su-toul}, the homomorphism 
$(f_{\bm})_{\sharp} \colon \pi_1(M)\to \pi_1(\C^*)$  
induced on fundamental groups by $f_{\bm}$ is the map 
$\mu_\bm\colon \pi_1(M)\to \Z$ given by $x_H \mapsto m_H$. 
In the case when all the multiplicities $m_H$ are equal to $1$, the polynomial 
$f=f_{\bm}$ is the usual defining polynomial and  $F=F_{\bm}$ is the usual 
Milnor fiber of $\A$.

For each $\theta\in [0,1]$, let $F_{\theta}= f_{\bm}^{-1}(e^{2\pi \ii \theta})$ 
be the fiber over the point $e^{2\pi \ii \theta}\in \C^*$. 
For each $z\in M$, the path $\gamma_{\theta}\colon [0,1] \to \C^*$, 
$t\mapsto e^{2\pi \ii t \theta}$ lifts to a path 
$\tilde\gamma_{\theta,z}\colon [0,1] \to M$, 
$t\mapsto e^{2\pi \ii t \theta/N} z$ which satisfies $\tilde\gamma_{\theta,z}(0)=z$. 
Notice that $f_{\bm}(\tilde\gamma_{\theta,z}(1))= e^{2\pi \ii  \theta} f_{\bm}(z)$;  
thus, if $z\in F_{0}=F_{\bm}$, then $\tilde\gamma_{\theta,z}(1)\in F_{\theta}$. 
By definition, the {\em monodromy}\/ of the Milnor fibration 
is the diffeomorphism $h\colon F_0\to F_1$ given by 
$h(z)=\tilde\gamma_{1,z}(1)$.  In view of these observations, 
we may interpret $h$ as the self-diffeomorphism $h\colon F_{\bm}\to F_{\bm}$ of 
order $N$ given by $z \mapsto e^{2\pi \ii/N} z$, and identify the complement 
$M$ with the mapping torus of $h$. 

\subsection{The Milnor fiber as a finite cyclic cover}
\label{subsec:mf cover} 
The monodromy diffeomorphism $h\colon F_{\bm}\to F_{\bm}$ 
generates a cyclic group of order $N=\sum_{H\in \A} m_H$    
which acts freely on $F_{\bm}$. The quotient space, $F_{\bm}/\Z_N$, 
may be identified with the projective complement, $U=\P(M)$, in a manner 
such that the projection map, $\sigma_{\bm}\colon F_{\bm} \surj F_{\bm}/\Z_N$, 
coincides with the restriction of the Hopf fibration map, $\pi\colon M\surj U$, 
to the subspace $F_{\bm}$.  Letting $\iota_{\bm}\colon F_{\bm} \to M$ denote the 
inclusion map, all this information may be summarized in the diagram
\begin{equation}
\label{eq:cd-milnor}
\begin{tikzcd}[column sep=24pt, row sep=24pt]
& \C^*  \ar[d, "\upsilon"] \ar[dr, "z\mapsto z^N"]
\\
F_{\bm} \ar[r, "\iota_{\bm}"]  \ar[dr, "\sigma_{\bm}"'] 
& M \ar[r, "f_{\bm}"] \ar[d,"\pi"] & \C^* ,\\
& U
\end{tikzcd}
\end{equation}
where both the row and the column are fibrations and the diagonal arrows are 
$N$-fold cyclic covers. Consequently, the Euler characteristic of the Milnor fiber 
is given by $\chi(F_{\bm})=N\cdot \chi(U)$. Taking fundamental groups 
in \eqref{eq:cd-milnor}, we obtain the diagram
\begin{equation}
\label{eq:fmz}
\begin{tikzcd}[column sep=32pt, row sep=28pt]
&[-10pt] & \Z \ar[dr, "\cdot N"]  \arrow[hook, "\upsilon_{\sharp}"]{d} & &[-10pt]
\\
1  \ar[r] & \pi_1(F_{\bm})\ar[dr, two heads, "(\sigma_{\bm})_{\sharp}"']   \ar[r, "(\iota_{\bm})_{\sharp}"]
& \pi_1(M) \arrow[two heads, "\pi_{\sharp}"]{d} \ar[r, "\mu_\bm"] & \Z\ar[r] & 1,
\\
& & \pi_1(U)
\end{tikzcd}
\end{equation}
with exact row and column. By construction, $\sigma_{\bm}=\pi\circ \iota_{\bm}$, 
and so the lower triangle commutes. The upper triangle in \eqref{eq:fmz} also 
commutes, since $\upsilon_{\sharp}(1)$ is the product of the meridians $\gamma_H$ 
(taken in the order given by an ordering of the hyperplanes), and since $N=\sum_{H\in \A} m_H$. 
Hence, the homomorphism $\mu_{\bm}\colon \pi_1(M) \surj \Z$ descends to an 
epimorphism, 
\begin{equation}
\label{eq:chi}
\begin{tikzcd}[column sep=22pt]
\chi_{\bm}\colon \pi_1(U)\ar[r, two heads] &\Z_N, 
\end{tikzcd}
\end{equation}
given by $\overline{\gamma}_H \mapsto m_H \bmod N$.
As shown in \cite{CS95, CDS, Su-conm11, Su-toul}, the regular, 
$N$-fold cyclic cover $\sigma_{\bm}\colon F_{\bm} \to U$ is classified 
by this epimorphism. 
In particular, the usual Milnor fiber $F=F(\A)$ is classified by the 
``diagonal" homomorphism, $\chi\colon \pi_1(U)\surj \Z_n$, given by 
$\chi(\overline{\gamma}_H )=1$, for all $H\in \A$.

\subsection{The characteristic polynomial of the algebraic monodromy}
\label{subsec:cp-mono}
We now fix an ordering on the $n$ hyperplanes of $\A$, 
and identify the character group $H^1(U;\C^*)$ with $(\C^*)^{n-1}$.  
Recall we also fixed an embedding $j \colon \Z_N\inj \C^*$, $1\mapsto e^{2\pi \ii /N}$. 
By \eqref{eq:chi}, the character $\rho_{\bm} = j\circ \chi_{\bm} \colon \pi_1(U)\to \C^*$ 
is given by $\overline{\gamma}_H \mapsto e^{2\pi \ii m_H/N}$; hence, for each divisor $k$ of $N$, 
the character $\rho_{\bm}^{N/k}$ takes $\overline{\gamma}_H$ to $e^{2\pi \ii /k}$.
By formula \eqref{eq:betti-cover}, the Betti numbers of the Milnor fiber 
$F_{\bm}=F_{\bm}(\A)$ are given by 
\begin{equation}
\label{eq:betti-mf}
b_q(F_{\bm})= b_q(U) + \sum_{1< k | N} \varphi(k) \depth_q(\rho_{\bm}^{N/k}).
\end{equation}
Likewise, formula \eqref{eq:cp-cover} implies that the characteristic 
polynomial of the algebraic monodromy 
$h_*\colon H_q(F_{\bm};\C) \to H_q(F_{\bm};\C)$ is given by
\begin{equation}
\label{eq:charpoly}
\Delta_q(t) =   (t-1)^{b_q(U)}  \cdot \prod_{1<k\mid N} 
\Phi_k(t)^{ \depth_q(\rho_{\bm}^{N/k}) }.
\end{equation}

In the above expressions, the crucial quantities are the (non-negative)
depths of the characters $\rho_{\bm}^{N/k}\in H^{1}(U;\C^*)$, which 
depend on the position of these characters with respect to the 
characteristic varieties $\VV^q_s(U)$.  Here are some basic (well-known) 
examples of how such a computation goes.

\begin{example}
\label{ex:mf pencil}
Let $\A$ be a pencil of $n\ge 3$ lines through the 
origin of $\C^2$ defined by the polynomial $f=x^{n}-y^{n}$.  
Then $U$ is homeomorphic to $\Sigma_{0,n}=\C\mysetminus \{\text{$n-1$ points}\}$, 
and so its characteristic varieties are $\VV^1_1(U)=\cdots = \VV^1_{n-2}(U)\cong (\C^*)^{n-1}$ 
and $\VV^1_{n-1}(U)= \{\bo\}$ (see Example \ref{ex:cv-surfaces}).  
It follows that $b_1(F)=n-1+(n-2)(n-1)=(n-1)^2$ and 
$\Delta_1(t)=(t-1)(t^{n}-1)^{n-2}$. In turn, either this computation or an Euler characteristic 
argument shows that $F=\Sigma_{g,n}$, a Riemann surface of genus $g=\binom{n-1}{2}$ 
with $n$ punctures.
\end{example}

\begin{figure}
\centering
\begin{tikzpicture}[scale=0.78]
\draw[style=thick,densely dashed,color=blue] (-0.5,3) -- (2.5,-3);
\draw[style=thick,densely dotted,color=red]  (0.5,3) -- (-2.5,-3);
\draw[style=thick,color=dkgreen] (-3,-2) -- (3,-2);
\draw[style=thick,densely dotted,color=red]  (3,-2.68) -- (-2,0.68);
\draw[style=thick,densely dashed,color=blue] (-3,-2.68) -- (2,0.68);
\draw[style=thick,color=dkgreen] (0,-3.1) -- (0,3.1);
\node at (-2,-2) {$\bullet$};
\node at (2,-2) {$\bullet$};
\node at (0,2) {$\bullet$};
\node at (0,-0.7) {$\bullet$};
\node at (-2.5,0.4) {$x+z$};
\node at (2.5,0.4) {$x-y$};
\node at (0,-3.5) {$y+z$};
\node at (-3.7,-2) {$y-z$};
\node at (-2.5,-3.3) {$x-z$};
\node at (2.5,-3.3) {$x+y$};
\end{tikzpicture}
\caption{A $(3,2)$-net on the braid arrangement}
\label{fig:braid}
\end{figure}
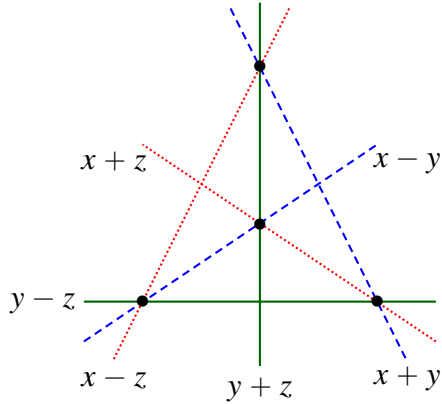

\begin{example}
\label{ex:mf braid arr}
Let $\A$ be the braid arrangement in $\C^3$, defined by the polynomial 
$f=(x+y)(x-y)(x+z)(x-z)(y+z)(y-z)$. Its complement $M$ is, up to a $\C$ factor, 
homeomorphic to the complement of the reflection arrangement of type 
$\operatorname{A}_3$ in $\C^4$; thus, $\pi_1(M)=P_4$. Labeling the 
hyperplanes of $\A$ as the factors of $f$, the flats in $L_2(\A)$ may 
be labeled as $136$, $145$, $235$, and $246$. 
The braid arrangement supports a $(3,2)$-net, 
corresponding to the partition $(12|34|56)$ depicted in Figure \ref{fig:braid}. 
This net defines a rational map, $\psi \colon \CP^2 \dashrightarrow \CP^1$, 
sending $\pcoor{x,y,z}\mapsto \pcoor{x^2-y^2,x^2-z^2}$. In turn, this map restricts to 
a pencil, $\psi\colon U\to \Sigma_{0,3}=\CP^1 \mysetminus 
\{  \pcoor{0,1}, \pcoor{1,0}, \pcoor{1,1} \}$, 
which yields by pullback a $2$-dimensional essential component 
of $\V^1_1(U)$, namely, the subtorus 
\begin{equation}
\label{eq:braid-torus}
T=\{ (s, s, t, t, (st)^{-1}) : s,t \in \C^*\}.
\end{equation}

Letting $\rho \colon \pi_1(U)\to \C^*$, 
$\overline{\gamma}_H\mapsto e^{2\pi \ii/6}$ 
be the diagonal character which defines the $\Z_6$-cover $\sigma\colon F\to U$, 
we have that $\rho^2\in T$, yet $\rho\notin T$. Since $\VV^1_2(U)=\{\bo\}$, 
it follows that $b_1(F)=5 + \varphi(3)\cdot \depth_1 (\rho^2) = 5+2\cdot 1=7$  
and $\Delta_1(t)=(t-1)^5(t^2+t+1)$.  
\end{example}

More generally, as shown in \cite[Thm.~1.6]{PS-plms17}, if an arrangement of 
projective lines in $\CP^2$ has only double or triple points, then the 
characteristic polynomial of the algebraic monodromy of the Milnor fibration 
is given by a completely combinatorial formula. 

For an arrangement $\A$  
and a prime $p$, define $\beta_p(\A)\coloneqq \max\{s : \omega \in \RR^1_s(M(\A);\Z_p)\}$, 
where $\omega=\sum_{H\in \A} e_H\in H^1(M(\A);\Z_p)$.  Clearly, the non-negative integer 
$\beta_p(\A)$ depends only on $L_{\le 2}(\A)$ and $p$. 

\begin{theorem}[\cite{PS-plms17}]
\label{thm:ps-triple}
Suppose $L_2(\A)$ has only flats of multiplicity $2$ and $3$.  Then 
$\beta_3(\A)\in \{0, 1, 2\}$ and 
\[
\Delta_{1}(t)=(t-1)^{\abs{\A}-1}\cdot (t^2+t+1)^{\beta_3(\A)}.
\]
Moreover, $\beta_3(\A)\ne 0$ if and only if $\A$ supports a $3$-net. 
\end{theorem}

\subsection{Trivial algebraic monodromy}
\label{subsec:tam}
Henceforth, we will concentrate mainly on the case when the 
algebraic monodromy of the Milnor fibration is trivial. 
More precisely, suppose $F_{\bm}\to M\to \C^*$ is the Milnor fibration of a 
multi-arrangement $(\A,\bm)$, with monodromy diffeomorphism $h\colon 
F_{\bm}\to F_{\bm}$.  We say that $(\A,\bm)$ has {\em trivial algebraic 
monodromy over $\k$}\/ (where $\k$ is either $\Z$ or a field) 
if $h_*\colon H_*(F_{\bm};\k)\to H_*(F_{\bm};\k)$ is the identity. 
Clearly, when $\k$ a field, this condition only depends on the 
characteristic of $\k$.

The condition that $h_*\colon H_q(F_{\bm};\Q)\to H_q(F_{\bm};\Q)$ be 
the identity is equivalent to $\Delta_q(t)=(t-1)^{b_q(F_{\bm})}$. Thus, in 
view of formulas \eqref{eq:betti-mf} and \eqref{eq:charpoly}, the 
condition is equivalent to $b_q(F_{\bm})=b_q(U)$, where $U=\P(M)$. 
Therefore,  $(\A,\bm)$ has trivial algebraic monodromy over $\Q$ if and only if 
$H_*(F_{\bm};\Q)\cong H_*(U;\Q)$. In fact, more is true. 
As noted previously, the homology 
groups of both $U$ and $M$ are torsion-free. Making use of the 
K\"{u}nneth formula for $M\cong U\times \C^*$ and the Wang 
exact sequence for the fibration $F_{\bm}\to M\to \C^*$, we conclude that 
$(\A,\bm)$ has trivial algebraic monodromy over $\k$ (where $\k=\Z$ 
or $\k$ a field) if and only if $H_*(F_{\bm};\k)\cong H_*(U;\k)$. 
Likewise, $h_*\colon H_1(F_{\bm};\Z)\to H_*(F_{\bm};\Z)$ 
is the identity if and only if $H_1(F_{\bm};\Z)=\Z^{n-1}$, where $n=\abs{\A}$. 

\begin{remark}
\label{rem:graphic}
Triviality of the algebraic monodromy in degree $q=1$ does not imply triviality 
of the action in higher degrees. For instance, if 
$\A$ is a graphic arrangement, that is, a sub-arrangement of the braid arrangement 
of type $\operatorname{A}_{n-1}$ from Example \ref{ex:config}, then $h_*$ always 
acts trivially on $H_1(F(\A);\Q)$, except when $\A$ is a reflection arrangement 
of type $\operatorname{A}_{2}$ or $\operatorname{A}_{3}$, see \cite[Thm.~B]{McP}. 
On the other hand, if $\A$ is the braid arrangement of type $\operatorname{A}_{n-1}$, 
then $h_*$ always acts non-trivially on the top homology group, $H_{n-2}(F(\A);\Q)$, 
see \cite[\S7]{CS98}. 
\end{remark}

Unlike the homology groups of the complement $M$, 
examples from \cite{CDS, DeS-plms, Yo20} 
show that the homology groups of the Milnor fiber $F_{\bm}$ may have 
non-trivial torsion. Therefore, if the monodromy $h\colon F_{\bm}\to F_{\bm}$ 
acts as the identity on $H_q(F_{\bm};\Q)$, for some $q\ge 1$, 
we cannot conclude that it also acts as the identity on $H_q(F_{\bm};\Z)$. 
Indeed, if $H_q(F_{\bm};\Z)$ has torsion, then the Wang sequence 
of the fibration $F_{\bm} \to M \to \C^*$ shows that $h_*\colon H_q(F_{\bm};\Z) 
\to H_q(F_{\bm};\Z)$ cannot be equal to the identity. 
We will illustrate this point in Sections \ref {sect:delB3}--\ref{sect:mf-yoshi}. 

\subsection{Triviality tests}
\label{subsec:triv-test}
Let $\A$ be a central arrangement of $n$ hyperplanes in $\C^3$. 
For the usual Milnor fiber $F=F(\A)$, there are two useful tests informing on whether 
the algebraic monodromy $h_*\colon H_1(F;\C) \to H_1(F;\C)$ is equal to the identity. 
Both of these tests are based on the nature of the multinets supported by $L(\A)$ and 
of the characteristic varieties of the complement $M=M(\A)$. 

We start with a criterion insuring the triviality of the algebraic monodromy. 
We will say that a subvariety of the algebraic torus $(\C^*)^n$ is {\em essential}\/ 
if it is not contained in any proper coordinate subtorus.

\begin{prop}
\label{prop:no-multinet} 
If the characteristic variety $\VV^1_1(M)$ has no essential irreducible components, 
then the algebraic monodromy $h_*\colon H_1(F;\C) \to H_1(F;\C)$ is trivial.
\end{prop}

\begin{proof}
Set $n=\abs{\A}$. 
By formulas \eqref{eq:betti-cover} and \eqref{eq:chi}, the first Betti 
number of $F$ is given by 
\begin{equation}
\label{eq:betti1F}
b_1(F)=\sum_{s\ge 1}\abs{\im(\chi^*) \cap \V^1_s(U)},
\end{equation}
where $U=\P(M)$ and $\chi\colon \pi_1(U)\to \Z_n$ is the homomorphism 
that sends each meridian curve $\overline{\gamma}_H$ to $1$. The 
cyclic subgroup $\im(\chi^*)\subset H^1(U;\C^*)\cong (\C^*)^{n-1}$ 
is generated  by the character $\rho=(\zeta, \dots, \zeta)$, where 
$\zeta=e^{2 \pi \ii/n}$. 

Recall that the Hopf map $\pi\colon M\to U$ induces 
a homomorphism $\pi^*\colon H^1(U;\C^*)\to H^1(M;\C^*)$ 
which restricts to an isomorphism $\V^1_1(U)\isom \V^1_1(M)$. 
Recall also that the map $\pi^*\colon  (\C^*)^{n-1} \to (\C^*)^{n}$
is given in coordinates by formula \eqref{eq:mono-map}. Since $\zeta^n=1$, 
it follows that $\pi^*(\im(\chi^*))$ is the cyclic subgroup of 
$(\C^*)^{n}$ generated by $\tilde{\rho}=(\zeta, \dots, \zeta, \zeta)$. 
Therefore, $\pi^*(\im(\chi^*))$ is contained in the diagonal 
subtorus $T_{\Delta}=\{ (z,\dots, z) \mid z\in \C^*\}\subset (\C^*)^{n}$. 
 
Now let $C$ be an irreducible component of $\V^1_1(M)$. 
By our assumption, $C$ lies in a proper coordinate subtorus 
of $H^1(M;\C^*)=(\C^*)^{n}$;  
hence, $C$ intersects intersects $T_{\Delta}$ only at the identity. 
It follows that $\pi^*(\im(\chi^*)) \cap \VV^1_1(M)=\{\bo\}$, and therefore  
$\im(\chi^*) \cap \V^1_1(U)=\{\bo\}$. In view of \eqref{eq:betti1F}, this shows  
that $b_1(F)=n-1$, and the proof is complete.
\end{proof}

The following criterion for non-triviality of the algebraic 
monodromy is proved in \cite[Thm.~8.3]{PS-plms17}, 
based on results from \cite{DP-pisa} and \cite{FY}.

\begin{prop}[\cite{PS-plms17}]
\label{prop:red-multi}
Let $\A$ be a central arrangement in $\C^3$. 
If $\A$ admits a reduced multinet, then the algebraic 
monodromy (in degree $1$) over $\C$ is non-trivial.
\end{prop}

If an arrangement supports essential multinets, but none of those multinets 
is reduced, then the algebraic monodromy (over $\C$) may still be trivial, 
as illustrated by the $\operatorname{B}_3$ reflection arrangement from 
Section \ref{subsec:b3}, though it may also be non-trivial, as illustrated 
by the complex reflection arrangements 
of type $G(3d+1,1,3)$ with $d>0$ from \cite[Ex.~8.11]{PS-plms17}.

\section{Cohomology jump loci of Milnor fibers}
\label{sect:cjl-milnor}

In this section, we analyze the resonance and characteristic 
varieties of the Milnor fibers of a hyperplane arrangement, 
under the assumption that the algebraic monodromy of the Milnor 
fibration is trivial. 

\subsection{Cohomology jump loci in finite regular covers}
\label{subsec:cjl-covers}

We start with some general results regarding the behavior of jump loci in 
finite regular covers. These results were proved by Dimca and Papadima 
in \cite[Prop.~2.1, Cor.~2.2, Thm.~2.8]{DP-pisa}. In the next two propositions, 
we state them in a slightly modified form, that is better adapted to our context. 

\begin{prop}[\cite{DP-pisa}]
\label{prop:jumpcov-1}
Let $p\colon Y\to X$ be a finite regular cover.  Then,
\begin{enumerate}[itemsep=2pt, topsep=-1pt]
\item \label{cov1} 
The induced homomorphism  $p^*\colon H^1(X;\C)\to H^1(Y;\C)$ is an injection 
which restricts to maps   
$p^*\colon \RR^q_s(X) \to  \RR^q_s(Y)$, for all $q\ge 0$ and $s\ge 1$. 
\item \label{cov2}
The morphism $p^*\colon H^1(X;\C^*)\to H^1(Y;\C^*)$ restricts to maps 
$p^*\colon \V^q_s(X) \to  \V^q_s(Y)$, for all $q\ge 0$ and $s\ge 1$. 
\end{enumerate}
\end{prop}

When the action of the group of deck transformations of the cover is homologically 
trivial (in degree $1$), more can be said. 

\begin{prop}[\cite{DP-pisa}]
\label{prop:jumpcov-2}
Let $p\colon Y\to X$ be a finite regular cover. Suppose the group of deck 
transformations acts trivially on $H_1(Y;\Q)$. Then, 
\begin{enumerate}[itemsep=2pt, topsep=-1pt]
\item \label{dp1}
The map $p^*\colon H^1(X;\C)\to H^1(Y;\C)$ is an isomorphism 
that identifies $\RR^1_s(X)$ with $\RR^1_s(Y)$, for all $s\ge 1$.
\item  \label{dp2}
The map $p^*\colon H^1(X;\C^*)^0\to H^1(Y;\C^*)^0$ 
is a surjection with finite kernel. Moreover, if $X$ is $1$-formal, 
this map establishes a bijection between the sets of irreducible 
components of $\WW^1_s(X)$ and $\WW^1_s(Y)$ that pass through 
the identity, for all $s\ge 1$. 
\end{enumerate}
\end{prop}

Let us note that the homological triviality hypothesis of this proposition is definitely needed. 
For instance, if $X$ is a wedge of $n$ circles ($n\ge 2$), and $p\colon Y\to X$ 
is a $k$-fold cover ($k\ge 2$), then $\RR^1_1(X)=\C^{n}$, whereas 
$\RR^1_1(Y)=\C^{k(n-1)+1}$, and so the map $p^*\colon \RR^1_1(X)\to \RR^1_1(Y)$ 
is not surjective. 

\subsection{Cohomology jump loci in extensions}
\label{subsec:cjl-milnor-split}

Next, we recall some general results relating cohomology jump loci in 
group extensions. In \cite{Su-abexact}, we made 
a detailed analysis of how the characteristic and resonance varieties 
behave under certain split extensions with trivial monodromy 
action in homology. We summarize those results in the form 
that will be needed here. 

\begin{theorem}[\cite{Su-abexact}]
\label{thm:cv-extensions}
Let $\begin{tikzcd}[column sep=16pt]
\!\!1\ar[r] & K\ar[r, "\iota"]
& G \ar[r] & Q\ar[r] & 1\!\!
\end{tikzcd}$ 
be a split exact sequence of finitely generated groups. 
Assume $Q$ is abelian. Then,

\begin{enumerate}[itemsep=2pt]
\item \label{cj1}
If $Q$ acts trivially on $H_1(K;\Z)$, 
then the induced homomorphism $\iota^*\colon H^1(G;\C^*) \to H^1(K;\C^*)$ 
restricts to maps $\iota^* \colon \VV^1_s(G)\to \VV^1_s(K)$ 
for all $s\ge 1$; furthermore, 
$\iota^* \colon \VV^1_1(G)\to \VV^1_1(K)$ is a surjection.

\item \label{cj2}
If $Q$ is torsion-free and acts trivially on $H_1(K;\Q)$, 
then the map $\iota^*\colon H^1(G;\C^*)^0 \to H^1(K;\C^*)^0$ 
restricts to maps $\iota^* \colon \WW^1_s(G)\to \WW^1_s(K)$ 
for all $s\ge 1$; furthermore, 
$\iota^* \colon \WW^1_1(G)\to \WW^1_1(K)$ is a surjection.

\item \label{cj3}
If $Q$ acts trivially on $H_1(K;\Q)$ and $G$ is $1$-formal, 
then the map $\iota^*\colon H^1(G; \C) \to H^1(K; \C)$ 
restricts to maps $\iota^* \colon \RR^1_s(G)\to \RR^1_s(K)$ for all $s\ge 1$; 
furthermore, $\iota^* \colon \RR^1_1(G)\to \RR^1_1(K)$ is a surjection.
\end{enumerate}
\end{theorem}

All these results are sharp. For instance, regarding part \eqref{cj3}, 
we make the following observation: In depth $s>1$, the map 
$\iota^* \colon \RR^1_s(G)\to \RR^1_s(K)$ is 
not necessarily a surjection, while in depth $s=1$ it is not necessarily 
an isomorphism. We illustrate both assertions with an example 
(see \cite{PS-jlms07, PS-adv09} for the necessary background).

\begin{example}
\label{ex:bb-path4}
Let $G=\langle a_1,\dots, a_4 \mid [a_1,a_2]=[a_2,a_3]=[a_3,a_4]=1\rangle$ be 
the right-angled Artin group associated to a path $\Gamma$ on $4$ vertices, and 
let $K$ be the corresponding Bestvina--Brady group. We then have 
an exact sequence $1\to K\xrightarrow{\iota} G \xrightarrow{\nu} \Z \to 1$, 
where $\nu$ is the homomorphism sending each generator $a_i$ to $1$. 
Since $\Gamma$ is a tree, the group $K$ is free (of rank $3$), 
and so $\RR^1_1(K)=\RR^1_2(K)=\C^3$. 
On the other hand, $\RR^1_1(G)=\{x_2=0\}\cup \{x_3=0\}$ and 
$\RR^1_2(G)=\{x_2=x_3=x_4=0\}\cup \{x_1=x_2=x_3=0\}$. 
Thus, the map $\iota^* \colon \RR^1_s(G)\to \RR^1_s(K)$ is 
not a surjection for $s=2$ and is not an isomorphism for $s=1$.
\end{example}

\subsection{Cohomology jump loci of Milnor fibers}
\label{subsec:cjl-milnor}

As before, let $(\A,\bm)$ be a multi-arrange\-ment. 
Denote by $\iota_{\bm}\colon F_{\bm}\inj M$ the inclusion map 
of the Milnor fiber $F_{\bm}=F_{\bm}(\A)$ into the complement $M=M(\A)$ 
and by $\sigma_{\bm}=\pi\circ \iota_{\bm}\colon F_{\bm} \to U$ the restriction 
of the Hopf map $\pi\colon M\to U=\P(M)$ to $F_{\bm}$. 
Applying Proposition \ref{prop:jumpcov-1} to the finite, regular cover 
$\sigma_{\bm}\colon F_{\bm} \to U$,  
we obtain the following immediate corollary. 

\begin{corollary}
\label{cor:cv-mf}
For all $q, s\ge 1$, the following hold.
\begin{enumerate}[itemsep=2pt, topsep=-1pt]
\item  \label{res-f-gen}
The induced morphism $\sigma_{\bm}^*\colon H^1(U;\C)\inj H^1(F_{\bm};\C)$ 
restricts to maps $\RR^q_s(U)\inj \RR^q_s(F_{\bm})$.
\item \label{cv-f-geb}
The morphism $\sigma_{\bm}^*\colon H^1(U;\C^*)\to H^1(F_{\bm};\C^*)$  
restricts to maps $\VV^q_s(U)\to \VV^q_s(F_{\bm})$.
\end{enumerate}
\end{corollary}

Consider now the usual Milnor fiber, $F=F(\A)$, and the 
finite cyclic cover $\sigma\colon F\to U$. In general, 
the morphism $\sigma^*\colon \V^1_1(U) \to \V^1_1(F)$ from 
Corollary \ref{cor:cv-mf}, part \eqref{cv-f-geb} is not surjective. 
For instance, suppose $\A$ admits a non-trivial, reduced multinet, 
and let $T$ be the corresponding component of $\V^1_1(U)$. It is 
then shown in \cite[Cor.~3.3]{DP-pisa} that $\V^1_1(F)$ has a component 
$W$ passing through the identity and containing $\sigma^*(T)$ 
as a proper subset. We illustrate this phenomenon with a concrete 
example.

\begin{example}
\label{ex:cv mf braid}
Let $\A$ be the braid arrangement from Example \ref{ex:mf braid arr}.  
Recall that $\VV^1_1(U)\subset (\C^*)^5$ has four local components, 
$T_1,\dots, T_4$, corresponding to the four triple points of $\bar{\A}$, 
and an essential, $2$-dimensional component $T$, corresponding to 
the $(3,2)$-net depicted in Figure~\ref{fig:braid}.  Let 
$\psi\colon U\to S=\Sigma_{0,3}$ 
be the pencil defined by this net, so that $T=\psi^*(H^1(S;\C^*))$. 
Note that $S=U(\B)$, where $\B$ is the arrangement in $\C^2$ 
defined by the polynomial $xy(x-y)$; therefore, the Milnor fiber 
of this arrangement, $\hat{S}=F(\B)$, may be identified with 
$\Sigma_{1,3}=S^1\times S^1 \mysetminus \{\text{$3$ points}\}$ 
(see Example \ref{ex:mf pencil}). Let $\nu \colon \hat{S} \to S$ 
be the corresponding $\Z_3$-cover, and consider the pull-back diagram,
\begin{equation}
\label{eq:3-fold-braid}
\begin{tikzcd}[column sep=32pt, row sep=28pt]
\hat{U}\ar[d, "\tau"]  \ar[r, "\hat{\psi}"] & \hat{S}\phantom{.} \ar[d, "\nu"]
\\
U \ar[r, "\psi"] & S .
\end{tikzcd}
\end{equation}

In the above, $\tau\colon \hat{U}\to U$ is the pull-back along $\psi$ 
of the cover $\nu \colon \hat{S} \to S$. By construction, $\tau$ is the  
$\Z_3$-cover defined by the diagonal homomorphism 
$\pi_1(U)\to \Z_3$. It is readily seen  
that $H_1(\hat{U};\Z)=\Z^7$. By \cite[Prop.~2]{Zu}, the map 
$\hat{\psi}$ is an (irrational) pencil on $\hat{U}$; therefore, the $4$-dimensional 
torus $W_0=\hat{\psi}^*(H^1(\hat{S};\C^*))$ is a component of the 
characteristic variety $\VV^1_1(\hat{U})\subset (\C^*)^7$. 

Finally, let $F=F(\A)$ be the Milnor fiber of $\A$. 
Note that the $\Z_6$-cover $\sigma\colon F\to U$ factors as the composite 
$F\xrightarrow{\kappa} \hat{U}\xrightarrow{\tau} U$, where $\kappa$ is a 
$2$-fold cover. Therefore, the characteristic variety $\VV^1_1(F)$ 
has four $2$-dimensional components, $\sigma^*(T_1), \dots ,\sigma^*(T_4)$, 
as well as a $4$-dimensional component, 
$W=\kappa^*(W_0)$, which strictly contains $\sigma^*(T)$.  
Direct computation shows that $\VV^1_1(F)$ has no other irreducible components. 
\end{example}

\subsection{Arrangements with trivial algebraic monodromy}
\label{subsec:cjl-mf-triv}
We return now to the general case of a multi-arrangement $(\A,\bm)$. 
As usual, let $F_{\bm}$ be the Milnor fiber of the multi-arrangement, 
and let $\sigma_{\bm}\colon F\to U$ be the corresponding $\Z_N$-cover, 
where $N=\sum_{H\in \A} m_H$. Using the machinery developed above, 
we obtain the following theorem, which sharpens results from \cite{DP-pisa} 
in a way that will be needed later on.

\begin{theorem}
\label{thm:cjl-mf-trivialmono}
Suppose the monodromy $h\colon F_{\bm}\to F_{\bm}$ 
induces the identity on $H_1(F_{\bm};\Q)$.  Then, 
\begin{enumerate}[itemsep=2pt, topsep=-1pt]
\item  \label{res-f-bis}
The induced homomorphism $\sigma^*_{\bm}\colon H^1(U;\C)\to H^1(F_{\bm};\C)$ 
is an isomorphism that identifies $\RR^1_s(U)$ with $\RR^1_s(F_{\bm})$, for all $s\ge 1$.
\item \label{cv-f-bis}
The induced homomorphism  $\sigma^*_{\bm}\colon H^1(U;\C^*)\to H^1(F_{\bm};\C^*)^0$ 
is a surjection with kernel isomorphic to $\Z_N$. Moreover, 
\begin{enumerate}
\item \label{pa}
For each $s\ge 1$, the map $\sigma^*_{\bm}$ establishes a bijection 
between the sets of irreducible components of $\VV^1_s(U)$ and 
$\WW^1_s(F_{\bm})$ that pass through the identity. 
\item \label{pb}
The map 
$\sigma^*_{\bm} \colon \VV^1_1(U)\to \WW^1_1(F_{\bm})$ is a surjection.
\end{enumerate}
\end{enumerate}
\end{theorem}

\begin{proof}
We start with some preliminary observations. From the discussion in 
Section \ref{subsec:mf cover}, we know that the map 
$\sigma_{\bm}\colon F_{\bm} \to U$ is a regular $\Z_N$-cover, 
corresponding to the exact sequence
\begin{equation}
\label{eq:ses-fu}
\begin{tikzcd}[column sep=34pt]
1  \ar[r] &[-10pt] \pi_1(F_{\bm}) \ar[r, "(\sigma_{\bm})_{\sharp}"]
& \pi_1(U) \ar[r, "\chi_\bm"] & \Z_N\ar[r] &[-10pt] 1 .
\end{tikzcd}
\end{equation}
As noted in Section \ref{subsec:tam}, the assumption that 
$h\colon F_{\bm}\to F_{\bm}$ induces the identity on 
$H_1(F_{\bm};\Q)$ is equivalent to 
$H_1(F_{\bm};\Q)\cong H_1(U;\Q)$. It follows that 
we have an exact sequence,
\begin{equation}
\label{eq:ses-fu-tm}
\begin{tikzcd}[column sep=34pt]
0  \ar[r] &[-10pt] H_1(F_{\bm};\Z)/\Tors \ar[r, "(\sigma_{\bm})_{*}"]
& H_1(U;\Z) \ar[r, "(\chi_\bm)_*"] & \Z_N\ar[r] &[-10pt] 0 .
\end{tikzcd}
\end{equation}

We now proceed with the proof. 
Claim \eqref{res-f-bis} follows directly from Proposition \ref{prop:jumpcov-2}, 
part \eqref{dp1}. To prove the first assertion of Claim \eqref{cv-f-bis}, we apply 
the functor $H^1(-;\C^*)=\Hom(-,\C^*)$ to the exact sequence \eqref{eq:ses-fu-tm}. 
Since the abelian groups $\C^*$ is divisible, and thus an injective $\Z$-module, 
we obtain an exact sequence, 
\begin{equation}
\label{eq:ses-fu-tm-dual}
\begin{tikzcd}[column sep=32pt]
0  &[-10pt]  \ar[l] H^1(F_{\bm};\C^*)^0
& H^1(U;\C^*)  \ar[l, "\sigma_{\bm}^{*}" ' , pos=0.4]  
& H^1(\Z_N;\C^*) \ar[l, "\chi_\bm^*" ', pos=0.4] &[-10pt] \ar[l] 0 .
\end{tikzcd}
\end{equation}
Identifying the group $H^1(\Z_N;\C^*)$ with its Pontryagin dual, $\Z_N$, 
completes the proof of the first part of Claim \eqref{cv-f-bis}. 

Since the space $U$ is formal, Claim  \eqref{pa} 
follows from Proposition \ref{prop:jumpcov-2}, part \eqref{dp2}.

Finally, recall from diagram \eqref{eq:fmz} that we have a (split) exact sequence, 
\begin{equation}
\label{eq:ses-fm}
\begin{tikzcd}[column sep=34pt]
1  \ar[r] &[-10pt] \pi_1(F_{\bm}) \ar[r, "(\iota_{\bm})_{\sharp}"]
& \pi_1(M) \ar[r, "\mu_\bm"] & \Z\ar[r] &[-10pt] 1 .
\end{tikzcd}
\end{equation}
Our hypothesis on the monodromy $h$ says that $\Z$ acts 
trivially on $H_1(F_{\bm};\Q)$. Thus, we may 
apply Theorem \ref{thm:cv-extensions} and conclude 
that the morphism $\iota_{\bm}^*\colon H^1(M;\C^*) \to H^1(F_{\bm};\C^*)^0$ 
restricts to a surjection, $\iota_{\bm}^* \colon \VV^1_1(M)\surj \WW^1_1(F_{\bm})$.
On the other hand, as shown in Proposition \ref{prop:cjl-MU}, part \eqref{mu2}, 
the map  $\pi^*\colon H^1(U;\C^*)\to H^1(M;\C^*)$ restricts to an isomorphism,  
$\pi^*\colon \VV^1_1(U)\isom \VV^1_1(M)$. Since $\sigma_{\bm}=\pi\circ \iota_{\bm}$, 
Claim \eqref{pb} follows, and the proof is complete.
\end{proof}

\section{Abelian duality and propagation of cohomology jump loci}
\label{sect:abel-prop}

\subsection{Abelian duality spaces}
\label{subsec:abel-dual}
Let $X$ be a space having the homotopy type of a connected, finite CW-complex 
of dimension $m$. Without loss of generality, we may assume $X$ has a 
single $0$-cell, say, $x_0$. Letting $G=\pi_1(X,x_0)$ be 
the fundamental group of $X$, the group ring of its abelianization, 
$R=\Z[G_{\ab}]$, may be viewed as a module over $\Z{G}$ via 
extension of scalars. Inspired by the classical notion of ``duality group" due 
to Bieri and Eckmann, the following concept was introduced 
in \cite{DSY17}.

We say that $X$ is an {\em abelian duality space}\/ (for short, 
{\em $\ab$-duality space}) of dimension $m$ if $H^q(X;R)=0$ for $q\ne m$ 
and $H^m(X,R)$ is non-zero and torsion-free. In that case, for all (left) 
$R$-modules $A$ and all $q\ge 0$, we have  isomorphisms 
\begin{equation}
\label{eq:abdualityiso}
H^q(X;A)\cong \Tor^{R}_{m-q}(D;A)\cong H_{m-q}(G_{\ab}; D\otimes_{\Z} A),
\end{equation}
where $D=H^m(X;R)$, viewed as an $R$-module. 
Consequently, if $Y\to X$ is a connected, regular abelian cover, classified 
by an epimorphism $G \xrightarrow{\ab} G_{\ab} \xrightarrow{\chi} H$, 
where $H$ is a (finitely generated) abelian group, then   
$H_{q}(Y;\Z) \cong \Ext_{R}^{m-q}(D,H)$, for all $q\ge 0$.

Motivated by our work in \cite{Su-abexact}, we adapt this definition 
to a related context. Let $G_{\abf}=G_{\ab}/\Tors$ 
be the maximal torsion-free abelian quotient of $G$. 
We say that $X$ is a {\em torsion-free abelian duality space}\/ 
(for short, {\em $\abf$-duality space}) 
of dimension $m$ if the above conditions are satisfied with 
$R=\Z[G_{\ab}]$ replaced by $\Z[G_{\abf}]$. 
Clearly, if $X$ is an abelian duality space and $G_{\ab}=H_1(X;\Z)$ is 
torsion-free, then $X$ is a torsion-free abelian duality space.

\subsection{Formality}
\label{subsec:mf-formal} 
Recall that both an arrangement complement, $M=M(\A)$, and its 
projectivization, $U=\P(M)$, are (rationally) formal spaces. Moreover, 
for every choice of multiplicities $\bm$ on $\A$, the Milnor fiber 
$F_{\bm}$ is a cyclic, regular cover of $U$. This raises the question 
of whether these Milnor fibers are also formal spaces---or, 
at least $q$-formal, for some $q\ge 1$. The following lemma 
gives a sufficient condition for this to happen.

\begin{lemma}[\cite{DP-pisa}]
\label{lem:formal covers}
Let $Y \to X$ be a finite, regular cover, and 
suppose the group of deck-transformations acts trivially 
on $H_i(Y; \Q)$, for all $i\le q$. 
Then $Y$ is $q$-formal if and only if $X$ is $q$-formal.
\end{lemma}

\begin{corollary}
\label{cor:mf-formal} 
Let $(\A,\bm)$ be a multi-arrangement of rank $r$, with Milnor fiber $F_{\bm}$
and monodromy $h\colon F_{\bm}\to F_{\bm}$. 
\begin{enumerate}[itemsep=2pt, topsep=-1pt]
\item \label{ff1}
If the algebraic monodromy 
$h_*\colon H_{i}(F_{\bm};\Q)\to H_{i}(F_{\bm};\Q)$ is the identity 
for all $i\le q$, for some $q\ge 1$, then $F_{\bm}$ is $q$-formal. 
\item \label{ff2}
If $h_*\colon H_{i}(F_{\bm};\Q)\to H_{i}(F_{\bm};\Q)$ is the identity 
for all $i\le r-2$, then $F_{\bm}$ is formal. 
\end{enumerate}
\end{corollary}

\begin{proof}
Part \eqref{ff1} follows directly from the above lemma. For part \eqref{ff2}, 
first recall that $F_{\bm}$ has the homotopy type of a finite CW-complex 
of dimension at most $r-1$. Thus, the claim follows from part \eqref{ff1} 
and the discussion in Section \ref{subsec:arr-formal}.
\end{proof}

In general, though, Milnor fibers may be non-formal, as illustrated 
by the following example of Zuber \cite{Zu}. 

\begin{example}
\label{ex:zuber}
Let $\A=\A(3,3,3)$ be the monomial arrangement in $\C^3$ defined by the polynomial 
$f=(x^3-y^3)(y^3-z^3)(x^3-z^3)$. There are four $(3,3)$-nets on $\A$, 
associated with the partitions $(123 | 456 | 789)$, 
$(147 | 258 | 369)$, $(159 | 267 | 348)$, and $(168 | 249 | 357)$ 
in a suitable ordering of the hyperplanes. 
The first of these nets defines a rational map, 
$\psi\colon \CP^2 \dashrightarrow \CP^1$, 
$\pcoor{x:y:z} \mapsto \pcoor{x^3-y^3:x^3-z^3}$, 
which in turn restricts to a pencil $\psi\colon U\to S$ 
from $U=U(\A)$ to 
$S=\CP^1 \mysetminus \{\pcoor{1:0},\pcoor{0:1},\pcoor{1:1}\}$.  
Let $T=\psi^*(H^1(S;\C^*))$ be the essential $2$-dimensional component 
of $\VV^1_1(U)$ obtained by pullback along this pencil. 
The subgroup generated by the diagonal character $\rho\colon \pi_1(U)\to \C^*$ 
intersects $\V^1_2(U)$ at the identity $\bo$ and two other points, 
both lying on $T$, and both of order $3$. Hence, 
$\Delta_1(t)=(t-1)^8(1+t+t^2)^2$.  

Next, let $\B$ be the arrangement in $\C^2$ defined by the 
polynomial $xy(x-y)$, and let $\nu \colon \hat{S}=F(\B) \to S=U(\B)$ 
be the corresponding $3$-fold cover. As shown in \cite[Prop.~2]{Zu}, 
the rational pencil $\psi\colon U\to S=\Sigma_{0,3}$ 
lifts to an irrational pencil, $\hat\psi\colon \hat{U}\to \hat{S}=\Sigma_{1,3}$, 
as in diagram \eqref{eq:3-fold-braid}. Here 
$\tau\colon \hat{U}\to U$ is the pull-back of $\nu$ along $\psi$, that is, 
the $\Z_3$-cover defined by the diagonal homomorphism 
$\pi_1(U)\surj \Z_3$. 
It is readily seen that $H_1(\hat{U};\Z)=\Z^{12}$; therefore, the $4$-dimensional 
torus $W_0=\hat{\psi}^*(H^1(\hat{S};\C^*))$ is a component of the 
characteristic variety $\VV^1_1(\hat{U})\subset (\C^*)^{12}$. 

Finally, let $F=F(\A)$ be the Milnor fiber of $\A$. 
Then the $\Z_9$-cover $\sigma\colon F\to U$ 
factors as the composite $F\xrightarrow{\kappa} \hat{U}\xrightarrow{\tau} U$, 
where $\kappa$ is a $3$-fold cover. Therefore, the characteristic variety $\VV^1_1(F)$ 
has a $4$-dimensional component, $W=\kappa^*(W_0)$, which strictly 
contains the $2$-dimensional subtorus $\sigma^*(T)$. Write $W=\exp(L)$, 
for some linear subspace $L\subset H^1(F;\C)$. 
Using the mixed Hodge structure on $H^*(F;\C)$, Zuber  showed in \cite{Zu} 
that $L$ cannot be a component of the resonance variety $\RR^1_1(F)$.  
Therefore, $\operatorname{TC}_{\bo}(\V^1_1(F))\subsetneqq \RR^1_1(F)$, 
and so, by the Tangent Cone theorem of \cite{DPS-duke}, $F$ is not $1$-formal. 
\end{example}

\subsection{Ab- and abf-exactness}
\label{subsec:ab-exact}

Let $\begin{tikzcd}[column sep=16pt] \!F\ar[r, "\iota"] & E \ar[r, "\pi"] & B\!
\end{tikzcd}$ be a fibration sequence of con\-nected CW-complexes.  
Setting $K=\pi_1(F)$, $G=\pi_1(E)$, and $Q=\pi_1(B)$, we have an 
exact sequence $K\xrightarrow{\iota_{\sharp}} G\xrightarrow{\pi_{\sharp}}  Q\to 1$.
Moreover, the exact sequence of low-degree terms in the Serre 
spectral sequence of the  fibration takes the form 
\begin{equation}
\label{eq:serre}
\begin{tikzcd}[column sep=18pt]
 H_2(E;\Z)\ar[r, "\pi_*"]& H_2(B;\Z) \ar[r, "\delta"]  &  
 H_1(F;\Z)_{Q} \ar[r, "\iota_*"] 
& H_1(E;\Z) \ar[r, "\pi_*"] & H_1(B;\Z) \ar[r]  & 0\, ,
\end{tikzcd}
\end{equation}
where  $H_1(F;\Z)_{Q}$ denotes the coinvariants of $K_{\ab}=H_1(F;\Z)$ 
under the action of $Q$.  

Following \cite{DSY17}, we say that the fibration is {\em $\ab$-exact}\/ if 
(1) $Q$ acts trivially on $K_{\ab}$; and (2) the homomorphism  
$\delta$ is zero. In the presence of the first condition, the second condition 
is equivalent to the exactness of the sequence
$0\to K_{\ab} \to G_{\ab} \to Q_{\ab}\to  0$.  Finally, as shown in 
\cite[Prop.~8.4]{Su-abexact}, if $K\triangleleft G$ and the 
sequence $1\to K \to  G \to Q\to 1$ is exact and admits a splitting, 
then the fibration is $\ab$-exact if and only if $Q$ acts trivially on $K_{\ab}$.

As shown in \cite[Prop.~4.13]{DSY17}, the notion of $\ab$-duality 
behaves well with respect to $\ab$-exact fibrations: if any two of the spaces 
have the abelian duality property, then the third one does, too. In particular, 
the product of two $\ab$-exact exact spaces is again $\ab$-exact. 
We record here the part of this result that will be needed later on. 

\begin{prop}[\cite{DSY17}]
\label{prop:ab-fibration}
Suppose $F\to E\to B$ is an $\ab$-exact fibration of connected, 
finite-type CW-complexes.
If $E$ and $B$ are $\ab$-duality spaces of dimensions $m$ and $n$,
respectively, and if $\dim F = m-n$, then $F$ is an $\ab$-duality 
space of dimension $m-n$.
\end{prop}

By analogy with the above notions, we say that a fibration 
$F\to E\to B$ is {\em $\abf$-exact}\/ if 
$Q$ acts trivially on $K_{\abf}$ and the composite 
$H_2(B;\Z)\xrightarrow{\delta} H_1(F;\Z)_Q \surj K_{\abf}$ is zero. 
In the presence of the first condition, the second condition 
is equivalent to the exactness of the sequence
$0\to K_{\abf} \to G_{\abf} \to Q_{\abf}\to  0$. 
Alternatively, let $\delta_{\Q} \colon H_2(B;\Q)\to H_1(K;\Q)$ be the analog 
of the map $\delta$ in the exact sequence \eqref{eq:serre}
with $\Q$-coefficients. Since $K_{\abf}$ is finitely generated, 
an argument similar to the one used in \cite[Lem.~9.2]{Su-abexact} 
shows that the fibration is $\abf$-exact if and only if $Q$ acts trivially 
on $H_1(F;\Q)$ and $\delta_{\Q}$ is the zero map. 
Finally, as shown in \cite[Prop.~9.4]{Su-abexact}, 
if $K \triangleleft G$ and the sequence $1\to K \to  G \to Q\to 1$ is split exact, 
then the fibration is $\abf$-exact if and only if $Q$ acts trivially on $H_1(F;\Q)$.

The same argument as in \cite{DSY17}, using now the Serre spectral sequence 
of the fibration $F\to E\to B$ with coefficients in $\Z[G_{\abf}]$ instead of 
$\Z[G_{\ab}]$, shows the following: if any two of the spaces 
have the torsion-free abelian duality property, then the third one does, too. 
In particular, the product of two $\abf$-exact exact spaces is again $\abf$-exact. 
We record here only the result that we shall need later in this section. 

\begin{prop}
\label{prop:abf-fibration}
Suppose $F\to E\to B$ is an $\abf$-exact fibration of connected, 
finite-type CW-complexes.
If $E$ and $B$ are $\abf$-duality spaces of dimensions $m$ and $n$,
respectively, and if $\dim F = m-n$, then $F$ is an $\abf$-duality 
space of dimension $m-n$.
\end{prop}

\subsection{Propagation of jump loci}
\label{subsec:propagate}

One of the main motivations for the study of the abelian duality properties of  
spaces is the implications these properties have on the nature of the cohomology 
jump loci and the Betti numbers of those spaces. We start with a 
result relating $\ab$-duality to propagation of characteristic varieties. 

\begin{theorem}[\cite{DSY17}]
\label{thm:cv-prop}
Let $X$ be an abelian duality space of dimension $m$. 
Then the characteristic varieties of $X$ {\em propagate}; that is, 
for any character $\rho\in H^1(X;\C^*)$ such that $H^p(X;\C_\rho)\ne 0$, 
it follows that $H^q(X;\C_\rho)\ne 0$ for all $p\le q\le m$. Equivalently, 
\begin{equation}
\label{eq:cv-prop}
\{\bo\}=\V^0_1(X) \subseteq \V^1_1(X)\subseteq \V^2_1(X)\subseteq\cdots
\subseteq\V^{m}_1(X).
\end{equation}
\end{theorem}

Applying this theorem to the trivial character $\rho=\bo$, it follows at once 
that $b_q(X)> 0$ for $0\le q\le m$.  Moreover, as shown in 
\cite[Prop.~5.9]{DSY17}, we also have $b_1(X)\ge m$. 
Finally, as noted in \cite[Thm.~1.8]{LMW}, the above result implies that the 
``signed Euler characteristic" of an $m$-dimensional $\ab$-duality space, 
$(-1)^{m}\chi(X)$, is non-negative.
A similar argument---using \cite[Prop.~2.8]{DSY17}, applied to the 
$\C[G_{\abf}]$-chain complex $C_*(X;\C[G_{\abf}])$---yields the 
following result.

\begin{theorem}
\label{thm:cw-prop}
Let $X$ be an $\abf$-duality space of dimension $m$. 
Then
\begin{equation}
\label{eq:cw-prop}
\{\bo\}=\WW^0_1(X) \subseteq \WW^1_1(X)\subseteq \WW^2_1(X)\subseteq\cdots
\subseteq\WW^{m}_1(X).
\end{equation}
\end{theorem}

Now suppose $X$ is formal. Then, the Tangent Cone theorem of \cite{DPS-duke, DP-ccm}, 
allows us to identify the tangent cone at $\bo$ to $\VV^q_1(X)$ with $\RR^q_1(X)$ 
for all $q\le m$. Applying Theorem \ref{thm:cv-prop} , we obtain the following 
immediate corollary.

\begin{corollary}
\label{cor:res-prop}
Let $X$ is an abelian duality space of dimension $m$. 
If $X$ is $q$-formal, for some $q\le m$, then 
$\RR^1_1(X)\subseteq\cdots\subseteq\RR^{q}_1(X)$.
In particular, if $X$ is formal, then the resonance varieties of $X$ 
propagate; that is, $\RR^1_1(X)\subseteq \cdots\subseteq\RR^{m}_1(X)$.
\end{corollary}

\begin{remark}[\cite{DSY17}]
\label{rem:non-prop}
If $X$ is a connected, finite, $2$-dimensional CW-complex with 
$\chi(X)\ge 0$ and $b_1(X)>0$, then both the characteristic and the 
resonance varieties of $X$ propagate (that is, $\VV^1_1(X)\subseteq \VV^2_1(X)$ and 
$\RR^1_1(X)\subseteq \RR^2_1(X)$), even though $X$ may be neither 
an abelian duality space nor a formal space.
On the other hand, if $X$ is a closed, orientable $3$-manifold with 
$b_1(X)$ even and non-zero, then the resonance varieties do not 
propagate, since $\RR^1_1(X)=H^1(X;\C)$, whereas $\RR^3_1(X)=\{\bz\}$.
\end{remark}

\subsection{Abelian duality and propagation for arrangements}
\label{subsec:abdual-arr}
A basic topological property of arrangement complements is provided by 
the following result, which is proved in \cite[Thm.~5.6]{DSY16} (see also 
\cite[Thm.~6.1]{DSY17}).

\begin{theorem}[\cite{DSY16, DSY17}]
\label{thm:arr-prop}
Let $\A$ be a central arrangement of rank $r$. Then the 
complement $M=M(\A)$ is an abelian duality space of 
dimension $r$ and the projectivized complement 
$U=\P(M)$ is an abelian duality space of dimension $r-1$.
\end{theorem}

In particular, if $\A$ is a central, essential arrangement of hyperplanes in 
$\C^{d+1}$, then $M(\A)$ is an abelian duality space of dimension $d+1$ and 
$U(\A)$ is an abelian duality space of dimension $d$. 

\begin{remark}
\label{rem:stein}
More generally, let $M$ be a connected, smooth, complex 
quasi-projective variety of dimension $m$.   Suppose 
$M$ has a smooth compactification $\overline{M}$ for which the 
components of $\overline{M}\mysetminus M$ form 
a non-empty arrangement of hypersurfaces, $\A$, such that, 
for each submanifold $X$ in the intersection poset $L(\A)$,  
the complement of the restriction of $\A$ to $X$ is either 
empty or a Stein manifold. Then, by 
\cite[Thm.~1.1]{DeS-sigma}, $M$ is an abelian duality space 
of dimension $m$. Another generalization of Theorem \ref{thm:arr-prop}
is given in \cite[Thm.~1.10]{LMW}: If $M$  has a smooth compactification 
$\overline{M}$ with $b_1(\overline{M})=0$ and $M$ admits a proper, 
semi-small map to a complex algebraic torus, then the same conclusion holds. 
\end{remark}

Recall now that arrangement complements are also formal. It follows from 
Theorem \ref{thm:arr-prop} and Corollary \ref{cor:res-prop} that 
both their characteristic and resonance varieties propagate. More 
precisely, we have the following corollary.

\begin{corollary}
\label{arr:arr-prop}
Let $\A$ be a central arrangement of rank $r$, with complement $M=M(\A)$ 
and projectivized complement $U=\P(M)$. Then 
\begin{enumerate}[itemsep=2pt, topsep=-1pt]
\item \label{cjmu1}
$\V^1_1(M)\subseteq \cdots\subseteq\V^{r}_1(M)$ and  
$\RR^1_1(M)\subseteq  \cdots\subseteq\RR^{r}_1(M)$.
\item \label{cjmu2}
$\V^1_1(U)\subseteq \cdots\subseteq\V^{r-1}_1(U)$ and  
$\RR^1_1(U)\subseteq  \cdots\subseteq\RR^{r-1}_1(U)$.
\end{enumerate}
\end{corollary}

In view of part \eqref{cjmu2} of this result, Proposition \ref{prop:cjl-MU} yields 
the following immediate corollary. 

\begin{corollary}
\label{cor:cjl-MU-bis}
Let $\pi\colon M\to U$ be the restriction of the Hopf map. Then,
\begin{enumerate}[itemsep=2pt, topsep=-1pt]
\item \label{mu1-bis}
The induced homomorphism $\pi^*\colon H^1(U;\C)\inj H^1(M;\C)$ restricts 
to isomorphisms 
$\RR^q_1(U)\isom \RR^q_1(M)$ for all $q\ge 1$.
\item \label{mu2-bis}
The induced morphism $\pi^*\colon H^1(U;\C^*)\inj H^1(M;\C^*)$ 
restricts to isomorphisms  
$\VV^q_1(U)\isom \VV^q_1(M)$ for all $q\ge 1$.
\end{enumerate}
\end{corollary}

\subsection{Abelian duality and propagation for Milnor fibers}
\label{subsec:abdual-mf}
We now turn to the Milnor fibration $F_{\bm} \to M \to \C^*$ 
of a multi-arrangement $(\A,\bm)$. 
To start with, let us note that Corollary \ref{cor:cv-mf}, when used in 
conjunction with Proposition \ref{prop:cjl-MU} and Corollary \ref{cor:cjl-MU-bis}, 
has the following consequence.

\begin{corollary}
\label{cor:cv-mf-bis}
Let $\iota_{\bm}\colon F_{\bm}\inj M$ be the inclusion map 
of the Milnor fiber into the complement of $\A$.
\begin{enumerate}[itemsep=2pt, topsep=-1pt]
\item \label{iota-1}
The epimorphism $\iota_{\bm}^*\colon H^1(M;\C)\surj H^1(F_{\bm};\C)$  
restricts to maps $\RR^1_s(M)\to \RR^1_s(F_{\bm})$, for all $s\ge 1$, and 
$\RR^q_1(M) \to \RR^q_1(F_{\bm})$, for all $q\ge 1$.
\item \label{iota-2}
The epimorphism  $\iota_{\bm}^*\colon H^1(M;\C^*)\surj H^1(F_{\bm};\C^*)$  
restricts to maps $\VV^1_s(M)\to \VV^1_s(F_{\bm})$, for all $s\ge 1$, and 
$\VV^q_1(M)\to \VV^q_1(F_{\bm})$, for all $q\ge 1$.
\end{enumerate}
\end{corollary}

The next result strengthens \cite[Thm.~6.7]{DSY17}, 
where only part \eqref{mf-ab} is proved (in the particular case when 
$F=F(\A)$ is the usual Milnor fiber of an essential arrangement), 
but not part \eqref{mf-abf}. 

\begin{theorem}
\label{thm:mf-abel-duality}
Let $\A$ be a central arrangement of rank $r$ and let 
$F_{\bm}=F_{\bm}(\A)$ be the Milnor fiber associated to a multiplicity 
vector $\bm\colon \A\to \N$. 
\begin{enumerate}[itemsep=1pt]
\item \label{mf-ab}
If the monodromy action on $H_1(F_{\bm};\Z)$ is trivial, 
then $F_{\bm}$ is an $\ab$-duality space of dimension $r-1$. 
\item \label{mf-abf}
If the monodromy action on $H_1(F_{\bm};\Q)$ is trivial, 
then $F_{\bm}$ is an $\abf$-duality space 
of dimension $r-1$. 
\end{enumerate}
\end{theorem}

\begin{proof}
From Theorem \ref{thm:arr-prop}, we know that the total space of the 
Milnor fibration, $M=M(\A)$, is an $\ab$-duality space of dimension $r$. 
Thus, $M$ is also and $\abf$-duality space of the same dimension, since 
$H_1(M;\Z)=\Z^{\abs{\A}}$ is torsion-free. Clearly, the base of the 
fibration, $B=\C^*$, is both an $\ab$- and $\abf$-duality space of 
dimension $1$. In view of our hypothesis on the monodromy of 
the fibration, the two claims regarding the fiber $F_{\bm}$ now follow 
directly from Propositions \ref{prop:ab-fibration} and \ref{prop:abf-fibration}, 
respectively.
\end{proof}

Applying this theorem, we obtain the following corollary regarding 
propagation of cohomology jump loci of Milnor fibers of arrangements 
with trivial algebraic monodromy.

\begin{corollary}
\label{cor:mf-propagate}
Let $\A$ be a central arrangement of rank $r$, and let $\bm\colon \A\to \N$ 
be a choice of multiplicities.  
\begin{enumerate}[itemsep=2pt]
\item \label{cjm1}
If the monodromy action on $H_1(F_{\bm};\Z)$ is trivial, then the 
characteristic varieties of $F_{\bm}$ propagate; that is,
$\V^1_1(F_{\bm})\subseteq \V^2_1(F_{\bm})
\subseteq\cdots\subseteq\V^{r-1}_1(F_{\bm})$.

\item \label{cjm2}
If the monodromy action on $H_1(F_{\bm};\Q)$ is trivial, then the 
restricted characteristic varieties of $F_{\bm}$ propagate; that is,
$\WW^1_1(F_{\bm})\subseteq \WW^2_1(F_{\bm})\subseteq
\cdots\subseteq\WW^{r-1}_1(F_{\bm})$.

\item \label{cjm3}
If the monodromy action on $H_i(F_{\bm};\Q)$ is trivial 
for $i\le q$, then the resonance varieties of $F_{\bm}$ propagate 
in that range; that is, 
$\RR^1_1(F_{\bm})\subseteq\cdots\subseteq \RR^{q}_1(F_{\bm})$.

\item \label{cjm4}
If the monodromy action on $H_i(F_{\bm};\Q)$ is trivial 
for $i\le r-2$, then the resonance varieties of $F_{\bm}$ propagate; that is, 
$\RR^1_1(F_{\bm})\subseteq\cdots\subseteq \RR^{r-1}_1(F_{\bm})$.
\end{enumerate}
\end{corollary}

\begin{proof}
Claim \eqref{cjm1} follows from Theorem \ref{thm:cv-prop}  and 
Theorem \ref{thm:mf-abel-duality}, part \eqref{mf-ab}, while 
Claim \eqref{cjm2} follows from Theorem \ref{thm:cw-prop}  and 
Theorem \ref{thm:mf-abel-duality}, part \eqref{mf-abf}.

Claims \eqref{cjm3} and \eqref{cjm4} follow from claim \eqref{cjm2} 
and the Tangent Cone theorem, using Corollary \ref{cor:mf-formal}, 
parts \eqref{ff1} and \eqref{ff2}, respectively.
\end{proof}

In particular, if $\A$ is a central, essential arrangement in $\C^3$ and 
the monodromy action on $H_1(F_{\bm};\Q)$ is trivial, then  
$\WW^1_1(F_{\bm})\subseteq \WW^2_1(F_{\bm})$ and 
$\RR^1_1(F_{\bm})\subseteq \RR^{2}_1(F_{\bm})$.

\begin{remark}
\label{rem:affine}
More generally, let $f\in \C[z_0,\dots ,z_d]$ be a homogeneous 
polynomial of degree $n$, and set $M=\C^{d+1}\setminus \{f=0\}$. 
We then have a (global) Milnor fibration, $f\colon M\to \C^*$, with fiber 
$F=f^{-1}(1)$ and monodromy $h\colon F\to F$ given by 
$h(z)=e^{2\pi \ii/n} z$. Now suppose $M$ satisfies one of the 
conditions laid out in Remark \ref{rem:stein}, so that $M$ is 
an abelian duality space of dimension $d+1$, and suppose 
further that $h_*\colon H_1(F;\Q)\to H_1(F;\Q)$ is the identity. 
Then similar proofs show that $F$ is an 
$\abf$-duality space of dimension $d$ and the 
restricted characteristic varieties of $F$ propagate, that is,
$\WW^1_1(F)\subseteq \cdots\subseteq\WW^{d}_1(F)$.
\end{remark}

\section{Trivial algebraic monodromy and lower central series}
\label{sect:lcs}
In this section, we investigate the lower central series ranks and the Chen ranks 
of the fundamental groups of Milnor fibers of arrangements for which the algebraic 
monodromy is trivial.

\subsection{Lower central series and nilpotent quotients}
\label{subsec:lcs}
The lower central series (LCS) of a group $G$ is defined inductively 
by setting $\gamma_1 (G)=G$ and 
$\gamma_{k+1}(G) =[G,\gamma_k (G)]$ for all $k\ge 1$. 
This is a central series (i.e., 
$[G,\gamma_k(G)]\subseteq \gamma_{k+1}(G)$ for all $k\ge 1$), 
and thus, a normal series (i.e., $\gamma_k(G)\triangleleft G$ 
for all $k\ge 1$). Consequently, each LCS quotient,
\begin{equation}
\label{eq:grG}
\gr_k(G) \coloneqq \gamma_k(G)/\gamma_{k+1}(G),
\end{equation}
lies in the center of $G/\gamma_{k+1}(G)$, and thus is an abelian 
group. The first such quotient, $\gr_1(G)=G/\gamma_{2}(G)$, 
coincides with the abelianization $G_{\ab}=H_1(G;\Z)$. 
The  associated graded Lie algebra of $G$ is the direct 
sum $\gr(G)=\bigoplus_{k\ge 1} \gr_k(G)$; the addition in $\gr(G)$ 
is induced from the group multiplication, 
while the Lie bracket (which is compatible with the grading) 
is induced from the group commutator. 
By construction, the Lie algebra $\gr(G)$ is generated by 
its degree $1$ piece. Thus, if $G_{\ab}$ is finitely generated, 
then so are the LCS quotients of $G$;  we let 
$\phi_k(G)\coloneqq \rank \gr_{k}(G)$ be ranks of those quotients. 

Replacing in this construction the group $G$ by its maximal 
metabelian quotient, $G/G''$, leads to the Chen Lie algebra  
$\gr(G/G'')$, and, in the case when $G_{\ab}$ is finitely generated, 
the Chen ranks $\theta_k(G)\coloneqq \rank \gr_k(G/G'')$. It is readily seen 
that $\theta_k(G)\le \phi_k(G)$ for all $k\ge 1$, with equality for $k\le 3$. 

For each $k\ge 1$, the group $G/\gamma_{k+1}(G)$ is nilpotent, 
and in fact, the maximal $k$-step nilpotent quotient of $G$.  
Letting $q_{k} \colon G/\gamma_{k+1}(G) \to G/\gamma_{k}(G)$ 
be the projection maps, we obtain a tower of nilpotent groups,    
starting at $G/\gamma_{2}(G)=G_{\ab}$. Moreover, at each 
stage in the tower, there is a central extension, 
\begin{equation}
\label{eq:extension}
\begin{tikzcd}[column sep=20pt]
0\ar[r]
	&\gr_{k}(G) \ar[r]
	& G/\gamma_{k+1}(G) \ar[r, "q_k"]
	& G/\gamma_k(G)\ar[r]
	& 0\, ,
\end{tikzcd}
\end{equation}
which is classified by an extension class (or, $k$-invariant), $\chi_k \colon 
H_2(  G/\gamma_k(G);\Z )\to \gr_{k}(G)$.

\subsection{Lower central series of arrangement groups}
\label{subsec:lcs-arr}

The LCS ranks, the Chen ranks, and the nilpotent quotients of arrangement 
groups have been much studied. The most basic example is 
that of the free group, $F_n=\pi_1(\C\mysetminus \{\text{$n$ points}\})$, 
of rank $n\ge 2$. 
Work of P.~Hall, W.~Magnus, and E.~Witt from the 1930s shows that, for each 
$k\ge 1$, the abelian group $\gr_k(F_n)$ is torsion-free, of rank equal to 
\begin{equation}
\label{eq:lcs-free}
\phi_k(F_n)=\tfrac{1}{k}\sum_{d\mid k} \mu(d) n^{k/d},
\end{equation}
where $\mu\colon \N\to \{0,\pm 1\}$ denotes the M\"{o}bius 
function. Furthermore, work of K.T.~Chen from 1951 shows 
that the group $\gr_k(F_n/F_n'')$ are also torsion-free, of rank 
equal to 
\begin{equation}
\label{eq:chen-free}
\theta_k(F_n)=(k-1) \binom{n+k-2}{k}\: \text{ for $k\ge 2$}. 
\end{equation}

Now let $M=M(\A)$ be any arrangement complement, and let $G=\pi_1(M)$ 
be its fundamental group. As mentioned previously, $M$ is formal, and hence 
the group is $G$ is $1$-formal. Classical results of Quillen and Sullivan in 
rational homotopy theory insure that the LCS ranks $\phi_k(G)$ are determined 
by the (truncated) cohomology algebra $H^{\le 2}(M;\Q)$. Since this algebra is 
determined by the (truncated) intersection lattice $L_{\le 2}(\A)$, it follows that 
the LCS ranks of $G$ are combinatorially determined. Explicit combinatorial formulas for 
these ranks are known in a few cases, e.g., when $\A$ is either supersolvable \cite{FR}  
or decomposable \cite{PS-cmh06}, but no such formula is known in general, 
even for $\phi_3(G)$. As shown in \cite{PS-imrn04}, the Chen ranks $\theta_k(G)$ 
are also combinatorially determined. An explicit combinatorial formula was 
conjectured in \cite{Su01}, expressing those ranks in terms of the dimensions 
of the irreducible components of $\RR^1_1(M)$, at least for $k$ large enough.
This formula has been verified by Cohen and Schenck in \cite{CSc-adv} (see also 
\cite{AFRS} for a more general setting). 

Turning to the nilpotent quotients of an arrangement group $G=G(\A)$, 
it was shown in \cite{PrS20} that all the quotients $G/\gamma_k(G)$ are 
combinatorially determined when $\A$ is decomposable (see Section 
\ref{subsec:decomp} below for more on this). On the other 
hand, Rybnikov \cite{Ryb} showed that the third nilpotent quotient, 
$G/\gamma_4(G)$, is {\em not}\/ combinatorially determined, in general. 
Nevertheless, the second nilpotent quotient, $G/\gamma_3(G)$, 
is always determined by $L_{\le 2}(\A)$. To see why, recall 
from Section \ref{subsec:OS} that $H^*(M;\Z)=E/I$, where $E=\bigwedge G_{\ab}$ 
and $I=I(\A)$ is the Orlik--Solomon ideal associated to $L(\A)$. 
As shown in \cite[Prop.~1.14]{MS00}, the abelian group 
$\gr_2(G)$ is the $\Z$-dual of $I^2$ (and thus, it is torsion-free), 
and the exact sequence \eqref{eq:extension} with $k=2$ 
is classified by the homomorphism 
$\chi_2 \colon H_2( G_{\ab};\Z )\to \gr_2(G)$ 
dual to the inclusion map $I^2\inj  E^2$. Set $n=\abs{\A}$ 
and let $F_n$ be the free group on generators $\{x_H : H\in \A\}$. 
It follows that $G/\gamma_3(G)$ is the quotient of the free, 
$2$-step nilpotent group $F_n/\gamma_3(F_n)$ by the normal subgroup 
generated by all commutation relations of the form 
\begin{equation}
\label{eq:nilp2-arr}
\Big[x_H , \prod_{\substack{K\in \A\\[2pt] K\supset X}} x_{K}\Big] \, ,
\end{equation}
indexed by all pairs of hyperplanes $H\in \A$ and flats $X\in L_2(\A)$ 
such that $H\supset X$.  
From this description, it is apparent that the second nilpotent quotient 
of an arrangement group is combinatorially determined; that is, if 
$L_2(\A)\cong L_2(\B)$, then $G(\A)/\gamma_3(G(\A)) \cong 
G(\B)/\gamma_3(G(\B))$.

\subsection{LCS and Chen ranks of Milnor fibers}
\label{subsec:lcs-trivial-mono}
Let $(\A,\bm)$ be a multi-arrangement, with complement $M=M(\A)$. 
Let $F_{\bm}=F_{\bm}(\A)$  be the Milnor fiber
and let $h\colon F_{\bm}\to F_{\bm}$ be the monodromy of the 
corresponding Milnor fibration. 

Denoting by $G=\pi_1(M)$ and $K=\pi_1(F_{\bm})$ the fundamental 
groups of the respective spaces, we have a (split) exact 
sequence, $1\to K\to G\to \Z\to 1$, so that the arrangement group 
splits as the semidirect product $G=K\rtimes_{\varphi} \Z$, where 
$\varphi=h_{\sharp}\in \Aut(K)$ is the automorphism of $K=\pi_1(F_{\bm})$ 
induced by $h$. Note that $\varphi_{\ab}\colon K_{\ab}\to K_{\ab}$ may 
be identified with the (integral) algebraic monodromy, 
$h_*\colon H_1(F_{\bm};\Z)\to H_1(F_{\bm};\Z)$.

\begin{theorem}
\label{thm:trivial-mono-z}
Suppose $h_*\colon H_1(F_{\bm};\Z)\to H_1(F_{\bm};\Z)$ 
is the identity map. We then have the following isomorphisms of graded 
Lie algebras.
\begin{enumerate}[itemsep=2.5pt, topsep=-1pt]
\item \label{z1}
$\gr(G) \cong \gr(K)\rtimes_{\bar{\varphi}} \Z$, where 
$\bar\varphi\colon \Z \to \Der(\gr(K))$ is the morphism of 
Lie algebras induced by the homomorphism 
$\varphi\colon \Z\to \Aut(K)$ sending $1$ to $h_{\sharp}$.
\item \label{z2}
$\gr_{\ge 2}(K) \cong \gr_{\ge 2}(G)$.
\item \label{z3}
$\gr_{\ge 2}(K/K'')  \cong  \gr_{\ge 2}(G/G'')$. 
\end{enumerate}
\end{theorem}

\begin{proof}
Part \eqref{z1} follows from a well-known result of Falk and Randell \cite[Thm.~3.1]{FR}, 
as refined in \cite[Cor.~6.7]{Su-lcs-mono}. Part \eqref{z2} is a direct consequence of 
part \eqref{z1}. Finally, part \eqref{z3} follows from \cite[Cor.~8.10]{Su-abexact}.
\end{proof}

\begin{theorem}
\label{thm:trivial-mono-q}
Suppose $h_*\colon H_1(F_{\bm};\Q)\to H_1(F_{\bm};\Q)$ is the identity map.  
We then have the following isomorphisms of graded Lie algebras.
\begin{enumerate}[itemsep=2.5pt, topsep=-1pt]
\item \label{q1}
$\gr(G)\otimes \Q \cong (\gr(K) \rtimes_{\bar{\varphi}} \Z) \otimes \Q$. 
\item  \label{q2}
$\gr_{\ge 2}(K)\otimes \Q \cong \gr_{\ge 2}(G) \otimes \Q$.
\item  \label{q3}
$\gr_{\ge 2}(K/K'') \otimes \Q  \cong  \gr_{\ge 2}(G/G'') \otimes \Q$. 
\end{enumerate}
Consequently, $\phi_k(\pi_1(F_{\bm}))= \phi_k(\pi_1(M))$ and 
$\theta_k(\pi_1(F_{\bm}))= \theta_k(\pi_1(M))$ for all $k\ge 2$.
\end{theorem}

\begin{proof}
Parts \eqref{q1} and \eqref{q2} follow from Proposition 7.5 and 
Theorem 9.5 from \cite{Su-lcs-mono}, while 
part \eqref{q3} follows from \cite[Cor.~8.10]{Su-abexact}. 
The equality between the respective LCS and Chen ranks follows 
at once from parts \eqref{q2} and \eqref{q3}.
\end{proof}

Consequently, if the algebraic monodromy $h_*\colon H_1(F_{\bm};\Q)\to H_1(F_{\bm};\Q)$  
is trivial, then both the LCS ranks and the Chen ranks of $\pi_1(F_{\bm})$ are determined 
by $L_{\le 2}(\A)$. Moreover, letting $U=\P(M)$, we have that 
$\phi_k(\pi_1(F_{\bm}))= \phi_k(\pi_1(U))$ and 
$\theta_k(\pi_1(F_{\bm}))= \theta_k(\pi_1(U))$ for all $k\ge 1$.

\section{Constructions of arrangements with trivial algebraic monodromy}
\label{sect:constructions}

In this section, we describe several classes of hyperplane arrangements 
for which the Milnor fibration has trivial algebraic monodromy (in some range).

\subsection{Boolean arrangements}
\label{subsec:boolean}
Arguably the simplest kind of arrangements are the {\em Boolean 
arrangements}, $\B_n$, consisting of the coordinate hyperplanes 
$\set{z_i=0}$ in $\C^n$. The intersection lattice $L(\B_n)$ is the 
Boolean lattice of subsets of $\{0,1\}^n$, while the complement 
$M(\B_n)$ is the complex algebraic torus $(\C^*)^n$. 

Given a multiplicity function $\bm\colon \B_n\to \N$, the  
map $f_{\bm}\colon (\C^*)^n \to \C^*$, 
$z\mapsto z_1^{m_1}\cdots z_{n}^{m_n}$ is 
a morphism of complex algebraic groups. Hence, the 
Milnor fiber $F_{\bm}= \ker (f_{\bm})$ 
is an algebraic subgroup, realized as 
the disjoint union of $\gcd(\bm)$ copies of $(\C^*)^{n-1}$,   
with the monodromy automorphism,  
$h\colon F_{\bm}\to F_{\bm}$, permuting those 
copies in a circular fashion. 

Now suppose $\gcd(\bm)=1$. Then $F_{\bm}$ is  
an algebraic $(n-1)$-torus and $h$ is isotopic 
to the identity, through the isotopy 
$h_t(z) = e^{2\pi \ii t/N} z$.  Thus, the bundle 
$F_{\bm} \to M(\B_n) \to \C^*$ is trivial, 
and the algebraic monodromy,
$h_*\colon H_*(F_{\bm};\Z)\to H_*(F_{\bm};\Z)$,  
is equal to the identity map. Consequently, 
the characteristic polynomial of the algebraic monodromy 
is given by $\Delta_q(t)=(t-1)^{\binom{n-1}{q}}$ for $0<q<n$.

\subsection{Generic arrangements}
\label{subsec:generic}
Let $\A$ be a central arrangement of $n$ hyperplanes in $\C^{d+1}$, 
where $n>d+1>2$. We say $\A$ is  {\em generic}\/ if the intersection of 
every subset of $d+1$ distinct hyperplanes is the origin, in which case, 
$\A$ is the cone over an affine, general position arrangement 
$\A'$ of $n-1$ hyperplanes in $\C^d$, see \cite{OR, OT}. 

By a classical result of Hattori (\cite[Thm.~1]{Ha}), the complement of 
$\A'$ is homotopy equivalent to the $d$-skeleton of the real, 
$(n-1)$-dimensional torus $T^{n-1}$.  Since $U(\A)\cong M(A')$, 
it follows that $\pi_1(U(\A))=\Z^{n-1}$ and $b_q(U(\A))=\binom{n-1}{q}$ 
for $q\le d$. Moreover, if $\rho\colon \pi_1(U(\A))\to \C^*$ is a non-trivial 
character, then \cite[Thm.~4]{Ha} insures that $H_q(U(\A);\C_{\rho})=0$ 
for $q\ne d$ and $\dim_{\C} H_d(U(\A);\C_{\rho})=\binom{n-2}{d}$. 
It follows that the characteristic varieties of $U(\A)$ are given by 
\begin{equation}
\label{eq:cv-generic}
\V^q_s(U(\A))=\begin{cases} 
\{\bo\} &\text{for $q<d$ and $1\le s \le \binom{n-1}{q}$},
\\[3pt]
\C^{n-1} &\text{for $q=d$ and $1\le s\le \binom{n-2}{d}$}
\end{cases}
\end{equation}
and are empty otherwise.

Now let $\bm\colon \A\to \N$ be a choice of multiplicities, and let $F_{\bm}$ be 
the corresponding Milnor fiber.  Applying formula \eqref{eq:betti-cover}, we find that
$b_q(F_{\bm}) = \binom{n-1}{q} $ for $q \le d-1$ and 
$b_d(F_{\bm}) =  \binom{n-1}{d}+(n-1)\binom{n-2}{d }$.
Consequently, the algebraic monodromy $h_q\colon H_q(F_{\bm};\Q)\to H_q(F_{\bm};\Q)$ 
is equal to the identity if $q<d$, and the characteristic polynomial of $h_q$ 
takes the form
\begin{equation}
\label{eq:delta-generic}
\Delta_q(t)=\begin{cases} 
(t-1)^{\binom{n-1}{q} }&\text{if $q \le d-1$},
\\[2pt]
(t-1)^{\binom{n-2}{d-1} } (t^n-1)^{\binom{n-2}{d}}&\text{if $q =d$}.
\end{cases}
\end{equation}
In the case when $F_{\bm}=F(\A)$ is the usual Milnor fiber, this recovers a 
result of Orlik and Randell \cite{OR} (see also \cite{OT, CS95}).

\subsection{Decomposable arrangements}
\label{subsec:decomp}
Recall from Section \ref{subsec:pi1} that every flat $X\in L_2(\A)$ gives 
rise to a ``localized" sub-arrangement, $\A_{X}$, which consists of all hyperplanes 
$H\in \A$ that contain $X$. Furthermore, the inclusions $\A_{X}\subset \A$ 
yield inclusions of complements, $j_X\colon M(\A) \inj M(\A_{X})$, which 
assemble into a map 
\begin{equation}
\label{eq:jmap}
\begin{tikzcd}[column sep=18pt]
j=(j_X) \colon M \ar[r]& \prod_{X\in L_2(\A)} M(\A_X).
\end{tikzcd}
\end{equation}

Let $j_{\sharp}\colon G(\A) \to \prod_{X\in L_2(\A)} G(\A_X)$ 
be the induced homomorphism on fundamental groups.
It was shown in \cite{Fa89, PS-cmh06} that the morphism 
\begin{equation}
\label{eq:gr-jmap}
\begin{tikzcd}[column sep=18pt]
\gr(j_{\sharp})\colon 
\gr(G(\A)) \ar[r]& \prod_{X\in L_2(\A)}  \gr(G(\A_X))
\end{tikzcd}
\end{equation}
between the respective associated 
graded Lie algebras is an isomorphism in degree $2$ and, after tensoring with $\Q$, 
becomes surjective in all degrees greater than $2$. Since each of the groups 
$G(\A_X)$ is isomorphic to $F_{\mu(X)} \times \Z$, it follows that the LCS ranks of $G(\A)$ 
admit the lower bounds 
\begin{equation}
\label{eq:phis-min}
\phi_k(G(\A))\ge \sum_{X\in L_2(\A)}\phi_k(F_{\mu(X)}) 
\end{equation}
for all $k\ge 2$, with equality for $k=2$.

Following \cite{PS-cmh06}, we say that a hyperplane arrangement $\A$ is 
{\em decomposable}\/ (over $\Q$) if the third LCS rank of the group $G(\A)$ 
attains the lower bound from \eqref{eq:phis-min}; that is,
\begin{equation}
\label{eq:phi3-min}
\phi_3(G(\A))=\sum_{X\in L_2(\A)} \binom{\mu(X)}{2} .
\end{equation}

It is shown in \cite{PS-cmh06} that once this condition is satisfied, 
equality is attained in \eqref{eq:phis-min} for all $k\ge 2$; in fact, 
the morphism  $\gr(j_{\sharp})\otimes \Q$ restricts to an isomorphism 
of graded Lie algebras in degrees $\ge 2$. 

More generally, let $\h(\A)$ be the holonomy Lie algebra of $\A$, that 
is, the quotient of the free Lie algebra on generators $\{x_H : H\in \A\}$ 
by the ideal generated by the Lie brackets of the form 
\begin{equation}
\label{eq:holo}
\Big[x_H , \sum_{\substack{K\in \A\\[2pt] K\supset X}} x_{K}\Big] \, ,
\end{equation}
for all hyperplanes $H\in \A$ and $2$-flats $X\in L_2(\A)$ such that $H\supset X$. 
There is then an epimorphism $\h(\A)\surj \gr(G(\A))$ that becomes an 
isomorphism upon tensoring with $\Q$ (due to the $1$-formality of the 
arrangement group). The arrangement $\A$ is said to be 
decomposable over $\k$ (where $\k$ is either $\Z$ or a field) if 
$\h_3(\A) \otimes \k$ decomposes as the direct sum 
$\bigoplus_{X\in L_2(\A)} \h_3(\A_X) \otimes \k$. 
It is shown in \cite{PS-cmh06} that once this condition is satisfied, 
a similar decomposition holds in all degrees $k\ge 2$.
Furthermore, the following is shown in \cite[Thm.~8.8]{PrS20}: 
If $\A$ is decomposable over $\Z$, then all the nilpotent quotients 
$G(\A)/\gamma_k(G(\A))$ are determined by $L_{\le 2}(\A)$. 
The same proof works if $\A$ is decomposable over $\Q$, 
with the nilpotent quotients replaced by their rationalizations. 

Let $B(\A)=G(\A)'/G(\A)''$ be the Alexander invariant of an arrangement $\A$, 
viewed as module over the group ring $\Z[G(\A)_{\ab}]$, and endowed with 
the filtration by the powers of the augmentation ideal. 
An in-depth study of the Alexander invariant and of the 
Milnor fibrations of a decomposable arrangement is 
done in \cite{Su-decomp}. We record in the next theorem 
one of the main results of this study.

\begin{theorem}[\cite{Su-decomp}]
\label{thm:decomp-mono}
Let $\A$ be an arrangement of rank $3$ or higher. Suppose 
$\A$ is decomposable over $\Q$ and $B(\A)\otimes \Q$ is 
separated in the $I$-adic topology. Then, for any choice of multiplicities 
$\bm\colon \A\to \N$, the algebraic monodromy of the Milnor fibration, 
$h_*\colon H_1(F_{\bm};\Q)\to H_1(F_{\bm};\Q)$, is trivial.
\end{theorem}

A large supply of decomposable arrangements may be constructed by taking 
suitable sections of products of (central) arrangements in $\C^2$. For such an 
arrangement $\A$, the group $G(\A)$ is a finite direct product of finitely generated 
free groups (see \cite{CDP} for a detailed study of such arrangements). We shall 
encounter two concrete examples of arrangements from this class in Section \ref{sect:mf-falk}. 

In general, though, there are decomposable arrangements for which the arrangement 
group is much more complicated. For instance, let $\A$ be the arrangement in $\C^3$ 
defined by the polynomial $f=xyz(x+y)(x-z)(2z+y)$. It is readily checked that $\A$ 
is decomposable (over $\Z$). Nevertheless, the group $G(\A)$ does not 
even have a finite-dimensional classifying space $K(G(\A),1)$, see \cite[Rem.~12.4]{Su-imrn}.

\subsection{Multiplicity conditions}
\label{subsec:multiplicities}
If $F=F(\A)$ is the Milnor fiber of a central arrangement $\A$ in $\C^{d+1}$, $d>1$, 
there are various combinatorial conditions insuring that the algebraic monodromy 
$h_*\colon H_1(F;\k)\to H_1(F;\k)$ over $\k=\Z$ or $\k$ a field is the identity, such 
as the ones given in \cite{CDO03,Li02,Wi,WS,LX}. 

In \cite{Wi}, Williams gave a very nice combinatorial upper bound on the 
first Betti number of $F$ and a criterion for triviality of the algebraic monodromy 
over $\Z$, stated in the case when $\A$ is the complexification of a real arrangement. 
A partial generalization was obtained in \cite{WS}, and the result was recently proved 
by Liu and Xie \cite{LX} in full generality. We summarize these results, as follows. 

\begin{theorem}[\cite{Wi, WS, LX}]
\label{thm:b1f-bound}
Let $\A$ be a central arrangement of $n$ hyperplanes. For each 
hyperplane $H\in \A$, set
\[
s_H=\sum_{\substack{X\in L_2(\A) \\[1pt] X\subset H}} (q_X-2)(\gcd(q_X,n)-1), 
\]
where $q_X=\abs{\A_X}$. Then,
\begin{enumerate}[itemsep=2pt, topsep=-1pt]
\item \label{w1} 
$\dim_{\k} H_1(F;\k) \le n-1 + \min \big\{ s_H : H\in \A \big\}$, for all fields $\k$.

\item \label{w2} 
$\Delta_1(t)=(t-1)^{n-1}p(t)$, 
for some $p(t)\in \C[t]$ dividing the polynomials 
\[
\left(\frac{t^{\gcd(q_X,n)}-1}{t-1}\right)^{q_X-2}
\]
for all $X\in L_2(\A)$. 

\item \label{w3} 
If there is a hyperplane $H\in \A$ such that $\gcd(q_X,n)=1$ for all $2$-flats $X$ 
with $q_X>2$ (for instance, if $n$ is a prime), then $H_1(F;\Z)=\Z^{n-1}$.
\end{enumerate}
\end{theorem}

\subsection{The double point graph}
\label{subsec:node-graph}
Let $\A$ be a central arrangement of planes in $\C^3$, and let $\bar{\A}=\P(\A)$ 
be the corresponding arrangement of projective lines in $\CP^2$. 
The {\em double point graph}\/ associated to $\A$ is the graph $\Gamma$ 
with vertex set $\A$ and with an edge joining two hyperplanes $H,K \in \A$ 
if $\bar{H} \cap \bar{K}$ is a double point (see \cite{Bt14, SS17}). 
The components of $\Gamma$ define a partition of $\A$ 
which is a refinement of all partitions induced by multinets on $\A$. 

Now suppose $\Gamma$ is connected. Using results from \cite{PS-plms17}, 
Bailet showed in \cite{Bt14} that the algebraic monodromy of the 
Milnor fibration, $h_*\colon H_1(F;\C)\to H_1(F;\C)$, is the 
identity map, provided $\abs{\A_X}\le 9$ for all $X\in L_2(\A)$ 
and  either $6\nmid \abs{\A}$, or there exists a hyperplane $H\in A$ 
such that $\abs{\A_X}\ne 6$, for all $X\subset H$. 
Under the same connectivity assumption on $\Gamma$, 
Salvetti and Serventi \cite{SS17} show that $\A$ admits no multinet. 
Furthermore, they show that $h_*=\id$ if $\Gamma$ 
admits a ``good" spanning tree, and conjecture that this holds for arbitrary connected graphs. 
In \cite{Ve} Venturelli establishes this conjecture under the assumption that 
$\bar{\A}$ has two multiple points, $P_1$ and $P_2$, such that every line in $\bar{\A}$ 
passes through either $P_1$ or $P_2$; in \cite{Su-decomp}, we give another proof 
of this result, in a more general setting. 

\section{The Falk arrangements}
\label{sect:mf-falk}

\subsection{A pair of arrangements and their complements}
\label{subsec:falk-arrs}
In this section, we analyze in detail a pair of hyperplane arrangements 
introduced by Falk in \cite{Fa93} and further studied in \cite{Su-revroum}.  
The two arrangements, $\A$ and $\hat\A$, are central arrangements of 
$6$ planes in $\C^3$, defined by the polynomials 
\begin{gather}
\label{eq:Falk-arrs}
\begin{aligned}
f&= z(x-y)y(x+y)(x-z)(x+z),\\
\hat{f}&=z(x+z)(x-z)(y+z)(y-z)(x-y+z).
\end{aligned}
\end{gather}

\begin{figure}
 \centering
 \begin{minipage}[t]{.4\textwidth}
  \centering
\begin{tikzpicture}[scale=0.75]
\draw[style=thick] (0,0) circle (3.2);
\draw[style=thick](-1,-2.1) -- (-1,2.5);
\draw[style=thick] (1,-2.1) -- (1,2.5);
\draw[style=thick] (-2.5,0) -- (2.5,0);
\draw[style=thick]  (-2,-2) -- (2,2);
\draw[style=thick](-2,2) -- (2,-2);
\node at (-3.5,0.5) {$\ell_0$};
\node at (2.1,1.55) {$\ell_1$};
\node at (2.85,0) {$\ell_2$};
\node at (2.1,-1.5) {$\ell_3$};
\node at (1,-2.5) {$\ell_4$};
\node at (-1,-2.5) {$\ell_5$};
\end{tikzpicture}
 \end{minipage}%
 \quad
  \begin{minipage}[t]{.4\textwidth}
   \centering
  \begin{tikzpicture}[scale=0.75]
\draw[style=thick] (0,0) circle (3.2);
\draw[style=thick](-1,-2.1) -- (-1,2.5);
\draw[style=thick] (1,-2.1) -- (1,2.5);
\draw[style=thick] (-2.5,1) -- (2.5,1);
\draw[style=thick] (-2.5,-1) -- (2.5,-1);
\draw[style=thick]  (-2.4,-1.5) -- (1.5,2.4);
\node at (-3.5,0.5) {$\hat\ell_0$};
\node at (2.1,1.45) {$\hat\ell_4$};
\node at (-0.15,0.15) {$\hat\ell_5$};
\node at (2.1,-0.55) {$\hat\ell_3$};
\node at (1,-2.5) {$\hat\ell_2$};
\node at (-1,-2.5) {$\hat\ell_1$};
\end{tikzpicture}
 \end{minipage}
\caption{The Falk arrangements $\A$ and $\hat\A$}
\label{fig:falk}
\end{figure}
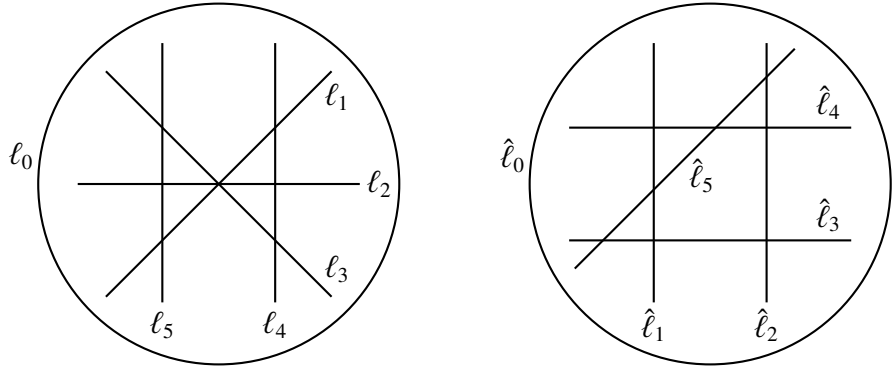

The projectivizations of $\A$ and $\hat\A$ are depicted in Figure \ref{fig:falk}; 
the numbering of the lines corresponds to the ordering of the linear factors in the 
respective defining polynomials. 
Both $\P(\A)$ and $\P(\hat\A)$ have $2$ triple points and $9$ double points, 
yet the two intersection lattices are non-isomorphic:  the two triple 
points of $\P(\A)$ do not lie on a common line, whereas the triple points 
of $\P(\hat{\A})$ lie on a common line (namely, $\hat\ell_0$). 
Nevertheless, as shown by Falk in \cite{Fa93}, the two projective 
complements, $U=\P(M)$ and $\hat{U}=\P(\hat{M})$, are homotopy equivalent. 
Let us note that $\P(\hat\A)$ has a line (namely, $\hat\ell_5$) in general position 
with the others. A well-known result of Oka and Sakamoto \cite{OkaS} then 
implies that $\pi_1(\hat{U})$ splits off a $\Z$ factor; 
it easily follows that both groups are isomorphic to $F_2\times F_2\times \Z$. 

The cohomology rings $A=H^*(U;\Z)$ and $\hat{A}=H^*(\hat{U};\Z)$ are the quotients 
of the exterior algebra $E=\bwedge(e_1,\dots,e_5)$ by the ideals 
$I=(\partial e_{123}, e_{45})$ and $\hat{I}=(e_{12}, e_{34})$, 
respectively. The automorphism $E\isom E$ given by 
$e_1\mapsto e_1-e_3$, $e_2\mapsto e_2-e_3$, 
$e_3\mapsto e_4$, $e_3\mapsto e_5$, and 
$e_5\mapsto e_1+e_2+e_3$ 
induces an isomorphism $\hat{A}\isom A$. 
It is readily verified that the only multinets supported on sub-arrangements 
of either $\A$ or $\hat{\A}$ are those coming from the triple points, and that 
the respective characteristic varieties are given by
\begin{gather}
\label{eq:Falk-cvs}
\begin{aligned}
\VV^1_1(U)&=\{t_1t_2t_3=t_4=t_5=1\} \cup \{  t_1=t_2=t_3=1\}, 
\\
\VV^1_1(\hat{U})&= \{  t_3=t_4=t_5=1\} \cup \{  t_1=t_2=t_5=1\}.
\end{aligned}
\end{gather}

\subsection{The Milnor fibers of the Falk arrangements}
\label{subsec:falk-mf}

Let $F=F(\A)$ and $\hat{F}=F(\hat\A)$ be the fibers of the Milnor fibrations 
$f\colon M\to \C^*$ and $\hat{f}\colon \hat{M}\to \C^*$. 
Since both $\P(\A)$ and $\P(\hat\A)$ have only double and triple points, 
and since neither of the two arrangements supports a $3$-net,  
Theorem \ref{thm:ps-triple} shows that the characteristic polynomial 
of the algebraic monodromy acting on either $H_1(F;\Q)$ or 
$H_1(\hat{F};\Q)$ is $(t-1)^5$. Alternatively, it is easily verified   
that both arrangements are decomposable (over $\Z$); therefore, 
Theorem \ref{thm:decomp-mono} shows once again that their  
algebraic monodromy is trivial in degree $1$. It now follows 
from Corollary \ref{cor:mf-formal} that both $F$ and $\hat{F}$ 
are formal spaces.  

Since $\hat\A$ contains a line meeting the other ones only in double points, 
Theorem \ref{thm:b1f-bound}, part \eqref{w3} implies that $H_1(\hat F;\Z)=\Z^5$. 
Direct computation shows that $H_1(F;\Z)=\Z^5$, too, and so the monodromy 
action on both these groups is trivial. 
Moreover, both Milnor fibers have Euler characteristic $6\cdot 4=24$, 
and thus $H_2(F;\Z)=H_2(\hat F;\Z)=\Z^{28}$.  Let $\zeta$ be a primitive $6$th 
root of unity, and let $\mathcal{H}_k$ be the $\zeta^k$-eigenspace of 
the monodromy action on $H_2(F;\C)$.  Then, by \cite{CS95}, 
we have that $\dim_{\C} \mathcal{H}_k =4$ for $1\le k\le 5$ and 
$\dim_{\C} \mathcal{H}_0 =8$. 

Let $K=\pi_1(F)$ and $\hat K=\pi_1(\hat F)$ be the fundamental groups 
of the two Milnor fibers and let $G=\pi_1(M)\cong\pi_1(\hat{M})$. 
Applying Theorem \ref{thm:trivial-mono-z}, we  
find that the associated graded Lie algebras, respectively, the Chen Lie algebras 
of all these groups are isomorphic in degrees $2$ and more:
\begin{gather}
\label{eq:grlie-mff}
\begin{aligned}
&\gr_{\ge 2} (K)\cong \gr_{\ge 2} (\hat K) \cong \gr_{\ge 2} (G),\\[1pt]
&\gr_{\ge 2} (K/K'')\cong \gr_{\ge 2} (\hat K/\hat K'')  \cong \gr_{\ge 2} (G/G'') .
\end{aligned}
\end{gather}
From the discussion in Section \ref{subsec:falk-arrs}, we have that $G\cong F_2^2\times \Z^2$. 
Therefore, all the LCS quotients and Chen groups of $K$ and $\hat K$ are torsion-free, 
with ranks in degrees $k\ge 2$ given by
\begin{gather}
\label{eq:Falk-lcs-chen}
\begin{aligned}
&\phi_k(K)= \phi_k (\hat K)= \tfrac{2}{k}\sum_{d\mid k} \mu(d) 2^{k/d}, 
\\
&\theta_k (K)=\theta_k (\hat K)  = 2(k-1) .
\end{aligned}
\end{gather}

Although all these homological and group-theoretic invariants of $F$ and $\hat F$ 
agree, the two Milnor fibers are {\em not}\/ homotopy equivalent, as the next result shows.

\begin{prop}
\label{prop:mf-falk}
Let $F$ and $\hat F$ be the Milnor fibers of the two Falk arrangements, and 
let $K$ and $\hat K$ be their fundamental groups. Then, 
\begin{enumerate}[itemsep=2.5pt, topsep=-1pt]
\item \label{mff1} 
$K/K'' \not\cong \hat K/\hat K''$.
\item \label{mff2} 
$K/\gamma_3(K)\not\cong \hat K/\gamma_3(\hat K)$.
\end{enumerate}
Consequently, $\pi_1(F)\not\cong \pi_1(\hat F)$.  
\end{prop}

A proof of this proposition will be given in the next two subsections. 

\subsection{The characteristic varieties of $F$ and $\hat F$}
\label{subsec:falk-mf-cv}
The (degree $1$) characteristic varieties of the Milnor fibers of the 
two Falk arrangements were first computed in \cite{Su-revroum}. 
Since that computation was based on a machine calculation, we 
redo it here by hand, using a method which works 
for any arrangement with trivial algebraic monodromy in degree $1$.

We start with the Milnor fiber $F=F(\A)$. As remarked above, 
$H_1(F;\Z)=\Z^5$. The inclusion map $\iota\colon F\to M$ induces 
a morphism $\iota^*\colon H^1(M;\C^*)\to H^1(F;\C^*)$ on character 
tori, given in coordinates by 
\begin{equation}
\label{eq:iota-star}
\iota^*(z_1,\dots , z_6) = (z_1/z_2,z_2/z_3,z_3/z_4,z_4/z_5,z_5/z_6). 
\end{equation}
It follows from Theorem \ref{thm:cjl-mf-trivialmono}, part \eqref{pb}, 
that the characteristic variety $\VV^1_1(F)\subset H^1(F;\C^*)$ is the 
image under the map $\iota^*$ of $\VV^1_1(M)\subset H^1(M;\C^*)$. Therefore,
\begin{align*}
\VV^1_1(F)&=\iota^*\big(\big\{\big( z_1,z_2, \tfrac{1}{z_1z_2},1,1,1\big) \mid z_1,z_2\in \C^*\big\}\big) \cup 
\iota^*\big(\big\{(1,1,1, z_4,z_5, \tfrac{1}{z_4z_5}) \mid z_4,z_5\in \C^*\big\}\big)
\\
&= \big\{\big( \tfrac{z_1}{z_2},z_1z^2_2, \tfrac{1}{z_1z_2},1,1\big) \mid z_1,z_2\in \C^*\big\} \cup 
\big\{\big(1,1, \tfrac{1}{z_4}, \tfrac{z_4}{z_5},z_4z_5^2\big) \mid z_4,z_5\in \C^*\big\}, 
\end{align*}
and so $\VV^1_1(F)\subset (\C^*)^5$ is the union of two $2$-dimensional subtori, 
$T_1=\{ u \in (\C^*)^5 \mid  u_1u_2^2u_3^3 = u_4=u_5=1\}$ and  
$T_2=\{ u \in (\C^*)^5 \mid   u_1 = u_2 =u_3^3 u_4^2 u_5 =1 \}$. 
Notice that 
\begin{equation}
\label{eq:v12-f}
T_1\cap T_2 = \{\bo , (1,1,\omega, 1, 1), (1,1,\omega^{2}, 1, 1)\},
\end{equation}
where $\omega=\exp(2\pi \ii /3)$. By Theorem \ref{thm:acm}, the torsion 
characters comprising $T_1\cap T_2$ lie in $\VV^1_2(F)$. In fact, direct 
computation reveals that $\VV^1_2(F)=T_1\cap T_2$. 

Proceeding in the same manner with the Milnor fiber of the second Falk arrangement, 
$\hat F=F(\hat\A)$, we obtain: 
\begin{align*}
\VV^1_1(\hat F)&=\iota^*\big(\big\{\big( z_1,z_2,1,1,1, \tfrac{1}{z_1z_2}\big) \mid z_1,z_2\in \C^*\big\}\big) \cup 
 \iota^*\big(\big\{(1,1, z_3,z_4,1, \tfrac{1}{z_3z_4}) \mid z_3,z_4\in \C^*\big\}\big)
\\
&= \big\{\big( \tfrac{z_1}{z_2},z_2,1,1,z_1z_2\big) \mid z_1,z_2\in \C^*\big\} \cup 
\big\{\big(1, \tfrac{1}{z_3}, \tfrac{z_3}{z_4},z_4,z_3z_4\big) \mid z_3,z_4\in \C^*\big\},
\end{align*}
and so $\VV^1_1(\hat F)=\hat T_1\cup \hat T_2$, where 
$\hat T_1=\{ u \in (\C^*)^5 \mid  u_1 u_2^2 u_5^{-1} = u_3 =u_4=1 \}$ and 
$\hat T_2=\{ u \in (\C^*)^5 \mid  u_1=u_2 u_3 u_4 = u_2 u_4^{-1}u_5 = 1\}$. 
Notice that these two subtori intersect only at the origin; in fact, 
direct computation shows that 
\begin{equation}
\label{eq:v12-f-prime}
\VV^1_2(\hat F)=\hat T_1\cap \hat T_2 =\{\bo\}.
\end{equation}

The above computations show that $\VV^1_2(F)\not\cong \VV^1_2(\hat F)$:  
the first variety consists of $3$ points, while the second consists of a single point. 
Finally, recall from Section \ref{subsec:cv} that the characteristic varieties $\VV^1_s(G)$ of a 
(finitely generated) group $G$ depend only on its maximal metabelian quotient, 
$G/G''$. Therefore, we have shown that $K/K'' \not\cong \hat K/{\hat K}''$, 
thereby completing the proof of part \eqref{mff1} of Proposition \ref{prop:mf-falk}.\hfill\qed

\begin{remark}
\label{rem:res-mff}
Since both Milnor fibers are formal, the tangent cones to their first 
characteristic varieties coincide with the first resonance varieties. 
Using either this observation, together with the computations 
from above, or Theorem \ref{thm:cjl-mf-trivialmono}, 
part \eqref{res-f-bis}, we find that
\begin{align*}
\RR^1_1(F)&=\{x_1+2x_2+3x_3=x_4=x_5=0 \} \cup 
 \{x_1= x_2=3x_3+2x_4+x_5=0\},
\\
\RR^1_1(\hat F)&=\{x_1+2x_2-x_5=x_3=x_4=0 \}\cup 
\{x_1=x_2+x_3+x_4=x_2-x_4+x_5=0 \},
\end{align*}
while $\RR^1_2(F)= \RR^1_2(\hat F) =\{\bz\}$.  Thus, the resonance varieties 
do not distinguish between $\pi_1(F)$ and $\pi_1(\hat F)$.
\end{remark}
 
\subsection{The second nilpotent quotients of $K$ and $\hat K$}
\label{subsec:falk-mf-nilp}
We now give a proof of Proposition \ref{prop:mf-falk}, part \eqref{mff2}. 
First consider the projectivized complement $U=U(\A)$ and its fundamental group, 
$\overline{G}=\pi_1(U)$. Recall that $H^*(U;\Z)=E^*/I^*$, where 
$E=\bwedge (e_1,\dots, e_5)$ and $I=(\partial e_{123},e_{45})$. 
Writing $E_r=(E^r)^{\vee}$ and $I_r=(I^r)^{\vee}$ for the $\Z$-dual groups, 
the second nilpotent quotient $\overline{G}/\gamma_3(\overline{G})$ is the central 
extension of $\gr_1(\overline{G})=E_1\cong \Z^5$ by $\gr_2(\overline{G})=I_2\cong \Z^2$ 
classified by the cocycle $\chi_2\colon E_2 \surj I_2$ given by the matrix 
\begin{equation}
\label{eq:chi-2}
\chi^{\intercal}_2=
 \bordermatrix{
&{\scriptstyle 12} & {\scriptstyle 13} & {\scriptstyle 23} 
&{\scriptstyle 14} & {\scriptstyle 24} & {\scriptstyle 34} 
&{\scriptstyle 15} & {\scriptstyle 25} & {\scriptstyle 35} & {\scriptstyle 45} 
\cr
&1&-1&1&0&0&0&0&0&0&0
\cr
&0&0&0&0&0&0&0&0&0&1
}.
\end{equation} 

To compute the Schur multiplier $H_2(\overline{G}/\gamma_3(\overline{G});\Z)$, we use 
an approach similar to the one used in the proof of \cite[Thm.~4.1]{PrS20}. 
Consider the homology spectral sequence of the central extension 
$0\to I_2 \to \overline{G}/\gamma_3(\overline{G})  \xrightarrow{\ab} E_1 \to 0$, 
\begin{equation}
\label{eq:lhs-ss}
E^2_{p,q}=H_p( E_1 ;H_q( I_2 ;\Z)) \Rightarrow 
H_{p+q} ( \overline{G}/\gamma_3(\overline{G});\Z ).
\end{equation}
Since the $(E_2,d^2)$ page of the cohomology spectral sequence 
is a CDGA, and since its $\Z$-dual is $(E^2,d_2)$---due to lack of torsion on 
either of these two pages---the differentials $d_2\colon E^2_{p,q}\to E^2_{p-2,q+1}$ 
in diagram \eqref{eq:serre-ss} are determined by the map $d^2_{2,0}=\chi_2$.
\begin{equation}
\label{eq:serre-ss}
\begin{gathered}
\begin{sseqpage}[ homological Serre grading, 
xscale = 1.8, yscale = 1.5 ,  classes = { draw = none } ]
\class["\Z"](0,0)
\class["\Z^2"](0,1)
\class["\Z"](0,2)
\class["\Z^5"](1,0)
\class["\Z^{10}"](2,0)
\class["\Z^{10}"](3,0)
\class["\Z^{10}"](1,1)
\class["\Z^{20}"](2,1)
\d [->>, "d^{\,2}_{2,0}" { pos = 0.75,  xshift = -2pt, yshift = -25pt} ] 2 (2,0)
\d ["d^{\,2}_{3,0}" { pos = 0.5, xshift = 6pt, yshift = -8pt } ] 2 (3,0)
\d [->>, "d^{\,2}_{2,1}" ]  2 (2,1)
\end{sseqpage}
\end{gathered}
\end{equation}

Clearly, $E^3_{2,0}=\ker(d^2_{2,0})=\Z^8$.  The differential $d^2_{3,0}$ is 
dual to the composite $E^1\otimes I^2\inj E^1\otimes E^2\surj E^3$, 
whose kernel is generated by the elements 
$u_1= (e_1-e_2)\otimes  \partial e_{123}$, $u_2=(e_2-e_3)  \otimes \partial e_{123}$, 
$u_3=e_4 \otimes  e_{45}$,  and $u_4=e_5 \otimes e_{45}$. Taking transposes, we 
see that $E^3_{1,1}=\coker(d^2_{3,0})$ is equal to $\Z^4$, generated by the duals 
$u_i^{\vee}$ of those elements (written in terms of the duals $\varepsilon_i=e_i^{\vee}$). 
Finally, note that the map $d^2_{2,1}\colon E_2\otimes I_2 \to I_2\wedge I_2$ 
is surjective, since it sends $\partial \varepsilon_{123} \otimes \varepsilon_{45}$ 
to the generator $\partial \varepsilon_{123} \wedge \varepsilon_{45}$  
of $I_2\wedge I_2=\Z$; hence, $E^3_{0,2}=0$. 
Looking at the domains and ranges of the higher-order differentials 
in the spectral sequence, we see that $E^3_{p,q} = E^\infty_{p,q}$ 
for $p+q \le 2$.   Therefore,
\begin{equation}
\label{eq:h2-gg3}
 H_2(\overline{G}/\gamma_3(\overline{G});\Z)=
 E^3_{2,0}\oplus E^3_{1,1} =\Z^8\oplus \Z^{4} =\Z^{12}.
 \end{equation}

Consider next the Milnor fiber $F=F(\A)$. The inclusion map $\iota \colon F(\A)\inj M(\A)$ 
induces a monomorphism $\iota_*\colon H_1(F;\Z)\inj H_1(M;\Z)$ given in suitable bases by 
$\alpha_i\mapsto x_i-x_{i+1}$ for $1\le i\le 5$. Letting $\alpha_i=a_i^{\vee}$, the ring morphism 
$\sigma^*\colon H^*(U;\Z)\to H^*(F;\Z)$ is given in degree $1$  by 
$e_1 \mapsto a_1 + a_5$, 
$e_2 \mapsto -a_1 + a_2 + a_5$, 
$e_3 \mapsto -a_2 +a_3 + a_5$, 
$e_4 \mapsto -a_3 +a_4 + a_5$, 
$e_5 \mapsto -a_4 + 2 a_5$. 
It follows that the group $J^2\coloneqq \sigma^*(I^2)$ is free abelian, with basis 
$\sigma^*(\partial e_{123})=3a_{12}-2a_{13}+a_{23}$ 
and $\sigma^*(e_{45})=a_{34}-2a_{35}+3a_{45}$.

The second nilpotent quotient of the group $K=\pi_1(F)$ fits into the central 
extension $0\to J_2 \to K/\gamma_3(K) \xrightarrow{\ab} H \to 0$, where 
$H=K_{\ab}\cong \Z^5$ and $J_2=(J^2)^{\vee}\cong \Z^2$. Furthermore, 
the extension is classified by the cocycle $\chi_2\colon \bwedge^2 H \surj J_2$ 
given by the matrix
\begin{equation}
\label{eq:chi-2-f}
\chi^{\intercal}_2=
 \bordermatrix{
&{\scriptstyle 12} & {\scriptstyle 13} & {\scriptstyle 23} 
&{\scriptstyle 14} & {\scriptstyle 24} & {\scriptstyle 34} 
&{\scriptstyle 15} & {\scriptstyle 25} & {\scriptstyle 35} & {\scriptstyle 45} 
\cr
&3&-2&1&0&0&0&0&0&0&0
\cr
&0&0&0&0&0&1&0&0&-2&3
}.
\end{equation} 

The spectral sequence of the extension has the same entries in the $E^2$ page as in 
display \eqref{eq:serre-ss}. The differentials $d^2_{2,0}$ and $d^2_{2,1}$ are still 
surjective, giving $E^3_{2,0}=\Z^8$ and $E^3_{0,2}=0$. The difference, though, 
lies with the differential $d^2_{3,0}$: the elements $\sigma^*(u_i)^{\vee}$ are still in 
$\coker(d^2_{3,0})$, generating a $\Z^4$-summand, but now there  is an extra 
element of order $3$ in  that cokernel,  namely, 
$a_4\otimes (3a_{12}-2a_{13}+a_{23})$. 
Therefore, $E^3_{1,1}=\Z^4\oplus \Z_3$. 
Proceeding as before, we find that $H_2(K/\gamma_3(K);\Z)=  \Z^{12}\oplus \Z_3$. 

For the group $\hat K=\pi_1(\hat F)$, an entirely similar computation shows that 
$\coker (d^2_{3,0})=\Z^4$, and hence 
$H_2(\hat K/\gamma_3(\hat K);\Z)=\Z^{12}$.
Therefore, $K/\gamma_3(K)\not\cong \hat K/\gamma_3(\hat K)$, 
thereby completing the proof of Proposition \ref{prop:mf-falk}, part \eqref{mff2}.\hfill\qed

\section{The $\operatorname{B}_3$ arrangement and its deletion}
\label{sect:delB3}

\subsection{The $\operatorname{B}_3$ arrangement}
\label{subsec:b3}
Let $\A$ be the rank-$3$ reflection arrangement of type $\operatorname{B}_3$, 
defined by the polynomial 
\begin{equation}
\label{eq:poly-b3}
f=xyz(x-y)(x+y)(x-z)(x+z)(y-z)(y+z).
\end{equation}
Figure \ref{fig:b3 arr} shows a plane section of $\A$. 
The $\operatorname{B}_3$ arrangement is of fiber-type, 
with exponents $\{1,3,5\}$. 
Thus, the complement $M=M(\A)$ is aspherical and its projectivization, 
$U=\P(M)$, has fundamental group which decomposes as a semidirect 
product of free groups, $\pi_1(U)=F_5\rtimes_{\alpha} F_3$. The braid 
monodromy algorithm from \cite{CS97} shows that the monodromy map 
$\alpha\colon F_3\to \Aut(F_5)$ takes values in the pure braid group $P_5$, 
viewed as a subgroup of $\Aut(F_5)$ via the Artin representation. Denoting 
by $u_i$ the generators of $F_3$ and by 
$A_{ij}$ the standard generators of the pure braid group (corresponding 
to the meridians around the hyperplanes $H_{ij}$ of the braid arrangement), 
the monodromy map is given by 
\begin{equation}
\label{eq:bmono-b3}
\alpha(u_1)=A_{23}A_{24}A_{34}, \: \  
\alpha(u_2)=A_{14}^{A_{24}A_{34}}A_{25},\: \ 
\alpha(u_3)=A_{35}^{A_{23}A_{25}},
\end{equation}
where $a^b=b^{-1}ab$, see \cite[Ex.~10.8]{Su02}. Since pure braid automorphisms 
act trivially in homology, the extension $1\to F_5\to \pi_1(U) \to F_3\to 1$ is $\ab$-exact. 
Thus, by the aforementioned result of Falk and Randell \cite{FR}, the LCS quotients  
$\gr_k(\pi_1(U))$ are isomorphic to $\gr_k(F_5)\oplus \gr_k(F_3)$, 
for all $k\ge 1$. Moreover, the Chen ranks are given by 
$\theta_k(\pi_1(U))=(k-1) (3k+19)$ for $k\ge 4$, see \cite{Su01, CSc-adv}, 

\begin{figure}
 \centering
 \quad
  \captionbox{The $\operatorname{B}_3$ reflection arrangement, 
   with $(3,4)$-multinet\label{fig:b3 arr}}
    [.4\textwidth]{
\begin{tikzpicture}[scale=0.9]
\draw[style=thick,color=dkgreen] (0,0) circle (3.1);
\node at (-2.4,0.3) {$z^2$};
\node at (0,-2.55) {$x^2$};
\node at (3.4,0.5) {$y^2$};
\draw[style=thick,densely dashed,color=blue] (-1,-2.1) -- (-1,2.5);
\draw[style=thick,densely dotted,color=red] (0,-2.2) -- (0,2.5);
\draw[style=thick,densely dashed,color=blue] (1,-2.1) -- (1,2.5);
\draw[style=thick,densely dotted,color=red] (-2.5,-1) -- (2.5,-1);
\draw[style=thick,densely dashed,color=blue] (-2.5,0) -- (2.5,0);
\draw[style=thick,densely dotted,color=red] (-2.5,1) -- (2.5,1);
\draw[style=thick,color=dkgreen]  (-2,-2) -- (2,2);
\draw[style=thick,color=dkgreen](-2,2) -- (2,-2);
\node at (-2.1,-0.75) {$y+z$};
\node at (-2.1,1.25) {$y-z$};
\node at (-1,-2.35) {$x-y$};
\node at (1,-2.35) {$x+y$};
\node at (-2,-1.4) {$x-z$};
\node at (2.1,-1.4) {$x+z$};
\end{tikzpicture}
}
\:\qquad
  \captionbox{The deleted $\operatorname{B}_3$ arrangement, with 
  multiplicities \label{fig:delb3}}[.41\textwidth]{%
\begin{tikzpicture}[scale=0.9]
\draw (0,0) circle (3.1);
\node at (0,-2.7) {2};
\node at (3.35,0.5) {1};
\clip (0,0) circle (2.9);
\draw (-1,-2.1) -- (-1,2.5);
\draw (0,-2.2) -- (0,2.5);
\draw (1,-2.1) -- (1,2.5);
\draw (-2.5,-1) -- (2.5,-1);
\draw (-2.5,1) -- (2.5,1);
\draw (-2,-2) -- (2,2);
\draw (-2,2) -- (2,-2);
\node at (-2.25,-0.7) {1};
\node at (-2.25,1.3) {1};
\node at (-1,-2.5) {2};
\node at (1,-2.5) {2};
\node at (-2,-1.5) {3};
\node at (2,-1.5) {3};
\end{tikzpicture}
}
\end{figure}
 
We now turn to the cohomology jump loci of the  $\operatorname{B}_3$ 
arrangement (see \cite[Rem.~6.4]{CS99} and \cite[Ex.~3.6]{FY}). 
Notably, $\A$ supports a (non-reduced) multinet $\mathcal{N}$, depicted 
in Figure \ref{fig:b3 arr}; ordering the hyperplanes as the factors of the 
defining polynomial \eqref{eq:poly-b3}, this multinet has associated partition 
$(189|267|345)$. 
The resonance variety $\RR^1_1(M)\subset H^1(M;\C)=\C^9$
has $7$ local components, corresponding to the $4$ triple points and $3$ 
quadruple points, $11$ components corresponding to braid sub-arrangements, 
and one essential, $2$-dimensional component, $P=P_{\mathcal{N}}$. 
All the components of the characteristic variety 
$\VV^1_1(M)\subset H^1(M;\C^*) = (\C^*)^9$ pass through the origin, 
and thus are obtained by exponentiating the linear subspaces comprising 
$\RR^1_1(M)$. In particular, there is a single essential component, $T=\exp(P)$. 
More explicitly, the multinet $\mathcal{N}$ determines a pencil, 
\begin{equation}
\label{eq:b3-pencil}
\begin{tikzcd}[column sep=20pt]
\psi\colon M \ar[r]& S=\CP^1 \mysetminus \{\pcoor{0:1},\pcoor{1:0},\pcoor{1:1}\},
\end{tikzcd}
\end{equation}
which is given by $\psi(x,y,z)=\pcoor{x^2(y^2-z^2):y^2(x^2-z^2)}$. In turn, 
the induced homomorphism $\psi^*\colon H^1(S;\Z)\to H^1(M;\Z)$ is given by  
$c^{\vee}_1\mapsto 2e_1+e_8+e_9$, $c^{\vee}_2\mapsto 2e_2+e_6+e_7$, 
$c^{\vee}_3\mapsto 2e_3+e_4+e_5$, where $c_i=[\gamma_i]$ are the homology 
classes of standard loops around the punctures of $S$ (see Section \ref{subsec:multinets}). 
Hence, 
\begin{equation}
\label{eq:b3-essential}
T=\psi^*(H^1(S;\C^*))=\{ (t^2, s^2 , (st)^{-2}, s, s, (st)^{-1},  (st)^{-1}, t, t) : 
s,t\in \C^*\}.
\end{equation}

Finally, let $F=F(\A)$ be the Milnor fiber of the $\operatorname{B}_3$ 
arrangement; then none of the aforementioned components of 
$\VV^1_1(M)$ contributes to a jump in $b_1(F)$. In fact, as 
first shown in \cite{CS95}, the monodromy $h\colon F\to F$ acts trivially on 
$H_1(F;\Q)$; analyzing more carefully that computation shows that  
$h$ acts trivially on $H_1(F;\Z)$. Applying Theorem \ref{thm:trivial-mono-z}, 
we conclude that $\gr_k(\pi_1(F))\cong \gr_k(F_5)\oplus \gr_k(F_3)$, 
and $\theta_k(\pi_1(F))=\theta_k(\pi_1(U))$ for all $k\ge 1$.

\subsection{The deleted $\operatorname{B}_3$ arrangement}
\label{subsec:del-b3}
Consider now the arrangement $\A'$ obtained from $\A$ by 
deleting the hyperplane $\{z=0\}$, as shown in Figure \ref{fig:delb3}. 
This is the {\em deleted $\operatorname{B}_3$ arrangement}, 
defined by the polynomial 
\begin{equation}
\label{eq:delb3-poly}
f'=xy(x-y)(x+y)(x-z)(x+z)(y-z)(y+z).
\end{equation}
This is again a fiber-type arrangement, with exponents $\{1,3,4\}$. 
Thus, the complement $M'=M(\A)$ is aspherical and its projectivization, 
$U'=\P(M')$, has fundamental group $\pi_1(U')=F_4\rtimes_{\alpha'} F_3$, 
where, as noted in \cite[Ex.~10.6]{Su01}, the monodromy automorphism 
$\alpha'$ is given by the pure braids $A_{23}$, $A_{13}^{A_{23}}A_{24}$, and 
$A_{14}^{A_{24}}$.

The cohomology jump loci of $M'$ were computed in \cite{Su02}.  
Briefly, the resonance variety $\RR^1_1(M')\subset H^1(M';\C)=\C^8$ 
contains $7$ local components, corresponding to the $6$ triple points and $1$ 
quadruple point, and $5$ non-local components, corresponding 
to braid sub-arrangements.  In addition to the $12$ subtori obtained 
by exponentiating these linear subspaces, the 
characteristic variety $\VV^1_1(M')\subset 
H^1(M';\C^*)=(\C^{*})^8$ also contains a component of the form 
$\rho\cdot T'$, where $T'$ is a $1$-dimensional algebraic subtorus 
and $\rho$ is a root of unity of order $2$, given by 
\begin{gather}
\label{eq:delb3-torus}
\begin{aligned}
T'&=\set{(t^2,t^{-2},t^{-1},t^{-1},1,1,t,t) : t\in \C^{*}}, \\
\rho&=(1,1,-1,-1,-1,-1,1,1).
\end{aligned}
\end{gather}

As explained in \cite[Ex.~5.7]{DeS-plms}, this translated subtorus arises
from the pencil $\psi$ from \eqref{eq:b3-pencil}, as follows.
The point $\pcoor{0 : 1}$ is not in the image of $\psi$; however, 
extending the domain of $\psi$ to $M'=M\cup\{z=0\}$ defines a map 
\begin{equation}
\label{eq:delb3-pencil}
\begin{tikzcd}[column sep=20pt]
\psi' \colon M'\ar[r] & \C^*= \CP^1 \mysetminus \{\pcoor{0:1},\pcoor{1:0}\}.
\end{tikzcd}
\end{equation}

Note that $\psi'(x,y,0)=\pcoor{x^2y^2 : x^2y^2}$, 
so the fiber over $\pcoor{1 : 1}$ has multiplicity $2$. Therefore, we may view the map 
$\psi' \colon M'\to (\C^*,(2))$ as an orbifold pencil, with one multiple fiber of multiplicity $2$.
The orbifold fundamental group $\Gamma=\pi_1^{\orb}(\C^*, (2))$ may be identified with the 
free product $\Z*\Z_2$, while the character group $H^1(\Gamma;\C^*)$ may be identified with 
$\C^*\times \{\pm 1\}$. It follows from \eqref{eq:v1piorb} that 
$\V^1_1(\Gamma) = \C^*\times \{-1\}$. The map $\psi'$ induces an epimorphism 
$\psi'_{\sharp}\colon \pi_1(M')\surj \Gamma$, which in turn induces a monomorphism 
$(\psi'_{\sharp})^*\colon H^1(\Gamma;\C^*) \inj H^1(\pi_1(M');\C^*)$. The 
image of $\V^1_1(\Gamma)$ under this morphism is precisely the translated 
torus $\rho T'\subset \V^1_1(M')$.  Moreover, if we let $j\colon M\inj M'$ 
be the inclusion map between the respective complements, 
then the induced homomorphism, $j^*\colon H^1(M';\C^*)\inj H^1(M;\C^*)$, 
embeds $\rho T'$ into the torus $T$ from \eqref{eq:b3-essential}. In fact, 
 $T \cap \{t\in (\C^*)^9 : t_3=1\}=T' \cup \rho T'$.

\subsection{Milnor fibrations of the deleted $\operatorname{B}_3$ arrangement}
\label{subsec:mf-del-b3}

It follows from the above discussion that the deleted $\operatorname{B}_3$ 
arrangement $\A'$ supports no essential, reduced multinet. It is readily verified 
that none of aforementioned components of $\VV^1_1(M')$ contributes to a 
jump in the first Betti number of $F'=F(\A')$. Direct computation shows that, 
in fact, $H_1(F';\Z)=\Z^8$, and so the monodromy acts trivially on $H_1(F';\Z)$.  
For suitable choices of multiplicities, though, the Milnor fiber of the multi-arrangement 
acquires non-trivial $2$-torsion. We treat in detail one such choice. 

Let $F'_{\bm}=F_{\bm}(\A')$ be the Milnor fiber of the multi-arrangement 
$(\A',\bm)$ with multiplicity vector $\bm=(2,1,2,2,3,3,1,1)$. 
As noted in \cite{CDS, DeS-plms}, the monodromy of the Milnor 
fibration acts trivially on $H_1(F'_{\bm};\Q)$, but not on $H_1(F'_{\bm};\Z)$, 
which has torsion subgroup $\Z_2\oplus \Z_2$ on which the monodromy 
acts as $\bigl(\begin{smallmatrix}0&1\\1&1\end{smallmatrix}\bigr)$. 

Let $U'=U(\A')$, and consider the pullback square on the right side 
of the following diagram 
\begin{equation}
\label{eq:3-fold-delb3}
\begin{tikzcd}[
        /tikz/row 1/.append style={row sep=10pt}, 
        /tikz/row 2/.append style={row sep=24pt}, 
        /tikz/column 1/.append style={column sep=28pt}, 
        /tikz/column 2/.append style={column sep=32pt}, 
        /tikz/column 3/.append style={column sep=0pt}, 
        /tikz/column 4/.append style={anchor=base west, column sep=0pt, inner xsep =0pt}]
F'_{\bm} \ar[dr, "\kappa"] \ar[ddr, "\sigma_{\bm}" '] 
\\
& \hat{U}'\ar[d, "\tau"]  \ar[r, "\hat{\psi}'"] & \hat{S} \ar[d, "\nu"] &=(\C^*,(2,2,2)) 
\\
& U' \ar[r, "\psi'"] & S&=(\C^*,(2)) .
\end{tikzcd}
\end{equation}
where $\psi'$ is the (projectivized) orbifold pencil from Section \ref{subsec:del-b3} and 
$\nu$ is the orbifold $3$-fold cover corresponding to the epimorphism 
$\pi^{\orb}_1(S)=\Z*\Z_2 \surj \Z_3$ that sends the (meridional) generator of 
$\pi_1(\C^*)=\Z$ to $1$ and the generator of $\Z_2$ to $0$. 
The orbifold fundamental group $\Gamma=\pi_1^{\orb}(\hat{S})$ 
is isomorphic to $\Z*\Z_2*\Z_2 *\Z_2$, 
and so $\T_{\Gamma}=\T_{\Gamma}^0\times \{(\pm 1, \pm 1, \pm 1)\}$, 
where $\T_{\Gamma}^0=\C^*$.  It follows from \eqref{eq:v1piorb} that 
\begin{gather}
\label{eq:delb3-cv-gamma}
\begin{aligned}
\V^1_1(\Gamma)&=\{\bo\}\cup (\T_{\Gamma} \mysetminus \T_{\Gamma}^0),
\\
\V^1_2(\Gamma)&=(1,-1,-1)\T_{\Gamma}^0 \cup (-1,1,-1)\T_{\Gamma}^0 
\cup (-1,-1,1)\T_{\Gamma}^0 \cup (-1,-1,-1)\T_{\Gamma}^0,
\\ 
\V^1_3(\Gamma)&=(-1,-1,-1)\T_{\Gamma}^0. 
\end{aligned}
\end{gather}
Moreover, the lift $\hat{\psi}'\colon \hat{U}'\to \hat{S}$ is again an orbifold pencil.

The $\Z_{15}$-cover $\sigma_{\bm}\colon F'_{\bm}\to U'$ factors as the composite 
$F'_{\bm}\xrightarrow{\kappa} \hat{U}'\xrightarrow{\tau} U'$, where $\kappa$ is a 
$5$-fold cover.  By Theorem \ref{thm:cjl-mf-trivialmono}, 
part \eqref{pa}, the subvariety $\WW^1_1(F'_{\bm})$ has $12$ components passing 
through the identity of $H^1(F'_{\bm};\C^*)^0=(\C^*)^7$: 
eleven subtori of dimension $2$ and one subtorus 
of dimension $3$ (which in fact is a component of $\WW^1_2(F'_{\bm})$), 
all obtained by pullback along $\sigma_{\bm}$. 
By Theorem \ref{thm:cjl-mf-trivialmono}, part \eqref{pb}, 
there is also a $1$-dimensional component of $\WW^1_1(F'_{\bm})$ 
of the form $\sigma^*_{\bm}(\rho T')$, where $\rho T'$ is the translated 
subtorus in $\V^1_1(U')$ from \eqref{eq:delb3-torus}. Pulling 
back along the map 
$(\hat\psi'\circ \kappa)^*\colon H^1(\hat{S};\C^*)\to H^1(F'_{\bm};\C^*)$ 
the translated tori comprising $\V^1_1(\Gamma)$ yields seven $1$-dimensional  
components of $\V^1_1(F'_{\bm})$, of the form 
$\rho'\sigma^*_{\bm}(T')$, for certain order $2$ 
characters $\rho'$. Of those, $4$ are also components 
of $\V^1_2(F'_{\bm})$, while one of those, namely,  
$(\hat\psi'\circ \kappa)^* \big((-1,-1,-1)\T_{\Gamma}^0\big) = 
\sigma^*_{\bm}(\rho T')$, is the unique component of $\V^1_3(F'_{\bm})$.

Finally, since $\A'$ is fiber-type with exponents $\{1,3,4\}$, 
the  lower central series quotients $\gr_k(\pi_1(U'))$ are isomorphic to 
$\gr_k(F_4)\oplus \gr_k(F_3)$ for $k\ge 2$, while, by \cite{Su01, CSc-adv}, 
the Chen ranks $\theta_k(\pi_1(U'))$ are equal to $(k-1)(k+12)$ for $k\ge 4$. 
By Theorem \ref{thm:trivial-mono-q}, the group $K=\pi_1(F'_{\bm})$ has the 
same LCS and Chen ranks as $\pi_1(U')$. In fact, it can be 
shown that $\gr_k(K)\otimes \Z_p \cong \gr_k(\pi_1(U'))\otimes \Z_p$ 
for all primes $p\ne 2$, and likewise for the Chen groups of $K$. 
Direct computation shows that the first few lower central series 
quotients of $K$ and $K/K''$  are as in the following table.
\setlength{\arraycolsep}{8pt}
\def\arraystretch{1.4}
\begin{equation*}
\label{eq:hom-hn}
\begin{array}{|c|c|c|c|c|c|c|c|c|c|c|c}
\hline
k & 1 & 2& 3 & 4 & 5\\
\hline
\gr_k(K) & \Z^{7}\oplus \Z_2^{2} & \Z^{9}\oplus \Z_2^{5} 
& \Z^{28}\oplus \Z_2^{15} & \Z^{78}\oplus \Z_2^{41} & \Z^{252}\oplus \Z_2^{117}\\
\hline
\gr_k(K/K'') & \Z^{7}\oplus \Z_2^{2} & \Z^{9}\oplus \Z_2^{5} 
& \Z^{28}\oplus \Z_2^{15} & \Z^{48}\oplus \Z_2^{?} & \Z^{68}\oplus \Z_2^{?}\\
\hline
\end{array}
\end{equation*}
\def\arraystretch{1}

\section{Yoshinaga's icosidodecahedral arrangement}
\label{sect:mf-yoshi}

In this final section, we describe an arrangement, introduced by Yoshinaga in \cite{Yo20}, 
which exhibits $2$-torsion in the first homology of its (usual) Milnor fiber. 

\subsection{Mod-$2$ Betti numbers of $2$-fold covers}
\label{subsec:2-fold}

Before proceeding with the example, we return to the general setup 
for computing the homology of finite abelian covers treated 
in Section \ref{subsec:hom-cov}, approached this time from a different angle. 

Let $p\colon Y\to X$ be a regular $\Z_{N}$-cover, classified by a 
homomorphism $\alpha\colon \pi_1(X)\to\Z_{N}$. Alternatively, we 
may view $\alpha$ as a cohomology class in $H^1(X;\Z_{N})$, 
called the characteristic class of the cover. 
The covering space $Y=X^{\alpha}$ is connected if and only if the 
homomorphism $\alpha$ is surjective, in which case $\pi_1(Y)=\ker(\alpha)$. 
In the case when $N=2$, more can be said. The next two results 
were first proved in \cite{Yo20} and then strengthened in \cite{Su-bock}.

\begin{lemma}[\cite{Yo20, Su-bock}]
\label{lem:2-fold-cover}
Let $p\colon Y\to X$ be a connected $\Z_2$-cover, with characteristic class 
$\alpha\in H^1(X;\Z_2)$. Then $p$ lifts to a connected, regular $\Z_4$-cover 
$\bar{p}\colon \overline{Y}\to X$ if and only if $\alpha^2=0$.
\end{lemma}

\begin{prop}[\cite{Yo20, Su-bock}]
\label{prop:betti-cover}
Let $p\colon Y\to X$ be a $2$-fold cover, classified by a non-zero class 
$\alpha\in H^1(X;\Z_2)$. Suppose that $\alpha^2=0$. 
Then, for all $q\ge 1$, 
\begin{equation}
\label{eq:hcov}
b_q(Y,\Z_2)= b_q(X,\Z_2) + \dim_{\Z_2} H^q(H^{*}(X;\Z_2),\delta_{\alpha}), 
\end{equation}
where the differential $\delta_{\alpha}\colon H^{*}(X;\Z_2)\to H^{*+1}(X;\Z_2)$ is given 
by $\delta_{\alpha}(u)=\alpha u$. In particular, $b_q(Y,\Z_2)\ge b_q(X,\Z_2)$. 
\end{prop}

Further work on the integral homology groups of double covers, and how this 
homology relates to the homology with coefficients in rank $1$ integral local 
systems on the base of the cover can be found in \cite{Sg23, LL, LMW2}. 
 
\subsection{Modular inequalities}
\label{subseq:modular}
Once again, let $Y\to X$ be a connected $\Z_2$-cover 
with characteristic class $\alpha\in H^1(X;\Z_2)$. Assuming $H_{*}(X;\Z)$ is 
torsion-free, it follows from \cite[Thm.~C]{PS-tams} that 
\begin{equation}
\label{eq:modular}
b_q(Y) \le  b_q(X) + \dim_{\Z_2} H^q(H^{*}(X;\Z_2),\delta_{\alpha}).
\end{equation}

When $U=U(\A)$ is the projectivized complement of a hyperplane arrangement $\A$, 
an explicit formula was proposed in \cite[Conjecture 1.9]{PS-plms17}, expressing 
the first Betti number $b_1(F)$ of the Milnor fiber of the arrangement in terms of 
the resonance varieties $\RR^1_s(U,\Z_p)$, for $p=2$ and $3$, generalizing 
the formula from Theorem \ref{thm:ps-triple}. At the prime $p=2$, the conjecture 
is equivalent to the inequality \eqref{eq:modular} holding as equality in degree $q=1$ 
for the $2$-fold cover $U^{\alpha}\to U$ corresponding to the class $\alpha\in H^1(U;\Z_2)$ 
which evaluates to $1$ on each meridional generator of $H_1(U;\Z_2)$.

In recent work \cite{ISY}, Ishibashi, Sugawara, and Yoshinaga revisit this 
topic. In \cite[Cor.~2.5]{ISY}, they prove that equality holds in \eqref{eq:modular} 
if and only if $H_1(Y;\Z)$ has no non-trivial $2$-torsion. 
Therefore, the formula conjectured in \cite{PS-plms17} fails  
at the prime $p=2$ precisely when $H_1(U^{\alpha};\Z)$ has 
non-trivial $2$-torsion. An explicit example where this happens 
is given next. 

\begin{figure}[t]
\centering
\begin{tikzpicture}[scale=1.0]
\draw (0,0) circle (3.3);
\draw (0.809,0)++(0,-3) -- +(90:6); 
\draw (-0.309,0)++(0,-3) -- +(90:6); 
\draw (0.118,0)++(0,-3) -- +(90:6); 
\draw (72:0.809)++(162:3) -- +(342:6); 
\draw (252:0.309)++(162:3) -- +(342:6); 
\draw (72:0.118)++(162:3) -- +(342:6); 
\draw (144:0.809)++(234:3) -- +(54:6); 
\draw (322:0.309)++(234:3) -- +(54:6); 
\draw (144:0.118)++(234:3) -- +(54:6);
\draw (216:0.809)++(126:3) -- +(306:6); 
\draw (36:0.309)++(126:3) -- +(306:6); 
\draw (216:0.118)++(126:3) -- +(306:6); 
\draw (288:0.809)++(198:3) -- +(18:6); 
\draw (108:0.309)++(198:3) -- +(18:6); 
\draw (288:0.118)++(198:3) -- +(18:6); 
\end{tikzpicture}
\caption{The icosidodecahedral arrangement}
\label{fig:yoshi}
\end{figure}
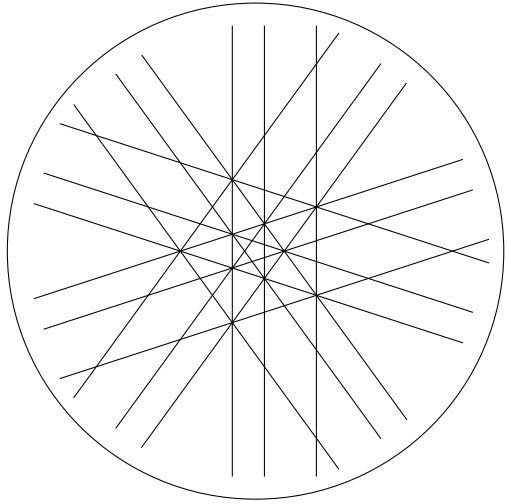

\subsection{The icosidodecahesdral arrangement}
\label{subsec:yoshi-arr}
In \cite{Yo20}, Yoshinaga constructed an arrangement of $16$ hyperplanes 
in $\C^3$  with some remarkable properties.  The construction is based on 
the symmetries of a polyhedron in $\R^3$, called the {\em icosidodecahedron}. 
This is a quasiregular polyhedron with $20$ triangular and $12$ pentagonal 
faces that has $30$ vertices (each one at the intersection of $2$ triangles 
and $2$ pentagons), and $60$ edges (each one separating a triangle from 
a pentagon). Letting $\phi=(1+\sqrt{5})/2$ denote the golden ratio, 
the vertices of an icosidodecahedron with edges of unit length
are given by the even permutations 
of $(0, 0, \pm 1)$ and $\tfrac{1}{2}(\pm 1, \pm \phi, \pm \phi^{2})$.

One can choose $10$ edges to form a decagon, corresponding to great circles in the 
spherical tiling; there are $6$ ways to choose these decagons, thereby giving $6$  
planes. Each pentagonal face has five diagonals, and there are $60$ such diagonals in all, which 
partition in $10$ disjoint sets of coplanar ones, thereby giving $10$ planes, each containing 
$6$ diagonals. These $16$ planes form an arrangement $\A_{\R}$ in $\R^3$, whose complexification 
is the icosidodecahedral arrangement $\A$ depicted in Figure \ref{fig:yoshi}.

The projective line arrangement $\P(\A)$ has $15$ quadruple points and $30$ double points. 
The projective complement $U=U(\A)$ is aspherical \cite{Liu}, and has Poincar\'e polynomial 
$P(t)=1+15t+60t^2$. Let $F=F(\A)$ be the Milnor fiber of is arrangement. As shown in \cite{Yo20}, 
we have that $H_1(F;\Z)=\Z^{15}\oplus \Z_2$. Thus, the algebraic monodromy 
of the Milnor fibration is trivial over $\Q$, but not over $\Z$.

Since the monodromy of the Milnor fibration acts trivially on $H_1(F;\k)$ 
for every field $\k$ of characteristic different from $2$, the 
results of \cite{Su-lcs-mono} show that $\gr(\pi_1(F))\otimes \k\cong \gr(\pi_1(U))\otimes \k$ 
for such fields $\k$. Direct computation shows that the 
first few lower central series quotients of the group $K=\pi_1(F)$ 
and of its maximal metabelian quotient  are given by
\setlength{\arraycolsep}{8pt}
\def\arraystretch{1.4}
\begin{equation*}
\label{eq:yoshi-gr}
\begin{array}{|c|c|c|c|c|c|c|c|c|c|}
\hline
k & 1 & 2& 3 & 4 \\
\hline
\gr_k(K) & \Z^{15}\oplus \Z_2 & \Z^{45}\oplus \Z_2^{7}
& \Z^{250}\oplus \Z_2^{43}& \Z^{1,405}\oplus T\\
\hline
\gr_k(K/K'') &  \Z^{15}\oplus \Z_2 & \Z^{45}\oplus \Z_2^{7}
& \Z^{250}\oplus \Z_2^{43}& \Z^{530}\oplus \overline{T} \\
\hline
\end{array}
\def\arraystretch{1}
\end{equation*}
where $T$ is a finite abelian $2$-group 
and $\overline{T}$ is a quotient of $T$.

\smallskip
\begin{ack}
\label{ack}
Preliminary versions of this work were presented at meetings  
at Oberwolfach (January 2021), Locarno (July 2022), 
Edinburgh (March 2023), Lille (June 2023), and Tokyo (December 2023). 
Much of the work was done in 2023 while the author was 
visiting the Mathematical Research Institute at the University 
of Sydney, Australia (January--February), the 
Institute of Mathematics of the Romanian Academy 
in Bucharest, Romania (May--July), the 
University of Lille, France (June), and Rykkio and 
Osaka Universities in Japan (December).
I wish to thank the organizers of those meetings 
for giving me the opportunity to lecture on these topics,  
and the aforementioned institutes and universities for 
their support and hospitality during my stays there. 
I also thank the referee for useful comments that 
helped improve the paper.
\end{ack}


\newcommand{\arxiv}[1]
{\texttt{\href{http://arxiv.org/abs/#1}{arXiv:#1}}}

\newcommand{\arx}[1]
{\texttt{\href{http://arxiv.org/abs/#1}{arXiv:}}
\texttt{\href{http://arxiv.org/abs/#1}{#1}}}

\newcommand{\arxx}[2]
{\texttt{\href{http://arxiv.org/abs/#1.#2}{arxiv:#1.}} 
\texttt{\href{http://arxiv.org/abs/#1.#2}{#2}}}

\newcommand{\doi}[1]
{\texttt{\href{http://dx.doi.org/#1}{doi:\nolinkurl{#1}}}}

\renewcommand{\MR}[1]
{\href{http://www.ams.org/mathscinet-getitem?mr=#1}{MR#1}}


\begin{thebibliography}{00}

\bibitem{AFRS}
M.~Aprodu, G.~Farkas, C.~Raicu, A.I. Suciu,
{\em Reduced resonance schemes and Chen invariants}, 
J. Reine Angew. Math. (to appear), \arxiv{2303.07855v2}.

\bibitem{Ar}  D.~Arapura, 
{\em Geometry of cohomology support loci for local systems {\rm I}}, 
J. Algebraic Geom. \textbf{6} (1997), no.~3, 563--597.  
\MR{1487227}

\bibitem{ACM}  E.~Artal Bartolo, J.I.~Cogolludo, D.~Matei,
\href{https://doi.org/10.2140/gt.2013.17.273}%
{\em Characteristic varieties of quasi-projective manifolds and orbifolds},
Geom. Topol. \textbf{17} (2013), no.~1, 273--309.
\MR{3035328} 

\bibitem{Bt14} P.~Bailet,
\href{https://doi.org/10.4153/CMB-2014-032-4}%
{\em On the monodromy of Milnor fibers of hyperplane arrangements}, 
Canad. Math. Bull. \textbf{57} (2014), 697--707.
\MR{3270792}

\bibitem{Br}  E.~Brieskorn, 
\href{https://doi.org/10.1007/BFb0069274}%
{\em Sur les groupes de tresses (d'apr\`es V. I. Arnol'd)},  
S\'eminaire Bourbaki, 24\`eme ann\'ee (1971/1972), Exp. No. 401, 
pp. 21--44, Lecture Notes in Math., vol.~317, Springer, Berlin, 1973. 
\MR{0422674}

\bibitem{BW} N.~Budur, B.~Wang, 
\href{https://doi.org/10.24033/asens.2242}%
{\em Cohomology jump loci of quasi-projective varieties}, 
Ann. Sci. \'{E}cole Norm. Sup. \textbf{48} (2015), no.~1, 227--236. 
\MR{3335842}

\bibitem{CC}  D.~Chataur, J.~Cirici, 
\href{https://doi.org/10.1515/forum-2015-0101}%
{\em Rational homotopy of complex projective varieties with normal 
isolated singularities}, Forum Math.\textbf{29} (2017), no.~1, 41--57.
\MR{3592593}

\bibitem{CDP} A.D.R.~Choudary, A.~Dimca, \c{S}.~Papadima, 
\href{https://doi.org/10.2140/agt.2005.5.691}%
{\em Some analogs of Zariski's Theorem on nodal line arrangements}, 
Algebr. Geom. Topol. \textbf{5} (2005), 691--711.
\MR{2153112}

\bibitem{CDS} D.C.~Cohen, G.~Denham, A.I.~Suciu,
\href{https://doi.org/10.2140/agt.2003.3.511}%
{\em Torsion in Milnor fiber homology}, 
Algebr. Geom. Topology \textbf{3} (2003), 511--535. 
\MR{1997327}

\bibitem{CDO03}  D.C.~Cohen, A.~Dimca, P.~Orlik,
\href{https://doi.org/10.5802/aif.1994}%
{\em Nonresonance conditions for arrangements}, 
Ann. Inst. Fourier (Grenoble) \textbf{53} (2003), 
no.~6, 1883--1896. 
\MR{2038782}

\bibitem{CSc-adv} D.C.~Cohen, H.K.~Schenck,
\href{https://dx.doi.org/10.1016/j.aim.2015.07.023}%
{\em Chen ranks and resonance},
Adv. Math. \textbf{285} (2015), 1--27.
\MR{3406494}

\bibitem{CS95} D.C.~Cohen, A.I.~Suciu, 
\href{https://doi.org/10.1112/jlms/51.1.105}%
{\em On {M}ilnor fibrations of arrangements},  
J. London Math. Soc. (2) \textbf{51} (1995), no.~1, 105--119.
\MR{1310725} 

\bibitem{CS97}  D.C.~Cohen, A.I.~Suciu,
\href{https://dx.doi.org/10.1007/s000140050017}%
{\em The braid monodromy of plane algebraic curves and 
hyperplane arrangements}, Comment. Math. Helvetici 
\textbf{72} (1997), no.~2, 285--315. 
\MR{1470093}

\bibitem{CS98}  D.C.~Cohen, A.I.~Suciu,
\href{http://dx.doi.org/10.1016/S0022-4049(96)00153-3}%
{\em Homology of iterated semidirect products of free groups}, 
J. Pure Appl. Algebra \textbf{126} (1998), no.~1-3, 87--120.  
\MR{1600518}

\bibitem{CS99} D.C.~Cohen, A.I.~Suciu, 
\href{https://doi.org/10.1017/S0305004199003576}%
{\em Characteristic varieties of arrangements},
Math. Proc. Cambridge Phil. Soc. \textbf{127} (1999), 
no.~1, 33--53. 
\MR{1692519} 

\bibitem{DeS-plms} G.~Denham, A.I.~Suciu,
\href{https://doi.org/10.1112/plms/pdt058}%
{\em Multinets, parallel connections, and {M}ilnor fibrations  
of arrangements}, Proc. London Math. Soc. 
\textbf{108} (2014), no.~6, 1435--1470.
\MR{3218315}

\bibitem{DeS-sigma} G.~Denham, A.I.~Suciu,
\href{https://dx.doi.org/10.1017/fms.2018.5}%
{\em Local systems on complements of arrangements 
of smooth, complex algebraic hypersurfaces}, 
Forum Math. Sigma \textbf{6} (2018), e6, 20 pages.
\MR{3810026}

\bibitem{DSY16} G.~Denham, A.I.~Suciu, and S.~Yuzvinsky, 
\href{https://dx.doi.org/10.1007/s00029-015-0196-8}%
{\em Combinatorial covers and vanishing of cohomology}, 
Selecta Math. (N.S.) \textbf{22} (2016), no.~2, 561--594. 
\MR{3477330}

\bibitem{DSY17} G.~Denham, A.I.~Suciu, S.~Yuzvinsky, 
\href{https://dx.doi.org/10.1007/s00029-017-0343-5}%
{\em Abelian duality and propagation of resonance}, 
Selecta Math. \textbf{23} (2017), no.~4, 2331--2367. 
\MR{3703455}

\bibitem{DP-pisa} A.~Dimca, S.~Papadima,
\href{http://www.numdam.org/item/ASNSP_2011_5_10_2_253_0/}%
{\em Finite {G}alois covers, cohomology jump loci, formality 
properties, and multinets},  Ann. Sc. Norm. Super. Pisa Cl. Sci. 
\textbf{10} (2011), no.~2, 253--268. 
\MR{2856148} 

\bibitem{DP-ccm} A.~Dimca, \c{S}.~Papadima,
\href{http://dx.doi.org/10.1142/S0219199713500259}%
{\em Non-abelian cohomology jump loci from an analytic viewpoint}, 
Commun. Contemp. Math. \textbf{16} (2014), 
no.~4, 1350025, 47~pp. 
\MR{3231055}

\bibitem{DPS-duke} A.~Dimca, \c{S}.~Papadima, A.~Suciu,
\href{https://dx.doi.org/10.1215/00127094-2009-030}%
{\em Topology and geometry of cohomology jump loci}, 
Duke Math. Journal \textbf{148} (2009), no.~3, 405--457.
\MR{2527322} 

\bibitem{Du} C.~Dupont, 
\href{https://doi.org/10.1093/imrn/rnv260}%
{\em Purity, formality, and arrangement complements}, 
Int. Math. Res. Not. IMRN \textbf{2016} (2016), no.~13, 4132--4144.
\MR{3544631}

\bibitem{Fa89} M.~Falk,
\href{https://doi.org/10.1090/conm/090/1000594}%
{\em The cohomology and fundamental group of a hyperplane
complement}, Singularities (Iowa City, IA, 1986), 55--72, 
Contemp. Math., vol.~90, Amer. Math. Soc., Providence, RI, 1989.
\MR{1000594}

\bibitem{Fa93} M.~Falk, 
\href{https://doi.org/10.1007/BF01231283}%
{\em Homotopy types of line arrangements}, 
Invent. Math. \textbf{111} (1993), no.~1, 139--150. 
\MR{1193601}

\bibitem{Fa97} M.~Falk,
\href{https://doi.org/10.1007/BF02558471}%
{\em Arrangements and cohomology},
Ann. Combin. \textbf{1} (1997), no.~2, 135--157.  
\MR{1629681} 

\bibitem{FR} M. Falk, R. Randell,
\href{https://dx.doi.org/10.1007/BF01394780}%
{\em The lower central series of a fiber-type arrangement},
Invent. Math. \textbf{82} (1985), no.~1, 77--88. 
\MR{808110}

\bibitem{FY} M.~Falk, S.~Yuzvinsky,
\href{https://doi.org/10.1112/S0010437X07002722}%
{\em Multinets, resonance varieties, and pencils of plane curves},
Compositio Math. \textbf{143} (2007), no.~4, 1069--1088.
\MR{2339840} 

\bibitem{Ha} A.~Hattori,
\href{https://doi.org/10.15083/00039752}%
{\em Topology of $\C^n$ minus a finite number of
affine hyperplanes in general position},
J. Fac Sci. Univ. Tokyo \textbf{22} (1975), 205--219.
\MR{0379883}

\bibitem{Hi97} E.~Hironaka,
\href{https://doi.org/10.5802/aif.1573}%
{\em Alexander stratifications of character varieties}, Annales 
de l'Institut Fourier (Grenoble) \textbf{47} (1997), no.~2, 555--583.
\MR{1450425}

\bibitem{ISY} S.~Ishibashi, S.~Sugawara, M.~Yoshinaga, 
{\em Betti numbers and torsions in homology groups of double coverings}, 
\arxiv{2209.02236v2}.

\bibitem{Li92} A.~Libgober,
\href{https://doi.org/10.1016/0166-8641(92)90137-O}%
{\em On the homology of finite abelian coverings},
Topology Appl. \textbf{43} (1992), no.~2, 157--166.
\MR{1152316} 

\bibitem{Li02}  A.~Libgober, 
\href{https://doi.org/10.1007/978-3-0348-8161-6_6}%
{\em Eigenvalues for the monodromy of the Milnor fibers of 
arrangements}, Trends in singularities, pp.~141--150, 
Trends Math., Birkh\"{a}user, Basel, 2002. 
\MR{1900784}

\bibitem{LY00} A.~Libgober, S.~Yuzvinsky,
\href{https://doi.org/10.1023/A:1001826010964}%
{\em Cohomology of {O}rlik--{S}olomon algebras and local 
systems},  Compositio Math. \textbf{21} (2000), no.~3, 337--361.  
\MR{1761630} 

\bibitem{Liu}  Ye~Liu, 
{\em  Topology of the icosidodecahedral arrangement}, 
\arxiv{1908.01280v1}.

\bibitem{LL}  Ye~Liu, Y.~Liu,
\href{http://dx.doi.org/10.1090/proc/16528}%
{\em Integral homology groups of double coverings and rank one 
$\mathbb{Z}$-local systems for a minimal CW complex}, Proc. 
Amer. Math. Soc. \textbf{151} (2023), no.~11, 5007--5012.
\MR{4634902}

\bibitem{LMW}  Y.~Liu, L.~Maxim, B.~Wang, 
\href{https://doi.org/10.1016/j.aim.2018.07.012}%
{\em Mellin transformation, propagation, and abelian duality spaces}, 
Adv.~Math. \textbf{335} (2108), 231--260.  
\MR{3836664}

\bibitem{LMW2}  Y.~Liu, L.~Maxim, B.~Wang, 
{\em Cohomology of $\Z$-local systems on complex hyperplane 
arrangement complements}, \arxiv{2209.13193v2}.

\bibitem{LX}  Y.~Liu, W.~Xie, 
{\em The homology groups of finite cyclic covering of line arrangement 
complement}, \arxiv{2312.10939v1}. 

\bibitem{McP}  A.~M\u{a}cinic, S.~Papadima,  
\href{https://doi.org/10.1016/j.topol.2008.09.014}%
{\em On the monodromy action on {M}ilnor fibers 
of graphic arrangements}, Topology Appl. 
\textbf{156} (2009), no.~4, 761--774. 
\MR{2492960} 

\bibitem{MS00}  D.~Matei, A.~Suciu,
\href{http://dx.doi.org/10.2969/ASPM/02710185}%
{\em Cohomology rings and nilpotent quotients of real and
complex arrangements}, in: Arrange\-ments--Tokyo 1998, 
185--215, Adv. Stud. Pure Math., vol.~27, 
Math. Soc. Japan, Tokyo, 2000.
\MR{1796900}

\bibitem{MS02}  D.~Matei, A.~Suciu,
\href{http://dx.doi.org/10.1155/S107379280210907X}%
{\em Hall invariants, homology of subgroups, and characteristic 
varieties}, Internat. Math. Res. Notices \textbf{2002} (2002), 
no.~9, 465--503.
\MR{1884468} 

\bibitem{Mi} J.~Milnor, 
\href{https://press.princeton.edu/books/paperback/9780691080659/singular-points-of-complex-hypersurfaces-am-61-volume-61}%
{\em Singular points of complex hypersurfaces},
Annals of Math. Studies, vol.~61, Princeton Univ. 
Press, Princeton, NJ, 1968.
\MR{0239612} 

\bibitem{OkaS}  M.~Oka, K.~Sakamoto, 
\href{https://doi.org/10.2969/jmsj/03040599}%
{\em Product theorem of the fundamental group of a reducible curve},
J. Math. Soc. Japan \textbf{30} (1978), no.~4, 599--602.
\MR{0513072}

\bibitem{OR} P.~Orlik, R.~Randell, 
\href{https://doi.org/10.1007/BF02559499}%
{\em The {M}ilnor fiber of a generic arrangement},  
Arkiv f\"ur Mat. \textbf{31} (1993), no.~1, 71--81. 
\MR{1230266}

\bibitem{OS} P.~Orlik, L.~Solomon,
\href{https://doi.org/10.1007/BF01392549}%
{\em Combinatorics and topology of complements of
hyperplanes}, Invent. Math. \textbf{56} (1980), no.~2, 167--189. 
\MR{0558866} 

\bibitem{OT} P.~Orlik, H.~Terao, 
\href{https://doi.org/10.1007/978-3-662-02772-1}%
{\em Arrangements of hyperplanes}, Grundlehren Math. Wiss., 
vol.~300, Springer-Verlag, Berlin, 1992; 
\MR{1217488}

\bibitem{PS-imrn04} S.~Papadima, A.I.~Suciu, 
\href{http://dx.doi.org/10.1155/S1073792804132017}%
{\em Chen {L}ie algebras}, Int. Math. Res. Not. 
\textbf{2004} (2004), no.~21, 1057--1086. 
\MR{2037049}

\bibitem{PS-cmh06} S.~Papadima and A.I. Suciu,
\href{https://dx.doi.org/10.4171/CMH/77}%
{\em When does the associated graded Lie algebra of an arrangement
group decompose?}, Comment. Math. Helv. \textbf{81} (2006), 859--875.
\MR{2271225}

\bibitem{PS-jlms07} S.~Papadima, A.I.~Suciu,
\href{https://doi.org/10.1112/jlms/jdm045}%
{\em Algebraic invariants for Bestvina-Brady groups},   
J. London Math. Soc. \textbf{76} (2007), no.~2, 273--292. 
\MR{2363416}

\bibitem{PS-adv09} S.~Papadima, A.I.~Suciu,
\href{https://dx.doi.org/10.1016/j.aim.2008.09.008}%
{\em Toric complexes and Artin kernels}, 
Adv. Math. \textbf{220} (2009), no.~2, 441--477. 
\MR{2466422}

\bibitem{PS-plms10} S.~Papadima, A.I.~Suciu,
\href{https://dx.doi.org/10.1112/plms/pdp045}%
{\em Bieri--{N}eumann--{S}trebel--{R}enz invariants and 
homology jumping loci}, Proc. Lond. Math.~Soc. 
\textbf{100} (2010), no.~3, 795--834.
\MR{2640291}  

\bibitem{PS-tams} S.~Papadima, A.I.~Suciu,
\href{https://dx.doi.org/10.1090/S0002-9947-09-05041-7}%
{\em The spectral sequence of an equivariant chain 
complex and homology with local coefficients}, 
Trans.~Amer.~Math.~Soc. \textbf{362} (2010), 
no.~5, 2685--2721.
\MR{2584616}

\bibitem{PS-plms17} S.~Papadima, A.I.~Suciu,
\href{https://dx.doi.org/10.1112/plms.12027}%
{\em The Milnor fibration of a hyperplane arrangement: from 
modular resonance to algebraic monodromy}, Proc. London 
Math. Soc.  \textbf{114} (2017), no.~6, 961--1004.
\MR{3661343}

\bibitem{PrS20} R.D.~Porter, A.I.~Suciu,
\href{https://doi.org/10.1007/s40879-019-00392-x}%
{\em Homology, lower central series, and hyperplane arrangements}, 
Eur. J. Math. \textbf{6} (2020), nr.~3, 1039--1072.
\MR{4151728}

\bibitem{Ryb} G.~Rybnikov, 
\href{https://doi.org/10.1007/s10688-011-0015-8}%
{\em On the fundamental group of the complement of a complex
hyperplane arrangement},  Funct. Anal. Appl. \textbf{45} (2011), 
no.~2, 137--148. 
\MR{2848779}

\bibitem{Sa95} M.~Sakuma,
\href{https://doi.org/10.4153/CJM-1995-010-2}%
{\em Homology of abelian coverings of links and spatial graphs},
Canad. J. Math. \textbf{47} (1995), no.~1, 201--224.
\MR{1319696}

\bibitem{SS17} M.~Salvetti, M.~Serventi,  
\href{https://doi.org/10.2422/2036-2145.201508_004}%
{\em Twisted cohomology of arrangements of lines and Milnor fibers},
Ann. Sc. Norm. Super. Pisa Cl. Sci. (5) \textbf{17} (2017), no.~4, 1461--1489. 
\MR{3752534}

 \bibitem{Sg23} S.~Sugawara, 
 \href{https://doi.org/10.1142/S0129167X23500441}%
 {\em $\Z$-local system cohomology of hyperplane arrangements and a 
 Cohen--Dimca--Orlik type theorem}, Intern. J. Math. \textbf{34} (2023), no.~8, 
 Paper No. 2350044, 15 pp.
 \MR{4620277}

\bibitem{Su01} A.I.~Suciu,
\href{https://dx.doi.org/10.1090/conm/276/04510}%
{\em Fundamental groups of line arrangements: 
Enumerative aspects}, in: Advances in algebraic geometry 
motivated by physics (Lowell, MA, 2000), pp. 43--79, Contemp. 
Math., vol 276, Amer. Math. Soc., Providence, RI, 2001.
\MR{1837109} 

\bibitem{Su02} A.I.~Suciu,
\href{https://dx.doi.org/10.1016/S0166-8641(01)00052-9}%
{\em Translated tori in the characteristic varieties of
complex hyperplane arrangements}, Topology Appl. 
\textbf{118} (2002), no.~1-2, 209--223.  
\MR{1877726} 

\bibitem{Su-conm11} A.I.~Suciu,
\href{https://dx.doi.org/10.1090/conm/538/10600}%
{\em Fundamental groups, Alexander invariants, and 
cohomology jumping loci},  in: Topology of 
algebraic varieties and singularities, 179--223, Contemp. 
Math., vol. 538, Amer. Math. Soc., Providence, RI, 2011. 
\MR{2777821} 

\bibitem{Su-imrn} A.I.~Suciu,
\href{https://dx.doi.org/10.1093/imrn/rns246}%
{\em Characteristic varieties and Betti numbers of free
abelian covers}, Int. Math. Res. Notices \textbf{2014} 
(2014), no. 4, 1063--1124.
\MR{3168402}

\bibitem{Su-toul} A.I.~Suciu,
\href{https://dx.doi.org/10.5802/afst.1412}%
{\em Hyperplane arrangements and {M}ilnor fibrations}, 
Ann. Fac. Sci. Toulouse Math. \textbf{23} (2014), no.~2, 417--481.  
\MR{3205599}

\bibitem{Su-revroum} A.I.~Suciu,
\href{http://imar.ro/journals/Revue_Mathematique/pdfs/2017/1/10.pdf}%
{\em On the topology of Milnor fibrations of hyperplane arrangements},  
 Rev. Roumaine Math. Pures Appl. \textbf{62} (2017), no.~1, 191--215.
\MR{3626439}

\bibitem{Su-bock} A.I.~Suciu,
\href{https://doi.org/10.1090/conm/790/15861}%
{\em Cohomology, Bocksteins, and resonance varieties in characteristic $2$}, 
131--157, Contemp. Math., vol.~790, Amer. Math. Soc., Providence, RI, 2023. 
\MR{4650243}

\bibitem{Su-formal} A.I.~Suciu,
\href{http://doi.org/10.4171/EMSS/74}%
{\em Formality and finiteness in rational homotopy theory}, 
Surv. Math. Sci. \textbf{10} (2023), no.~2, 321--403.
\MR{4667423}

\bibitem{Su-lcs-mono} A.I.~Suciu,
{\em Lower central series and split extensions},
\arxiv{2105.14129v2}.

\bibitem{Su-abexact} A.I.~Suciu,
\href{https://doi.org/10.2422/2036-2145.202112_005}%
{\em  Alexander invariants and cohomology jump loci in 
group extensions}, Ann. Sc. Norm. Super. Pisa Cl. Sci. (to appear),
\arxiv{2107.05148v2}.

\bibitem{Su-decomp} A.I.~Suciu,
{\em On the topology and combinatorics of decomposable arrangements}, 
\arxiv{2404.04784v1}.

\bibitem{SW-forum} A.I.~Suciu, H.~Wang, 
\href{https://doi.org/10.1515/forum-2018-0098}%
{\em Formality properties of finitely generated groups and Lie 
algebras}, Forum Math. \textbf{31} (2019), no.~4, 867--905. 
\MR{3975666}

\bibitem{Sullivan77} D.~Sullivan, 
\href{https://dx.doi.org/10.1007/BF02684341}%
{\em Infinitesimal computations in topology}, Inst. Hautes \'Etudes 
Sci. Publ. Math. (1977), no.~47, 269--331. 
\MR{0646078}

\bibitem{Ve} F.~Venturelli, 
{\em Gysin morphisms for non-transversal hyperplane sections with 
an application to line arrangements}, \arxiv{1912.11681v1}.

\bibitem{Wi}  K.~Williams, 
\href{https://dx.doi.org/10.2140/agt.2011.11.587}%
{\em Line arrangements and direct products of free groups}, 
Algebr. Geom. Topol. \textbf{11} (2011), 587--604. 
\MR{2783239}

\bibitem{WS} K.T.~Wong, Y.~Su, 
{\em On Betti numbers of Milnor fiber of hyperplane arrangements}, 
\arxx{1510}{03770v2}.

\bibitem{Yo20}  M.~Yoshinaga, 
\href{https://doi.org/10.1007/s40879-019-00387-8}%
{\em Double coverings of arrangement complements and $2$-torsion 
in Milnor fiber homology},  Eur. J. Math \textbf{6} (2020), nr.~3, 1097--1109.
\MR{4151730}

\bibitem{Zu}  H.~Zuber,
\href{https://dx.doi.org/10.1112/blms/bdq046}%
{\em Non-formality of Milnor fibers of line arrangements}, 
Bull. London Math. Soc. \textbf{42} (2010), no.~5, 905--911. 
\MR{2728693} 

\end{thebibliography}
\end{document}